\documentclass[12pt]{amsart}
\usepackage{amscd,amsmath,amssymb,amsfonts} 
\usepackage[cmtip, all]{xy}

\newtheorem{thm}{Theorem}[section]
\newtheorem{prop}[thm]{Proposition}
\newtheorem{lem}[thm]{Lemma}
\newtheorem{cor}[thm]{Corollary}

\theoremstyle{definition}
\newtheorem{exm}[thm]{Example}
\newtheorem{defn}[thm]{Definition}

\theoremstyle{remark}
\newtheorem{remk}[thm]{Remark}
\newtheorem{remks}[thm]{Remarks}
\newtheorem{exms}[thm]{Examples}
\newtheorem{notat}[thm]{Notation}
\newtheorem{ack}{Acknowledgements}

\numberwithin{equation}{section}

{\hfill$\square$\end{defn}}

{\hfill$\square$\end{remk}}

{\hfill$\square$\end{remks}}

{\hfill$\square$\end{exm}}

{\hfill$\square$\end{exms}}

{\hfill$\square$\end{notat}}

\newcommand{\sC}{{\mathcal C}}

\newcommand{\sH}{{\mathcal H}}

\newcommand{\sK}{{\mathcal K}}

\newcommand{\sO}{{\mathcal O}}

\newcommand{\sU}{{\mathcal U}}
\newcommand{\sV}{{\mathcal V}}

\newcommand{\sZ}{{\mathcal Z}}
\newcommand{\wt}{\widetilde}
\newcommand{\wh}{\widehat}

\newcommand{\A}{{\mathbb A}}

\renewcommand{\H}{{\mathbb H}}

\newcommand{\N}{{\mathbb N}}
\renewcommand{\P}{{\mathbb P}}
\newcommand{\Q}{{\mathbb Q}}

\newcommand{\Z}{{\mathbb Z}}

\newcommand{\surj}{\twoheadrightarrow}
\newcommand{\inj}{\hookrightarrow}

\newcommand{\ds}{{/\kern-3pt/}}

\newcommand{\ov}{\overline}

\begin{document}
\title{Riemann-Roch for Equivariant $K$-theory}
\author{Amalendu Krishna}


\baselineskip=10pt 
                                              
\begin{abstract}
The goal of this paper is to prove the equivariant version of Bloch's
Riemann-Roch isomorphism between the higher algebraic $K$-theory and the
higher Chow groups of smooth varieties. We show that for a 
linear algebraic group $G$ acting on a smooth variety $X$, although there is
no Chern character map from the equivariant $K$-groups to equivariant higher 
Chow groups, there is indeed such a map ${K^G_i(X)} {\otimes}_{R(G)} 
\widehat{R(G)} \xrightarrow {ch} {CH^*_G(X,i)} {\otimes}_{S(G)} 
\widehat{S(G)}$ with rational coefficients, which is an isomorphism. 
This implies the Riemann-Roch isomorphism $\widehat{K^G_i(X)} 
\xrightarrow{{\widehat{\tau}}^G_X} \widehat{CH^*_G\left(X, i \right)}$.
The case $i = 0$ provides a stronger form of the Riemann-Roch theorem of 
Edidin and Graham ({\sl cf.} \cite{ED1}). 
\end{abstract}

\maketitle   

\section{Introduction}
A {\sl variety} in this paper will mean a reduced, connected and
separated scheme of finite type with an ample line bundle over a field $k$. 
This base field $k$ will be fixed throughout this paper.
Let $G$ be a linear algebraic group over $k$ acting on such a variety 
$X$. Recall that this action on $X$ is said to be {\sl linear} if $X$ admits a 
$G$-equivariant ample line bundle, a condition which is always satisfied
if $X$ is normal ({\sl cf.} \cite[Theorem~2.5]{Sumihiro} for $G$ connected
and \cite[5.7]{Thomason1} for $G$ general). All $G$-actions in this paper
will be assumed to be linear. 
For $i \ge 0$, let $K^G_i(X)$ (resp. $G^G_i(X)$) denote the $i$th 
homotopy group of the $K$-theory spectrum of $G$-equivariant vector 
bundles (resp. coherent sheaves) on $X$. The $G$-equivariant higher Chow
groups $CH^*_G(X, i)$ of $X$ were defined by Edidin and Graham 
({\sl cf.} \cite{ED2}, also see below for more detail) as the ordinary higher 
Chow groups of the quotient space $X \stackrel{G}{\times} U$, where
$U$ is an open subscheme of a representation of $G$ on which the action of
$G$ is free, and its complement is of sufficiently high codimension. 

It has been known for a long time that in the non-equivariant case, 
there is a Riemann-Roch isomorphism ${G_0(X)}_{\Q} \to 
{CH^*\left(X, 0\right)}_{\Q}$ ({\sl cf.} \cite{Fulton}).
It is an important question to know if there are functorial Chern character 
and Riemann-Roch maps from $K$-theory spaces to a given cohomology theory, and
if these maps are isomorphisms with rational coefficients. It was proved
by Bloch ({\sl cf.} \cite[Theorem~9.1]{Bloch}) that for a quasi-projective
variety $X$, there are Riemann-Roch maps $G_i(X) \to CH^*(X,i)$ which
are isomorphisms with rational coefficients. It then follows that if
$X$ is smooth, the Chern character maps ${K_i(X)}_{\Q} \to {CH^*(X,i)}_{\Q}$ 
defined by Gillet in \cite{Gillet} are also isomorphisms.
  
In this paper, we address the question of extending this result to the
equivariant setting. It is not difficult to
see however that unlike in the non-equivariant  case, one can not expect such
an isomorphism between the equivariant $K$-groups and higher Chow groups
without perturbing them in some way. For example, for a finite cyclic
group $G$ of order $m$, one knows that ${K^G_0(k)}_{\Q}$ is a 
$\Q$-vector space of rank $m$, while ${CH^*_G(X, 0)}_{\Q}$ is a
$\Q$-vector space of rank $1$. This makes the equivariant Riemann-Roch     
problem more subtle than the ordinary one. 

For finite group actions, the structure of $G^G(X)$ was analysed in
\cite{Vistoli1}. The case that $G$ acts on a smooth variety $X$ with finite
stabilizers was treated in \cite{Toen} and \cite{VV0}. In this case,
one shows that ${G^G(X)}_{\Q}$ splits as a product from which one can
deduce the equivariant Riemann-Roch for $CH^*_G\left(X, 0\right)$.

The situation is much more subtle when the stabilizers are not finite.
Edidin and Graham proved the following fact. Denote by $R(G)$ the ring
of representations of $G$, tensored with $\Q$; call $I_G \subset R(G)$
the ideal of virtual representations of rank zero.
\begin{thm}[Edidin-Graham, \cite{ED1}]\label{thm:ED}
Let $X$ be a separated algebraic space, and let $\widehat{G^G_0(X)}$
denote the $I_G$-adic completion of the $R(G)$-module 
${G^G_0(X)}_{\Q}$. Then there is a Riemann-Roch map 
\[
{\widehat{\tau}}^G_X : {\widehat{G^G_0(X)}} \to 
\stackrel{\infty}{\underset {j=0}{\prod}}{CH^j_G(X,0)}_{\Q},\]
which is an isomorphism.
\end{thm}
Now, one can ask, first, if the completions of Edidin-Graham 
at the $G_0$-level is the minimal perturbation one requires to define and
prove Riemann-Roch isomorphisms in the equivariant setting, and second,
if such isomorphisms could be established also for the higher equivariant
$K$-groups. Our goal in this paper is to answer these two questions, in
the particular case that $X$ is smooth. We show that the difference between the
equivariant and the non-equivariant cases occurs already at the level of
$K_0$ of the base field: once this difference is taken 
into account, one can define and prove the Riemann-Roch and Chern character
isomorphisms for equivariant higher $K$-theory just like Bloch's
theorem in the non-equivariant case.  
 
We set up some notations before we state our main result. For a linear
algebraic group $G$ over $k$ as above, let $R(G)$ denote the ring of
virtual representations of $G$ and let $I_G$ be the ideal of the 
rank zero virtual representations, i.e., $I_G$ is the kernel of the
rank map $R(G) \to \Z$. It is easy to see that $R(G)$ is same as the
Grothendieck group $K^G_0(k)$. Let $\widehat{R(G)}$ denote the completion
of the ring $R(G)$ with respect to the ideal $I_G$. 
Let $S(G)$ denote the equivariant Chow ring 
$CH^*_G(k,0) = {\bigoplus}_{j \ge 0} CH^j_G(k,0)$
and let $J_G$ denote the irrelevant ideal ${\bigoplus}_{j \ge 1} CH^j_G(k,0)$.
Let $\widehat{S(G)}$ denote the $J_G$-adic completion of $S(G)$.
We shall call $I_G$ and $J_G$ as the augmentation ideals of the rings
$R(G)$ and $S(G)$ respectively, conforming to the notations already in
use in the literature. Let $\widehat{I_G}$ (resp. $\widehat{J_G}$) 
denote the extension of $I_G$ (resp. $J_G$) in the completion
$\widehat{R(G)}$ (resp. $\widehat{S(G)}$). 

For a variety $X$ with a $G$-action, put
\[CH^*_G(X,i) = {\underset {j \ge 0}{\bigoplus}} CH^j_G(X,i) \ \ {\rm for} \ \
i \ge 0.\]
It is known ({\sl cf.} \cite{ED2}) that the term on the right is an
infinite sum in general. 
For any such variety $X$, $G^G_i(X)$ and $K^G_i(X)$ are $K^G_0(X)$-modules
and hence the ring homomorphism $R(G) \to K^G_0(X)$ makes $G^G_i(X)$ and
$K^G_i(X)$ $R(G)$-modules. Moreover, for any $G$-equivariant map
$f: X \to Y$, the pull-back map $f^*: K^G_i(Y) \to K^G_i(X)$ is 
$R(G)$-linear, and the projection formula ({\sl cf.} 
\cite[Ex. II.8.3]{Hart}) implies that so is the
push-forward map if $f$ is proper.
At the level of equivariant higher Chow groups, the  
smoothness of the classifying stack $BG$ implies that there is an action of 
$CH^*_G(k,0)$ on 
$CH^*_G(X,i)$ ({\sl cf.} \cite[Proposition~5.5]{Bloch}) and this makes 
the latter an $S(G)$-module. As in the case of $K$-groups, any 
$G$-equivariant map $f:X \to Y$ induces an $S(G)$-linear pull-back map on the 
equivariant higher Chow groups if $Y$ is smooth, and an $S(G)$-linear
push-forward map if $f$ is proper ({\sl cf.} \cite[5.8]{Bloch}). 
For a $G$-variety $X$, let $\widehat{G^G_i(X)}$ denote the $I_G$-adic
completion of the $R(G)$-module $G^G_i(X)$, and let $\widehat{CH^*_G(X,i)}$
denote the $J_G$-adic completion of the $S(G)$-module $CH^*_G(X,i)$.
Finally, we recall that if $X$ is smooth, then there is a poincar{\'e}
duality isomorphism of spectra $K^G(X) \xrightarrow{\cong} G^G(X)$
({\sl cf.} \cite[Theorem~1.8]{Thomason1}). All the $K$-theory and Chow groups
in this paper (except in Section~2) will be tensored with $\Q$. 
We now state our main result.
\begin{thm}\label{thm:main*}
Let $X$ be a smooth variety with a $G$-action. Then for any $i \ge 0$,
there are Chern character maps
\begin{equation}\label{eqn:main1}
{\wt{ch}}^G_X : K^G_i(X){\otimes}_{R(G)} {\widehat{R(G)}}
\longrightarrow     
CH^*_G(X,i) {\otimes}_{S(G)} {\widehat{S(G)}}
\end{equation}
\begin{equation}\label{eqn:main2} 
{\widehat{ch}}^G_X : {\widehat{K^G_i(X)}} \longrightarrow 
{\widehat{CH^*_G(X,i)}}
\end{equation}
and a commutative diagram
\begin{equation}\label{eqn:main**}
\xymatrix{
{K^G_i(X) {\otimes}_{R(G)} {\widehat{R(G)}}}
\ar[r]^{{\wt{ch}}^G_X} \ar[d]_{u^G_X} &
{CH^*_G(X,i) {\otimes}_{S(G)} {\widehat{S(G)}}}
\ar[d]^{{\ov{u}}^G_X} \\
{\widehat{K^G_i(X)}} \ar[r]_{{\widehat{ch}}^G_X} &
{\widehat{CH^*_G(X,i)}}}
\end{equation}
such that the horizontal maps are isomorphisms and the vertical maps
are injective. Moreover, these Chern character maps commute with
the pull-back maps on $K$-groups and higher Chow groups of smooth
$G$-varieties, and with products.
\end{thm}
\begin{cor}\label{cor:mainRR}
Let $X$ be a smooth variety with a $G$-action. Then for every $i \ge 0$,
there is a Riemann-Roch isomorphism 
\[
{\widehat{{\tau}}}^G_X : {\widehat{K^G_i(X)}} 
\xrightarrow{\cong} {\widehat{CH^*_G(X,i)}}.
\]
\end{cor}
This result is the correct generalization of the Riemann-Roch theorem
of Edidin and Graham \cite{ED1} to the higher $K$-theory.
In fact for $i = 0$, one knows  ({\sl cf.} \cite[Proposition~2.1]{Graham},
\cite[Corollary~2.3]{Brion}) that the $J_G$-adic 
and the graded filtrations induce the same topology on 
$CH^*_G\left(X, 0 \right)$.
Hence the natural map ${\widehat{CH^*_G(X,i)}} \to
\stackrel{\infty}{\underset {j=0}{\prod}}{CH^j_G(X,0)}$ is an isomorphism.
In particular, the main result of Edidin and Graham \cite{ED1} is a
special case of Corollary~\ref{cor:mainRR}. 

For actions with finite stabilizers, the above results can be
further refined to give the following stronger form which is strikingly 
similar to the Bloch's non-equivariant Riemann-Roch theorem.
For $i=0$, the Riemann-Roch in this form  was conjectured by Vistoli 
({\sl cf.} \cite{Vistoli}) and proved (without the surjectivity assertion) 
by Edidin and Graham ({\sl cf.} \cite[Corollary~5.2]{ED1}). 
\begin{thm}\label{thm:finite} 
Let $G$ act on a (possibly singular) variety $X$ with finite stabilizers. 
Assume that $X$ is either smooth or, the stack $[X/G]$ has a
coarse moduli scheme. Then for $i \ge 0$,
the Riemann-Roch map of Theorem~\ref{thm:singular} induces a map
\[
{G^G_i(X)} \xrightarrow{{\tau}^G_X} {CH^*_G(X,i)},
\]
which is surjective, and $\alpha \in {\rm Ker}\left({{\tau}^G_X}\right)$
if and only if there exists a virtual representation $\epsilon \in
R(G)$ of non-zero rank such that $\epsilon \alpha = 0$.
\end{thm}

We make a few remarks on the above results. First of all, our results  
seem to be best possible form of equivariant
Riemann-Roch Theorem one could possibly hope for. This is because it
is unavoidable to tensor the left and the right sides of ~\ref{eqn:main1} 
with $\widehat{R(G)}$ and $\widehat{S(G)}$ respectively, as can be seen
even at the level of finite group actions.

The ring ${R(G)}_{\Q}$ has in general
infinitely many maximal ideals, besides the augmentation ideal $I_G$.
For any maximal ideal $\mathfrak m$, let ${\widehat{R(G)}}_{\mathfrak m}$
denote the $\mathfrak m$-adic completion of ${R(G)}_{\Q}$, and let
${\wt{G^G_i(X)}}_{\mathfrak m}$ denote the tensor product
$G^G_i(X){\otimes}_{R(G)} {\widehat{R(G)}}_{\mathfrak m}$.
Then the results of this paper give a description of the group
${\wt{G^G_i(X)}}_{I_G}$ in terms of equivariant higher Chow groups. 
One would like to give a similar description of 
${\wt{G^G_i(X)}}_{\mathfrak m}$, for any given maximal ideal 
$\mathfrak m$. This will be the subject of a forthcoming sequel
to this paper. 

We end this section with a brief description of the contents of the
various sections of this paper. We review the definitions and various 
properties of the equivariant higher Chow groups in the next section.
The main result here is a self intersection formula for the
higher Chow groups. This formula is crucial for the decomposition
theorem of equivariant higher Chow groups of varieties with an action
of a diagonalizable group. In Section~3, we prove various reduction
techniques for describing the equivariant higher Chow groups for
action of arbitrary groups in terms of the higher Chow groups for action of
tori. We also prove a decomposition theorem for the equivariant 
higher Chow groups of a $G$-variety $X$ when certain subgroup of $G$ acts
trivially on $X$. Section~4 is devoted to the description
and comparison of various completions of the $S(G)$-modules
$CH^*_G\left(X, \cdot \right)$ for a $G$-variety $X$. In Section~5,
we study the notion of cohomological rigidity, which is then used to
construct various specialization maps between the equivariant higher Chow
groups, analogous to the specialization maps in equivariant $K$-theory
in \cite{VV}. In the following section, we prove our main decomposition
theorem (Theorem~\ref{thm:decomposition*}) for the equivariant higher Chow 
groups for action of diagonalizable groups. This is one of the crucial
steps in proving Theorem~\ref{thm:main*} for diagonalizable groups.
In Section~7, we construct our main objects of study,
the equivariant Chern character and Riemann-Roch maps. We also prove various
properties of these maps which are useful in proving our main results.
Our approach to the proof of our main result is to eventually reduce it to the
case when the underlying group acts either with finite stabilizers
or with a constant dimension of stabilizers.   
Sections~8 and 9 are devoted to proving our results in these cases
for the action of diagonalizable groups. In Section~10, we prove
our main result for actions of diagonalizable groups. Section~11
is devoted to proving some results for relating the equivariant 
$K$-groups and higher Chow groups for actions of any linear algebraic groups 
with these groups for actions of diagonalizable groups. We finally prove
the above stated results in the last section of the paper. We also give
an application of the above Riemann-Roch theorems to the equivariant
$K$-theory.
      
\section{Equivariant Higher Chow Groups}
Unlike the case of the rest of this paper, all the groups in this section
will be considered with integral coefficients. 
We begin with a brief review of the equivariant higher
Chow groups from \cite{ED2} and their main properties, especially those
which will be used repeatedly in this paper. Our main result in this
section is the self intersection formula for the higher Chow
groups of smooth varieties. The analogous formula for the higher 
$K$-theory was proved by Thomason ({\sl cf.} \cite[Theorem~3.1]{Thomason2}). 
Surprisingly, this formula for the higher Chow groups has still been unknown. 
The equivariant version of this formula will be crucial in proving our
decomposition Theorem~\ref{thm:decomposition*} for the equivariant higher 
Chow groups of smooth varieties with an action of a diagonalizable group. 

Let $G$ be a linear algebraic group and let $X$ be an equidimensional
variety over $k$ with a $G$-action. All representations of $G$ in this paper
will be finite dimensional. The definition of equivariant higher Chow groups
of $X$ needs one to consider certain kind of mixed spaces which in general
may not be a scheme even if the original space is a scheme. The following
well known ({\sl cf.} \cite[Proposition~23]{ED2}) lemma shows that this
problem does not occur in our context and all the mixed spaces in this
paper are schemes with ample line bundles.
\begin{lem}\label{lem:sch}
Let $H$ be a linear algebraic group acting freely and linearly on a 
$k$-variety $U$ such that the quotient $U/H$ exists as a quasi-projective
variety. Let $X$ be a $k$-variety with a linear action of $H$.
Then the mixed quotient $X \stackrel{H} {\times} U$ exists for the 
diagonal action of $H$ on $X \times U$ and is quasi-projective.
Moreover, this quotient is smooth if both $U$ and $X$ are so.
In particular, if $H$ is a closed subgroup of a linear algebraic group $G$ 
and $X$ is a $k$-variety with a linear action of $H$, then the quotient 
$G \stackrel{H} {\times} X$ is a quasi-projective scheme.
\end{lem}
\begin{proof} It is already shown in \cite[Proposition~23]{ED2} using
\cite[Proposition~7.1]{GIT} that the quotient $X \stackrel{H} {\times} U$ 
is a scheme. Moreover, as $U/H$ is quasi-projective, 
\cite[Proposition~7.1]{GIT} in fact shows that $X \stackrel{H} {\times} U$ 
is also quasi-projective. The similar conclusion about 
$G \stackrel{H} {\times} X$ follows from the first case by taking $U = G$
and by observing that $G/H$ is a smooth quasi-projective scheme
({\sl cf.} \cite[Theorem~6.8]{Borel}). The assertion about the smoothness
is clear since $X \times U \to X \stackrel{H} {\times} U$ is a principal
$H$-bundle.
\end{proof}   

For any integer $j \ge 0$, let $V$ be a representation of $G$ and let $U$ be 
a $G$-invariant open subset of $V$ such that the codimension of the 
complement $V-U$ in $V$ is larger than $j$, and $G$ acts freely on $U$ such 
that the quotient $U/G$ is a quasi-projective scheme. Such a pair 
$\left(V,U\right)$ will be called a {\sl good} pair for the $G$-action
corresponding to $j$. 
It is easy to see that a good pair always exists ({\sl cf.}
\cite[Lemma~9]{ED2}). Let $X_G$ denote the quotient 
$X \stackrel{G} {\times} U$ of the product $X \times U$ by the 
diagonal action of $G$ which is free. We define the equivariant higher Chow 
group $CH^j_G(X, i)$ as the homology group $H_i\left(\sZ(X_G, \cdot)\right)$,
where $\sZ(X_G, \cdot)$ is the Bloch's cycle complex of the variety $X_G$.
It is known ({\sl loc. cit.}) that this definition of $CH^j_G(X, i)$
is independent of the choice of a good pair $(V,U)$ for the $G$-action. 
One should also
observe that $CH^j_G(X, i)$ may be non-zero   for infinitely many values of
$j$, a crucial change from the non-equivariant  higher Chow groups.
The following result summarizes most of the essential properties of 
the equivariant higher Chow groups that will be used in this paper. 
Let ${\sV}_G$ denote the category of $G$-varieties with $G$-equivariant maps
and let ${\sV}^S_G$ denote the full subcategory of smooth $G$-varieties.
\begin{prop}\label{prop:EHCG}
The equivariant higher Chow groups as defined above satisfy the following
properties. \\
$(i) \ \ Functoriality:$ Covariance for proper maps and contravariance for
flat maps. Moreover, if $f:X \to Y$ is a morphism in ${\sV}_G$ with
$Y$ in ${\sV}^S_G$, then there is a pull-back map $f^*: CH^*_G(Y,i)
\to CH^*_G(X,i)$. \\
$(ii) \ \ Homotopy:$ If $f:X \to Y$ is an equivariant vector bundle, then
$f^*: CH^*_G(Y,i) \xrightarrow{\cong} CH^*_G(X,i)$. \\
$(iii) \ \ Localization:$ If $Y \subset X$ is of pure codimension $d$ with
complement $U$, then there is a long exact localization sequence 
\[\cdots \to  CH^{*-d}_G(Y,i) \to CH^{*}_G(X,i) \to  CH^{*}_G(U,i) \to
CH^{*-d}_G(Y,i-1) \to \cdots.
\]
$(iv) \ \ Exterior \ product:$ There is a natural product map 
\[
CH^{j}_G(X,i) {\otimes}  CH^{j'}_G(Y,i') \to  CH^{j+j'}_G(X \times Y,i+i').
\]
Moreover, if $f: X \to Y$ is such that $Y \in {\sV}^S_G$, then there is
a pull-back via the graph map ${\Gamma}_f : X \to {X \times Y}$, which makes
$CH^*_G(Y, \cdot)$ a bigraded ring and $CH^*_G(X, \cdot)$ a module
over this ring. \\
$(v) \ \ Chern  \ classes:$ For any $G$-equivariant vector bundle of rank $r$,
there are equivariant Chern classes $c^G_l(E):  CH^{j}_G(X,i) \to
 CH^{j+l}_G(X,i)$ for $1 \le l \le r$, having the same functoriality 
properties as in the non-equivariant  case. \\
$(vi) \ \ Projection \ formula:$ For a proper map $f: X \to Y$ in
${\sV}_G$ with $Y \in {\sV}^S_G$, one has for $x \in CH^{*}_G(X,\cdot),
y \in CH^{*}_G(Y,\cdot)$, $f_*\left(x \cdot f^*(y)\right) =
f_*(x) \cdot y$. \\
$(vii) \ \ Free \ action:$ If $G$ acts freely on $X$ with quotient $Y$, then 
there is  a natural isomorphism $CH^*_G(X,i) \xrightarrow{\cong}
CH^*(Y,i)$.
\end{prop}  
\begin{proof}
Since the equivariant higher Chow groups of $X$ are defined in terms of
the the higher Chow groups of $X_G$, the proposition (except possibly the 
last property) can be easily deduced from the similar results for the 
non-equivariant  higher Chow groups as in \cite{Bloch} and the techniques 
of {\sl loc. cit.}. We therefore skip the proof. For the last property,
fix $j \ge 0$ and choose a good pair $\left(V,U\right)$ for the $G$-action
corresponding to $j$. Since $G$ acts freely on $X$, it acts likewise also
on $X \times V$ with quotient, say $X_V$. Then $X_G$ is an open subset
of $X_V$ and $X_V \to Y$ is a locally trivial $V$-fibration, which implies
that the map $CH^j\left(Y, i\right) \to CH^j\left(X_V, i\right)$ is an
isomorphism by the homotopy invariance. On the other hand, the pull-back
map $CH^j\left(X_V, i\right) \to CH^j\left(X_G, i\right) = 
CH^j_G\left(X, i\right)$ is an isomorphism by the property $(iii)$ as 
$j$ is sufficiently large.    
\end{proof}    
If $X$ is not equidimensional, then one defines the equivariant higher Chow
groups $CH^G_j(X,i)$ as $CH_{j+l-g}(X_G,i)$, where $X_G$ is formed from an
$l$-dimensional representation $V$ such that $V-U$ has sufficiently high
codimension and $g$ is the dimension of $G$. The groups $CH^G_j(X,i)$ enjoy 
many of the properties stated above and in particular, one has the
localization sequence as above even if the closed subscheme $Y \subset$
is not equidimensional ({\sl cf.} [{\sl loc. cit.}, Proposition~5]). 
It is easy to see that $CH^G_j(X,i) \cong CH^{d-j}_G(X,i)$ if $X$ is 
equidimensional of dimension $d$. 

We next recall that the Chern classes $c^G_l(E)$ of an equivariant 
vector bundle $E$, as described in Proposition~\ref{prop:EHCG} above,
live in the operational Chow groups $A^l(X_G)$. If $X$ is in ${\sV}^S_G$
however, this operational Chow group is isomorphic to
the equivariant Chow group $CH^l_G(X,0)$ and the action of $c^G_l(E)$ on 
$CH^*_G(X,\cdot)$ then coincides with the intersection product in the ring 
$CH^*_G(X, \cdot)$. 
 
Finally, we recall from {\sl loc. cit.} that if $H \subset G$ is a closed 
subgroup and if $(V,U)$ is a good pair, then for $X \in {\sV}_G$, 
the natural map of 
quotients $X \stackrel{H}{\times} U \to X \stackrel{G}{\times} U$ is a
$G/H$-principal bundle and hence there is a natural restriction map
\begin{equation}\label{eqn:res}
r^G_H : CH^*_G\left(X,\cdot\right) \to CH^*_H\left(X,\cdot\right).
\end{equation} 
Taking $H = \{1\}$, and using the homotopy invariance and the localization
sequence, one gets a natural map
\begin{equation}\label{eqn:forgetful}
r^G_X : CH^*_G\left(X,\cdot\right) \to CH^*\left(X, \cdot\right).
\end{equation}
Moreover, as $r^G_X$ is the pull-back under a flat map, it commutes
({\sl cf.} Proposition~\ref{prop:EHCG}) with the pull-back for any flat 
map, and with the push-forward for any proper map in ${\sV}_G$. 
We remark here that although the definition of $r^G_H$ uses a good pair
$(V,U)$ for any given $j \ge 0$, it is easy to check from the homotopy
invariance that $r^G_H$ is independent of the choice of the good pair
$(V,U)$.

As mentioned before, our main goal in this section is to prove a 
self-intersection formula for the ordinary and equivariant higher Chow 
groups. Our main technical tool to prove this is the {\sl deformation to 
the normal cone} method. Since this technique will be used several times 
in this paper, we briefly recall the construction from 
\cite[Chapter~5]{Fulton} for our as well as reader's convenience.
Let $X$ be a smooth variety over $k$ and let $Y \stackrel{f}{\inj} X$ be a
smooth closed subscheme of codimension $d \ge 1$. 
Let $\wt{M}$ be the blow-up of $X \times {\P}^1$ along $Y \times {\infty}$. 
Then $Bl_Y(X)$ is a closed subscheme of $\wt{M}$ and one denotes its
complement by $M$. There is a natural map ${\pi} : M \to {\P}^1$ such that
${\pi}^{-1}({\A}^1) \cong X \times {\A}^1$ with $\pi$ the projection map
and ${\pi}^{-1}(\infty) \cong X'$, where $X'$ is the total space of the
normal bundle $N_{Y/X}$ of $Y$ in $X$. One also gets the following 
diagram, where all the squares and the triangles commute.
\begin{equation}\label{eqn:DNC}
\xymatrix{
Y \ar[rr]_{i_0} \ar[rd]^{u'} \ar[dd]^{f} & & {Y \times {\P}^1} 
\ar@/_/[ll]_{p_Y}
\ar[dd]^{F} & Y \ar[l]^{i_{\infty}} \ar[dd]^{f'} \\
& {Y \times {\A}^1} \ar[dd]_{F'} \ar[ru]^{j'} & & \\
X \ar[rr]^{h} \ar[dr]^{u} & & M & X' \ar[l]_{i} \\
& {X \times {\A}^1} \ar[ru]^{j} & & }
\end{equation}   
In this diagram, all the vertical arrows are the closed embeddings, $i_0$
and $i_{\infty}$ are the obvious inclusions of $Y$ in $Y \times {\P}^1$ along
the specified points, $i$ and $j$ are inclusions of the inverse
images of $\infty$ and ${\A}^1$ respectively under the map $\pi$, 
$u$ and $f'$ are are zero section embeddings and $p_Y$ is the projection
map. In particular, one has ${p_Y} \circ {i_0} = {p_Y} \circ {i_{\infty}}
= {id}_Y$.

We also make the observation here that in case $X$ is a $G$-variety
and $Y$ is $G$-invariant, then by letting $G$ act trivially on ${\P}^1$
and diagonally on $X \times {\P}^1$, one gets a natural action of $G$ on
$M$, and all the spaces in the above diagram become $G$-spaces and all the
morphisms become $G$-equivariant. This observation will be used later on 
in this paper. 
 
We shall need the following result about the higher Chow groups which
is an easy consequence of Bloch's moving lemma.
\begin{lem}\label{lem:commute} 
Let 
\[
\xymatrix{
W \ar[r]^{i'} \ar[d]_{j'} & Y \ar[d]^{j} \\
Z \ar[r]_{i} & X}
\]
be a fiber diagram of closed immersions such that $X$ and $Y$ are smooth.
Then one has $i^*\circ j_* = {j'}_* \circ {i'}^* : CH^*(Y, \cdot) \to
CH^*(Z, \cdot)$.   
\end{lem}
\begin{proof} Since $X$ and $Y$ are smooth, we can assume them to be
equidimensional. Let $\sZ^p_{ZW}(Y, \cdot) \stackrel{i_Y} \inj 
\sZ^p(Y, \cdot)$ be the
subcomplex which is generated by cycles on $Y \times {\Delta}^{\cdot}$
which intersect all faces of $Z \times {\Delta}^{\cdot}$ and 
$W \times {\Delta}^{\cdot}$ properly. 
Similarly, let $\sZ^p_{Z}(X, \cdot) \stackrel{i_X} \inj 
\sZ^p(X, \cdot)$ be the subcomplex generated by cycles on 
$X \times {\Delta}^{\cdot}$ which intersect all faces of 
$Z \times {\Delta}^{\cdot}$ properly. Then $i_X$ and $i_Y$ are 
quasi-isomorphisms
by the moving lemma ({\sl cf.} \cite[Theorem~1.10]{KrishnaL}). However,
if $V \in \sZ^p_{ZW}(Y, \cdot)$ is an irreducible cycle in 
$Y \times {\Delta}^n$, then the conclusion of the lemma is checked easily.
\end{proof}
\begin{lem}\label{lem:DNC1}
Consider the diagram ~\ref{eqn:DNC} and let $y \in CH^*(Y,m)$. Then
there exists $z \in  CH^*(M,m)$ such that $f_*(y) = h^*(z)$ and
${f'}_*(y) = i^*(z)$.
\end{lem}
\begin{proof} Put $\tilde{y} = {p_Y^*}(y)$ and $z = F_*(\tilde{y})$.
Then
\[
\begin{array}{lllll}
f_*(y) & = & f_*\left({\left(p_Y \circ {i_0}\right)}^*(x)\right) & & \\
& = & f_* \circ {i_0^*} \circ {p_Y^*} \left(y\right) & = &
f_* \circ {i_0^*}\left(\tilde{y}\right) \\
& = & f_* \circ {u'}^* \circ {j'}^*\left(\tilde{y}\right) & &  \\
& = & u^* \circ {F'}_*\left({j'}^*\left(\tilde{y}\right)\right) & &
\hspace*{2cm} \left({\rm{by \ Lemma~\ref{lem:commute}}}\right) \\
& = & u^* \circ j^* \circ F_*\left(\tilde{y}\right) & &  
\hspace*{2cm}\left({\rm{since}} \ j \ 
{\rm{is \ an \ open \ immersion}}\right) \\ 
& = & h^* \circ F_*\left(\tilde{y}\right) & = & h^*(z). 
\end{array}
\]
Similarly, 
\[
\begin{array}{lllll}
{f'}_*(y) & = & {f'}_*\left({\left(p_Y \circ 
{i_{\infty}}\right)}^*(x)\right) & & \\
& = & {f'}_* \circ i_{\infty}^* \circ {p_Y^*} \left(y\right) & = &
{f'}_* \circ i_{\infty}^*\left(\tilde{y}\right) \\
& = & i^* \circ {F}_*\left(\tilde{y}\right) & &
\hspace*{2cm} \left({\rm{by \ Lemma~\ref{lem:commute}}}\right) \\
& = & i^*(z).
\end{array}
\]
\end{proof}
\begin{thm}[\bf{Self-intersection Formula}]\label{thm:SIF}
Let $Y \stackrel{f}\inj X$ be a closed immersion of smooth varieties of
codimension $d \ge 1$, and let $N_{Y/X}$ be the normal bundle of $Y$ in $X$.
Then one has for every $y \in CH^*(Y, \cdot)$, $f^*\circ f_*(y) =
c_d\left(N_{Y/X}\right)\cdot y$.
\end{thm}
\begin{proof} We first consider the case when $X \xrightarrow{p} Y$ is a 
vector bundle of rank $d$ and $f$ is the zero section embedding so that
$p \circ f = {id}_Y$. In that case, we have 
\[
\begin{array}{lllll}
f^* \circ f_*(y) & = & f^* \circ f_*\left(f^* \circ p^*(y)\right) & & \\
& = & f^*\left(f_*(1) \cdot p^*(y)\right) & 
\hspace*{2cm} \left({\rm{by \ Proposition~\ref{prop:EHCG} \ (vi)}}\right) \\
& = & f^*\left(f_*(1)\right) \cdot \left(f^* \circ p^*(y) \right) & & \\
& = & f^*\left(f_*(1)\right) \cdot y & & \\
& = & c_d\left(N_{Y/X}\right) \cdot y, & &
\end{array}
\]
where the last equality follows from the self-intersection formula for 
Fulton's Chow groups ({\sl cf.} \cite[Corollary~6.3]{Fulton}).
This proves the theorem in the case of zero section embedding.

Now let $Y \inj X$ be as in the theorem. We consider the deformation to the
normal cone diagram ~\ref{eqn:DNC}, and choose $z \in CH^*(M,\cdot)$ as
in Lemma~\ref{lem:DNC1}. Then we have 
\[
\begin{array}{lllll}
f^* \circ f_*(y) & = & f^* \circ h^*\left(z\right) & = & 
{i_0}^* \circ F^*(z) \\
& = & i_{\infty}^* \circ F^*(z) & = &
{f'}^* \circ i^*(z) \\
& = & c_d\left(N_{Y/{X'}}\right) \cdot y & & \left(
{\rm{by \ the \ case \ of \ vector \ bundle \ above}}\right) \\
& = & c_d\left(N_{Y/{X}}\right) \cdot y.
\end{array}
\]
This completes the proof of the theorem.
\end{proof}
\begin{cor}\label{cor:ESIF}
Let $G$ be a linear algebraic group over $k$ and let 
$Y \stackrel{f}\inj X$ be a closed immersion of codimension 
$d \ge 1$ in ${\sV}^S_G$. Then one has for every 
$y \in CH^*_G(Y, \cdot)$, $f^*\circ f_*(y) =
c^G_d\left(N_{Y/X}\right)\cdot y$.
\end{cor} 
\begin{proof} Fix $i, j \ge 0$ and choose and good pair $(V,U)$ for 
$n \gg j+d$. We can then identify $CH^p_G(X,i)$
with $CH^p(X_G,i)$ (and same for $Y$) for $p \le n$.
We can also identify $c^G_d(E)$ with $c_d(E_G)$ for any equivariant 
vector bundle $E$ on $Y$ ({\sl cf.} \cite[Section~2.4]{ED2}). 
Now the proof of the corollary would follow
straightaway from Theorem~\ref{thm:SIF}, once we show that
${\left(N_{Y/X}\right)}_G$ is the normal bundle of $Y_G$ in $X_G$.
But this follows immediately from the elementary fact that if
$G$ acts freely on a smooth variety $Z$ and $W$ is a smooth closed
and $G$-invariant subvariety of $Z$ with normal bundle $N$, then
$G$ acts freely on $N$, and moreover, $N/G$ is the normal bundle of
$W/G$ in $Z/G$. We leave the proof of this fact to the reader.
\end{proof}   

\section{Reduction techniques For Equivariant Higher Chow Groups}
One of the important tools in the equivariant geometry is the technique of
reducing the study of varieties with action of an arbitrary linear algebraic 
groups to 
the reductive and then to the diagonalizable groups. In order to successfully
apply this technique in practice, one often needs to know how certain
invariants of varieties with an action of a group $G$ are related to these
invariants for the actions of various subgroups or quotients of $G$.
In this section, we establish some basic results in this direction
about the equivariant higher Chow groups. We recall here our convention
that an abelian group $A$ in the rest of this paper will actually mean the 
group $A {\otimes}_{\Z} {\Q}$.

Although many of the results that follow hold also with the integral
coefficients, we shall not need them in that form.
\begin{prop}[Morita isomorphism]\label{prop:Morita*}
Let $H$ be a normal subgroup of a linear algebraic group $G$ and let
$F = G/H$. Let $f: X \to Y$ be a $G$-equivariant morphism of $G$-varieties
which is an $H$-torsor for the restricted action. Then for every $i, j \ge 0$,
the map $f$ induces an isomorphism of the equivariant higher Chow groups 
\[
CH^j_F\left(Y, i\right) \xrightarrow{f^*} CH^j_G\left(X, i\right).
\]
\end{prop}
\begin{proof} We first observe from \cite[Corollary~12.2.2]{Springer} that
$F$ is also a linear algebraic group over the given ground field $k$.
Now, since $f$ is an $H$-torsor, it is clear that $G$ acts on $Y$
via $F$. Fix $j \ge 0$ and choose a good pair $\left(V,U\right)$
for the $F$-action on $Y$ corresponding to $j$. Then $V$ is also a 
representation of $G$ in which $U$ is $G$-invariant. In particular,
$G$ acts on $X \times U$ via the diagonal action, which is easily
seen to be free since $H$ acts freely on $X$ and $F$ acts freely on $U$.
By the same reason, we see that $X\times U \to Y \times U$ is $G$-equivariant
which is a principal $H$-bundle. This in turn implis that the map
${\left(X \times U\right)}/G \to Y_F$ is an isomorphism and hence we get
\begin{equation}\label{eqn:M**}
CH^j_F\left(Y, i\right) \cong CH^j\left(Y_F, i\right) 
\xrightarrow{f^*} CH^j\left({\left(X \times U\right)}/G, i\right).
\end{equation}
On the other hand, we have 
\[
CH^j_G\left(X, i\right) \cong CH^j_G\left(X \times V, i\right)
\cong CH^j_G\left(X \times U, i\right) \cong 
CH^j\left({\left(X \times U\right)}/G, i\right),
\]
where the first isomorphism is due to the homotopy invariance, the second 
follows from the localization property 
({\sl cf.} Proposition~\ref{prop:EHCG}, $(iii)$) as $j$ is sufficiently large,
and the third isomorphism follows from Proposition~\ref{prop:EHCG},
$(vii)$. The proof of the proposition now follows by combining this with
~\ref{eqn:M**}.
\end{proof}   
\begin{cor}[{\sl cf.} \cite{ED1}]\label{cor:Morita}
Let $H \subset G$ be a closed subgroup and let $X \in {\sV}_H$.
Then for any $i, j \ge 0$, there is a natural isomorphism
\begin{equation}\label{eqn:MoritaI}
CH^j_G\left(G \stackrel{H}{\times} X, i\right) \xrightarrow{\cong}
CH^j_H\left(X, i\right).
\end{equation}
\end{cor}
\begin{proof} Define an action of $H \times G$ on $G \times X$ by
\[
(h, g) \cdot (g', x) = \left(gg'h^{-1}, hx\right),
\]
and an action of $H \times G$ on $X$ by $(h, g) \cdot x = hx$. 
Then the projection map $G \times X \xrightarrow{p} X$ is 
$\left(H \times G\right)$-equivariant which is a 
$G$-torsor. Hence by 
Proposition~\ref{prop:Morita*}, the natural map 
$CH^j_H\left(X, i\right)
\xrightarrow{p^*} CH^j_{H \times G}\left(G \times X, i\right)$. 
On the other hand, the projection map $G \times X 
\to G \stackrel{H}{\times} X$ is $\left(H \times G\right)$-equivariant which 
is an $H$-torsor. Hence we get an isomorphism 
$CH^j_G\left(G \stackrel{H}{\times} X , i \right) \xrightarrow{\cong}
CH^j_{H \times G}\left(G \times X, i\right)$. The corollary follows
by combining these two isomorphisms. 

\end{proof}
\begin{thm}\label{thm:Borel}
Let $G$ be a connected and reductive group over $k$. Let $B$ be a Borel 
subgroup of $G$ containing a maximal torus $T$ over $k$. Then the restriction 
maps
\begin{equation}\label{eqn:Borel2}
CH^*_B\left(X, \cdot\right) \xrightarrow{r^B_T} CH^*_T\left(X, \cdot\right),
\end{equation}
\begin{equation}\label{eqn:Borel3*}
CH^*_G\left(X, \cdot\right) \xrightarrow{r^G_T} CH^*_T\left(X, \cdot\right) 
\end{equation}
are respectively isomorphism and split monomorphism. Moreover, this
splitting is natural for morphisms in ${\sV}_G$. In particular, if $H$ is
any closed subgroup of $G$, then then there is a split injective map 
\begin{equation}\label{eqn:Borel3}
CH^*_H\left(X, \cdot\right) \xrightarrow{r^G_T} 
CH^*_T\left(G \stackrel{H}{\times} X, \cdot\right) 
\end{equation}
\end{thm}
\begin{proof} We first prove ~\ref{eqn:Borel2}.
By Corollary~\ref{cor:Morita}, we only need to show that
\begin{equation}\label{eqn:Borel1}
CH^*_B\left(B \stackrel{T}{\times} X, \cdot\right) \cong
CH^*_T\left(X, \cdot\right).
\end{equation}
By \cite[XXII, 5.9.5]{SGA3}, there exists a characteristic
filtration $B^u = U_0 \supseteq U_1 \supseteq \cdots \supseteq U_n =
\{1\}$ of the unipotent radical $B^u$ of $B$ such that ${U_{i-1}}/{U_i}$
is a vector group, each $U_i$ is normal in $B$ and $TU_i = T \ltimes U_i$. 
Moreover, this filtration also implies that for each $i$, the natural map
$B/{BU_i} \to B/{TU_{i-1}}$ is a torsor under the vector bundle
${U_{i-1}}/{U_i} \times B/{TU_{i-1}}$ on $B/{TU_{i-1}}$. Hence, the 
homotopy invariance gives an isomorphism
\[
CH^*_B\left(B/{TU_{i-1}} \times X, \cdot\right) \xrightarrow{\cong}
CH^*_B\left(B/{TU_i} \times X, \cdot\right).
\]
Composing these isomorphisms successively for $i = 1, \cdots ,n$, we get
\[
CH^*_B\left(X, \cdot\right) \xrightarrow{\cong}
CH^*_B\left(B/T \times X, \cdot\right).
\]
The isomorphism of $B$-varieties $B \stackrel{T}{\times} X \cong
B/T \times X$ ({\sl cf.} Corollary~\ref{cor:Morita}) 
now proves ~\ref{eqn:Borel1} and hence ~\ref{eqn:Borel2}.

To prove ~\ref{eqn:Borel3*}, we can apply ~\ref{eqn:Borel2} to reduce to
showing that the map
$CH^*_G\left(X, \cdot\right) \to CH^*_B\left(X, \cdot\right)$
is a naturally split monomorphism. By ~\ref{eqn:MoritaI},
there is an isomorphism $CH^*_B\left(X, \cdot\right) \cong
CH^*_G\left(G \stackrel{B}{\times} X, \cdot\right)$. Moreover, there is an
isomorphism of $G$-varieties $G \stackrel{B}{\times} X \cong G/B 
\times X$. Thus it suffices to show for the flat and proper
map $f : G/B \times X \to X$ that $f^*$ is split by the map $f_*$.
Using the projection formula of Proposition~\ref{prop:EHCG}, it suffices
to show that $f_*(1) = 1$. But this follows directly from
\cite[Theorem~2.1]{Kock}. Finally, ~\ref{eqn:Borel3} follows from
~\ref{eqn:Borel3*} and Corollary~\ref{cor:Morita}.
\end{proof}
\begin{prop}\label{prop:NRL}
Let $H$ be a possibly non-reductive group over $k$.
Assume that $H$ has a Levi decomposition $H = L \ltimes H^u$ such
that $H_u$ is split over $k$ (e.g., when $k$ is of characteristic zero).
Then the restriction map 
\begin{equation}\label{eqn:NRL0}
CH^*_H\left(X, \cdot\right) \xrightarrow{r^H_L} CH^*_L\left(X, \cdot\right),
\end{equation}
is an isomorphism.
\end{prop}
\begin{proof}
Since the unipotent radical of $H$ is split over $k$, the proof is exactly
same as in the proof of ~\ref{eqn:Borel2}, where we just have to replace
$B$ and $T$ by $H$ and $L$ respectively.
\end{proof}
\begin{remk} We point out here that though we have assumed all abelian
groups to be tensored with $\Q$ in this and the latter sections, the
readers can check from the proofs that the results of this section so
far, remain true with the integral coefficients.
\end{remk}

A consequence of Theorem~\ref{thm:Borel} is that the equivariant higher
Chow groups for the action of a connected reductive group $G$ are subgroups of
the equivariant higher Chow groups for the action of a maximal torus 
of $G$. In our next result, we prove a refinement of this by giving an 
explicit description of these subgroups. This is a generalization of
the analogous result Proposition~6 of \cite{ED2} to equivariant higher
Chow groups. We begin with following result about the non-equivariant  
higher Chow groups.
\begin{lem}\label{lem:linear}
If $L$ is a linear variety over $k$ and $X = Y \times L$, then the
exterior product map
\[
CH^*(Y, \cdot) {\otimes}_{CH^*(k, \cdot)} CH^*(L, \cdot)     
\to CH^*(X, \cdot)
\]
is an isomorphism of $CH^*(k, \cdot)$-modules.
\end{lem}
\begin{proof} We prove by the induction on the dimension of $L$. If $L$ is 
$0$-dimensional, there is nothing to prove. So we assume that 
${\rm dim}(L) \ge 1$ and that the lemma holds for all linear varieties
of dimension less than the dimension of $L$. 
Put $CH^*(k) = {\bigoplus}_{i \ge 0} CH^*(k,i)$.

Since $L$ is linear, there exists an open dense subset $U \subset L$ such
that $U \cong {\A}^n$ for some $n$ and $L' = L - U$ is a linear variety of
dimension less than that of $L$. Moreover, the localization sequence
\[
0 \to CH^*(L', \cdot) \to CH^*(L, \cdot) \to CH^*(U, \cdot) \to 0
\]
is split exact ({\sl cf.} \cite{Joshua1}).
In particular, this sequence remains exact after tensoring with 
$CH^*(Y, \cdot)$, and we get a diagram of localization exact sequences
\[
\xymatrix@C.5pc{
0 \ar[r] & CH^*(Y, \cdot) {\otimes} CH^*(L', \cdot)  
\ar[r] \ar[d] & CH^*(Y, \cdot) {\otimes} CH^*(L, \cdot) 
\ar[r] \ar[d] & CH^*(Y, \cdot) {\otimes} CH^*(U, \cdot) 
\ar[r] \ar[d] & 0 \\
\ar[r] &  CH^*(Y \times L', \cdot) \ar[r]^{i_*} &
 CH^*(Y \times L, \cdot) \ar[r]^{j^*} &  CH^*(Y \times U, \cdot) \ar[r] &,}
\] 
where the tensor product in the top row is over the ring $CH^*(k)$.
The left vertical arrow is an isomorphism by induction, and the right
vertical arrow is an isomorphism by the homotopy invariance.
In particular, $j^*$ is surjective in all indices. We conclude that
$i_*$ is injective in all indices and the middle vertical arrow is
an isomorphism.
\end{proof}
Recall that a connected and reductive group $G$ over $k$ is said to be
{\sl split}, if it contains a split maximal torus $T$ over $k$ such that
$G$ is given by a root datum relative to $T$. One knows that every
connected and reductive group containing a split maximal torus is
split ({\sl cf.} \cite[Chapter~XXII, Proposition~2.1]{SGA3}).  
In such a case, the normalizer $N$ of $T$ in $G$ and all its connected
components are defined over $k$ and the quotient $N/T$ is the Weyl group
$W$ of the corresponding root datum.
\begin{lem}\label{lem:fibration}
Let $G$ be a connected reductive group and let $T$ be a split maximal torus
of $G$ contained in a Borel subgroup $B$. Put $H = G/N$, where $N$ is the 
normalizer of $T$ in $G$.
Then any {\'e}tale locally trivial $H$-fibration $f: X \to Y$
induces an isomorphism of higher Chow groups
\[
f^*: CH^*(Y, \cdot) \xrightarrow{\cong} CH^*(X, \cdot).
\]
\end{lem}
\begin{proof} We prove the lemma in several steps. In the first step, we
show that the natural map 
\begin{equation}\label{eqn:fibration1}
CH^*(k, \cdot) \to CH^*(H, \cdot)
\end{equation}
is an isomorphism. 

Since $N$ is a closed subgroup of $G$ defined over $k$, it follows from
\cite[Theorem~12.2.1]{Springer} that the quotient $H$ is defined over
$k$. If $W$ denotes the Weyl group of $G$, then the fibration sequence
\[
0 \to W \to G/T \to H \to 0
\]
together with Corollary~\ref{cor:finite2} give an isomorphism 
\begin{equation}\label{eqn:fibration1*}
CH^*(H, \cdot) \cong 
{\left(CH^*\left(G/T, \cdot\right)\right)}^W.
\end{equation}
Moreover, the characteristic filtration $B^u = U_0 \supseteq U_1 \supseteq 
\cdots \supseteq U_n = \{1\}$ of the unipotent radical $B^u$ of $B$
({\sl cf.} proof of Theorem~\ref{thm:Borel}) has the property that
the map $G/TU_i \to G/TU_{i-1}$ is a torsor for the vector group
$U_{i-1}/U_i$ for $0 \le i \le n$. The homotopy invariance then implies that
the map $CH^*\left(G/TU_{i-1}, \cdot\right) \to 
CH^*\left(G/TU_{i}, \cdot\right)$ is an isomorphism. Using this isomorphism
for $i = 0, \cdots, n$, we get an isomorphism
\[
CH^*(H, \cdot) \cong 
{\left(CH^*\left(G/T, \cdot\right)\right)}^W \cong
{\left(CH^*\left(G/B, \cdot\right)\right)}^W.
\]
However, as $G/B$ is a linear variety, 
the natural map 
\[CH^*\left(G/B, 0\right) {\otimes}_{\Q} 
CH^*\left(k, \cdot\right) \to CH^*\left(G/B, \cdot\right)
\]
is an isomorphism ({\sl cf.} \cite{Joshua1}).
In particular, we get 
\[
CH^*(H, \cdot) \cong {\left(CH^*\left(G/B, \cdot\right)\right)}^W
\cong {\left(CH^*\left(G/B, 0\right)\right)}^W {\otimes}_{\Q} 
CH^*\left(k, \cdot\right).\]
The proof of ~\ref{eqn:fibration1} now follows using the isomorphism
${\left(CH^*\left(G/B, 0\right)\right)}^W \cong \Q$ ({\sl loc. cit.}).

Now we consider the case when $X = Y \times H$ and $f$ is the projection map.
Since $G/T \times Y$ is a principal $W$-bundle over $H \times Y = X$,
the same argument as above shows that there is an isomorphism
\[
CH^*\left(X, \cdot \right) \cong 
{CH^*\left(G/B \times Y,  i \right)}^W
\cong {CH^*\left(G/B, \cdot \right)}^W {\otimes}_{CH^*\left(k, \cdot\right)}
CH^*\left(Y, \cdot \right) 
\cong CH^*\left(Y, \cdot \right),
\]
where the second isomorphism follows from Lemma~\ref{lem:linear} and the 
last isomorphism follows from ~\ref{eqn:fibration1*} followed by 
~\ref{eqn:fibration1}.
We finally prove the lemma by the Noetherian induction on $Y$. We can assume 
$Y$ to be reduced. If $Y$ is $0$-dimensional, then
the lemma follows from ~\ref{eqn:fibration1}. So we assume that 
${\rm dim}(Y) \ge 1$ and the lemma holds for when $f$ is restricted to
all proper closed subvarieties of $Y$.     

Since $f$ is {\'e}tale locally trivial, there is an {\'e}tale cover
${\pi}: Y' \to Y$ such that $X' = X {\times}_{Y} Y'$ is isomorphic
to $H \times Y'$ and the pull-back of $f$ to $X'$ is the projection
map $f': X {\times}_{Y} Y' \to Y'$. Moreover, as $\pi$ is dominant
and generically finite, there exists an open set $U \subset Y$ such that
${\pi}_U : U' = {\pi}^{-1}(U) \to U$ is finite and {\'e}tale. Letting
$X_U = f^{-1}(U)$, we thus get a fiber square
\[
\xymatrix{
X'_U \ar[d]_{{\pi}'_U} \ar[r]^{f'_U} &
U' \ar[d]^{{\pi}_U} \\
X_U \ar[r]_{f_U} & U}
\]
where $X'_U = H \times U'$ and $f'_U$ is the projection map.   
Since ${\pi}_U$ is finite {\'e}tale and $f_U$ is smooth, we get a 
commutative diagram ({\sl cf.} Proposition~\ref{prop:EHCG})
\[
\xymatrix{
CH^*\left(U, \cdot \right) \ar[r]^{{\pi}_U^*} \ar[d]^{f_U^*} &
CH^*\left(U', \cdot \right) \ar[r]^{{\pi}_{U*}} \ar[d]^{{f'}_U^*} &
CH^*\left(U, \cdot \right)  \ar[d]^{f_U^*} \\
CH^*\left(X_U, \cdot \right) \ar[r]_{{{\pi}'}_U^*} &
CH^*\left(X'_U, \cdot \right) \ar[r]_{{\pi}'_{U*}} & 
CH^*\left(X_U, \cdot \right)}
\]
and the projection formula implies that ${\pi}_{U*} \circ {\pi}_U^*$ and 
${{\pi}'}_{U*} \circ {{\pi}'}_U^*$ are both multiplication by the degree of 
${\pi}_U$.
On the other hand, we have just shown above that the middle vertical arrow is 
an isomorphism. Hence $f_U^*$ is also an isomorphism.

Now we let $Z = Y - U$ and put $X_Z = f^{-1}(Z)$. This gives the
following commutative diagram of fibration sequences of cycle complexes.
\[
\xymatrix@C.5pc{
{\sZ}^j\left(Z, \cdot\right) \ar[r] \ar[d]^{f^*_Z} &
{\sZ}^j\left(Y, \cdot\right) \ar[r] \ar[d]^{f^*} &
{\sZ}^j\left(U, \cdot\right) \ar[d]^{f^*_U} \\
{\sZ}^j\left(X_Z, \cdot\right) \ar[r] &
{\sZ}^j\left(X, \cdot\right) \ar[r] &
{\sZ}^j\left(X_U, \cdot\right)}
\]
We have just now shown that $f^*_U$ is a quasi-isomorphism.
The left vertical arrow is a quasi-isomorphism by the Noetherian
induction. We conclude that the middle vertical arrow is also
a quasi-isomorphism.
\end{proof}

The following theorem is an analogue of a result of Merkurjev
({\sl cf.} \cite[Proposition~8]{Merkurjev}) about the equivariant
$K$-theory. However, this result for the equivariant higher
Chow groups has an advantage over Merkurjev's theorem in that it holds for 
the action of any split reductive group (though with rational coefficients)
whereas \cite[Proposition~8]{Merkurjev}
is known only for the groups whose derived subgroups are 
simply connected, e.g., $GL_n(k)$.
\begin{thm}\label{thm:reductiveI} 
Let $G$ be a connected and split reductive group and let $T$ be a split maximal
torus of $G$. Then for any $X \in {\sV}_G$, the natural map of $S(T)$-modules
\[
CH^*_G\left(X, \cdot \right) {\otimes}_{S(G)} S(T) \to
CH^*_T\left(X, \cdot \right)
\]
is an isomorphism.
\end{thm}
\begin{proof} It is easy to see using the rules of the intersection product
that the above map is $S(T)$-linear, where the $S(T)$-module structure
on the left is given by the extension of scalars. So we only need to
prove that this map is an isomorphism of abelian groups.
Let $N$ denote the normalizer of $T$ in $G$ and let
$W = N/T$ be the Weyl group. If $(V,U)$ is good pair for the action of $G$, 
then $X_N  \to X_G$ is an {\'e}tale locally trivial
$G/N$-fibration. Hence it follows from Lemma~\ref{lem:fibration} that the map
\begin{equation}\label{eqn:reduction1}
CH^*_G\left(X, \cdot\right) \to CH^*_N\left(X, \cdot\right)
\end{equation}
is an isomorphism. In particular, we have $S(G) \cong S(N)$.
Thus we only need to show that the natural map
\begin{equation}\label{eqn:reductiveI0} 
CH^*_N\left(X, \cdot \right) {\otimes}_{S(N)} S(T) \to
CH^*_T\left(X, \cdot \right)
\end{equation}
is an isomorphism.

We give the Weyl group $W$ a reduced induced scheme structure and let 
$r$ be the cardinality of $W$. Then $CH^*\left(W,0\right) \cong
{CH^*\left(k,0\right)}^{{\oplus}r} \cong {\Q}^{r}$ with a basis 
$\{u_1, \cdots , u_r\}$.
It suffices then to prove that if $(V,U)$ is a good pair for
the action of $N$ such that it acts freely on $Y = X \times U$ and if
$Y/T \xrightarrow{f} Y/N$ is the flat map, then
the natural map
\begin{equation}\label{eqn:reductiveI}
CH^*\left(Y/N, \cdot \right) {\otimes}_{\Q} CH^*\left(W,0\right)
\xrightarrow{{\wt{f}}^*} CH^*\left(Y/T, \cdot \right)
\end{equation}
\[
\left(a_1, \cdots, a_r\right) \mapsto \stackrel{r}{\underset{i= 1}{\Sigma}}
f^*(a_i)
\]
is an isomorphism. For, this would imply that $S(N) {\otimes}_{\Q}
CH^*\left(W,0\right) \xrightarrow{\cong} S(T)$ and 
$CH^*_N\left(X, \cdot \right) {\otimes}_{\Q}  CH^*\left(W,0\right)
\cong CH^*_T\left(X, \cdot \right)$. But the left term of this
isomorphism is same as 
\[
\left(CH^*_N\left(X, \cdot \right){\otimes}_{S(N)} S(N) \right)
{\otimes}_{\Q} CH^*\left(W,0\right) \cong 
CH^*_N\left(X, \cdot \right){\otimes}_{S(N)} S(T).
\] 
In order to prove ~\ref{eqn:reductiveI}, let $Z = Y/N$ and $Z' = Y/T$.
Since $Z' \xrightarrow{f} Z$ is a principal $W$-bundle,
one has a fiber diagram
\[
\xymatrix{
{Z' \times W} \ar[r]^{\pi} \ar[d]_{\pi} & Z' \ar[d]^{f} \\
Z' \ar[r]_{f} & Z}
\]
This gives a commutative diagram
\[
\xymatrix@C.5pc{
CH^*\left(Z, \cdot \right) {\otimes} CH^*\left(W,0\right)  
\ar[r]^{f^* {\otimes} id} \ar[d]^{{\tilde{f}}^*} &
CH^*\left(Z', \cdot \right) {\otimes} CH^*\left(W,0\right) 
\ar[d]^{{\pi}^*} \ar[r]^{f_* {\otimes} id} &
CH^*\left(Z, \cdot \right) {\otimes} CH^*\left(W,0\right) 
 \ar[d]^{{\tilde{f}}^*} \\
CH^*\left(Z', \cdot \right) \ar[r]^{{\pi}^*} &
CH^*\left({Z' \times W}, \cdot \right) \ar[r]^{{\pi}_*} &
CH^*\left(Z', \cdot \right)}
\]
where the tensor product in the top row is over $\Q$.
The right (and the left) vertical map is given by
${\tilde{f}}^*\left(a_1, \cdots , a_r\right) = 
\stackrel{r}{\underset {i=1}{\Sigma}}{f^*(a_i)}$
and so is the middle vertical map. Moreover, the composite horizontal
arrows are identity. Hence ${\tilde{f}}^*$ is an isomorphism as
it is a retract of the middle vertical arrow, which is clearly 
an isomorphism. This proves ~\ref{eqn:reductiveI}
and hence the theorem.
\end{proof}
\begin{cor}\label{cor:reduction}
Let $G$ be a connected and split reductive group and let $T$ be a split 
maximal torus of $G$ with the Weyl group $W$. 
Then for any $X \in {\sV}_G$, the restriction map $r^G_T$
induces an isomorphism
\[
CH^*_G\left(X, \cdot\right) \xrightarrow{\cong} 
{\left(CH^*_T\left(X, \cdot\right)\right)}^W.
\]
\end{cor}
\begin{proof} This follows directly by taking the Weyl group invariants
on the both sides of the isomorphism in Theorem~\ref{thm:reductiveI}
and then using the isomorphism $S(G) = {S(T)}^W$ 
({\sl cf.} \cite[Proposition~6]{ED2}).
\end{proof}
We end this section with the following structure theorem for
the equivariant higher Chow groups of a variety with the action
of a diagonalizable group on which certain subgroup acts trivially. 
An analogous result for the equivariant $K$-theory for the torus action was 
proved by Thomason ({\sl cf.} \cite[Lemma~5.6]{Thomason3}). 
\begin{thm}\label{thm:trivial}
Let $T$ be a split diagonalizable group and let $X \in {\sV}_T$.
Let $H$ be a connected closed subgroup of $T$ which acts 
trivially on $X$. Then there is a natural isomorphism
\[
CH^*_{T/H}\left(X, \cdot \right) {\otimes}_{\Q} 
S(H) \xrightarrow{i^T_H} CH^*_T\left(X, \cdot \right).
\]                   
\end{thm}
\begin{proof} 
Put $T' = T/H$. Since $H$ is a torus, we 
can choose a decomposition (not necessarily canonical) $T = H \times T'$.   
Fix an integer $j \ge 0$ and let $(V, U)$ and $(V',U')$
be good pairs for the actions of $H$ and $T'$ respectively
corresponding to $j$ as in \cite[Example~3.1]{ED2}.
Thus $U$ is a product of punctured affine spaces and
${U}/{H} = {\left({\P}^n\right)}^{r}$ for some $n \gg 0$,
where $r = {\rm rank}(H)$.
Then $(V_T, U_T)$ with $V_T = V \times V'$ and $U_T = U \times U'$,
is a good pair for $j$ for the action of $T$. Note that 
$CH^j_T\left(X, \cdot \right)$ does not depend on the choice
of the decomposition of $T$ since it does not depend on the choice of 
the good pair $(V_T, U_T)$. Now we have 
\[
X_T = \left(X \times U \times U' \right)/({H \times T'})
= \left(X \times U' \right) \stackrel{T'}{\times}{U/{H}}
= X_{T'} \times {\left({\P}^n\right)}^{r},
\]
where the second equality holds since $H$ acts trivially on
$X \times U'$ and the third equality holds because $T'$ acts
trivially on $U$.   
We now apply the projective bundle formula for the non-equivariant  
higher Chow groups to conclude that the map
\[
{\underset {p+q = j}{\bigoplus}} CH^p \left(X_{T'}, \cdot \right) 
{\otimes}_{\Q} CH^q\left({\left({\P}^n\right)}^{r}, 0\right)
\to CH^j\left(X_T, \cdot \right)
\]
is an isomorphism. 
However, $CH^p \left(X_{T'}, \cdot \right) \cong 
CH^p_{T'}\left(X, \cdot \right)$ for all $p \le j$ and
$CH^j\left(X_T, \cdot \right) \cong CH^j_T\left(X, \cdot \right)$.
This finishes the proof.
\end{proof}
\section{Completions Of Equivariant Higher Chow Groups}
Let $G$ be a linear algebraic group. Recall that the {\sl Chow ring}
$S(G)$ of $G$ is the graded ring $CH^*_G\left(k,0\right) = 
\stackrel{\infty}{\underset {j=0}{\bigoplus}} 
CH^j_G\left(k,0\right)$, and $\widehat{S(G)}$ is the $J_G$-adic
completion of $S(G)$, where $J_G$ is the augmentation ideal of cycles 
of positive codimension. For $X \in {\sV}_G$, we can define various
completions of the graded $S(G)$-module $CH^*_G\left(X, \cdot\right)$.
Our objective in this section is to analyze the relation between
these completions. Our main ingredient for this analysis is the
following general algebraic result. 

Let $R_0$ be a commutative Noetherian ring and let $R =  
\stackrel{\infty}{\underset {j=0}{\bigoplus}} R_j$ be 
a finitely generated graded $R_0$-algebra. Let $I = 
\stackrel{\infty}{\underset {j=1}{\bigoplus}} R_j$ be the
irrelevant ideal of $R$. Let $M =  
\stackrel{\infty}{\underset {j=0}{\bigoplus}} M_j$ be a graded
$R$-module which need not be finitely generated.   
Put $M^i = \stackrel{\infty}{\underset {j= i}{\bigoplus}} M_j$
for $i \ge 0$. Then $M^{\cdot}$ defines a filtration of $M$ by 
$R$-submodules. We shall call this the {\sl filtration of $M$ by grading}.
Let $\widehat{R}$ be the $I$-adic completion of $R$.
Associated to $M$, we define the following three modules. \\
$\wt{M} = M {\otimes}_R \widehat{R} \hspace*{1cm} 
(the \ weak \ completion \ of \ M)$ \\
$\widehat{M} = {\widehat{M}}_I  \hspace*{1cm}     
(the \ I-adic \ completion \ of \ M)$ \\
$\ov{M} =$ the completion defined by the filtration $M^{\cdot}$ \\
$\hspace*{3cm} (the \ graded \ completion \ of M)$. \\
Note that $\ov{M}$ is an $\widehat{R}$-submodule of the 
product $\stackrel{\infty}{\underset {i=0}{\prod}} M/{M^i}$ and
the natural map 
\begin{equation}\label{eqn:comp0}
\stackrel{\infty}{\underset {i=0}{\prod}} M_i
\longrightarrow \stackrel{\infty}{\underset {i=0}{\prod}} M/{M^i}
\end{equation}
\[
\left(m_i\right) \mapsto \left((m_0, \cdots , m_{i-1})\right)
\]
identifies $\stackrel{\infty}{\underset {i=0}{\prod}} M_i$
with $\ov{M}$. Moreover,
the map $M \to \ov{M}$ is the natural embedding of the direct sum
into the direct product.  
All the above completions of $M$ are $\widehat{R}$-modules and there
are natural maps
\begin{equation}\label{eqn:completion}
\xymatrix{
M \ar[r] & \wt{M} \ar[r]^{{\phi}_M} \ar@/_1pc/[rr]_{{\psi}_M} &
\widehat{M} \ar[r]^{{\theta}_M} & \ov{M},}
\end{equation}
where ${\phi}_M$ and ${\theta}_M$ (hence their composite ${\psi}_M$) are 
$\widehat{R}$-linear.
\begin{prop}\label{prop:completion1}
For the graded ring $R$, the following hold. \\
$(i)$ The graded completion is an exact functor on the category of
graded $R$-modules. \\
$(ii) \ \ {\phi}_M$ is an isomorphism if $M$ is finitely generated graded 
$R$-module. \\
$(iii) \ \ {\theta}_M$ is an isomorphism if $M$ is generated as 
$R$-module by a (possibly infinite) set $S$ of homogeneous elements of
bounded degree. \\  
$(iv) \ \ {\psi}_M$ (hence ${\phi}_M$) is injective for any
graded $R$-module $M$. \\
$(v) \ \ {\psi}_M$ need not be surjective for any graded $R$-module $M$.
\end{prop}
\begin{proof} To prove $(i)$, we note that a sequence 
\[
0 \to M' \to M \to M'' \to 0 
\]
of graded $R$-linear maps is exact if and only if 
\[
0 \to M'_j \to M_j \to M''_j \to 0 
\]
is exact sequence of $R_0$-modules for every $j \ge 0$.
Equivalently, the sequence  
\[
0 \to \stackrel{\infty}{\underset {j=0}{\prod}} M'_j \to
\stackrel{\infty}{\underset {j=0}{\prod}} M_j \to 
\stackrel{\infty}{\underset {j=0}{\prod}} M''_j \to 0
\]
is exact, which proves $(i)$.
The part $(ii)$ is obvious since $R$ is Noetherian.

For $(iii)$, it suffices to show that the $I$-adic filtration and the 
filtration by grading give the same topology on $M$. We already
have $I^iM \subseteq M^i$ for every $i \ge 0$. So we need to prove that
for given $n \ge 1$, one has $M^i \subseteq I^nM$ for all $i \gg 0$.
Since $R$ is a finitely generated $R_0$-algebra, we can assume that there 
exists a finite set $T$ of homogeneous elements of positive degree in $R$ 
which generate $R$ as an $R_0$-algebra. Let $N$ denote the maximum of the 
bounded degrees of the sets $S$ and $T$, and the cardinality of $T$. It 
suffices then to show that 
\begin{equation}\label{eqn:com1}
M^i \subseteq I^nM \ {\rm for} \ i \ge N' = (N^2n+1)N.
\end{equation}
So let $m \in M_j$ with $j \ge i \ge N'$ and write
$m = \stackrel{r}{\underset {u=0}{\Sigma}}a_u m_u$ with ${\rm deg}(m_u)
\le N$. Then for every
$u$, we must have $j = {\rm deg}(a_u) + {\rm deg}(m_u) \ge N'$,
which implies that $ {\rm deg}(a_u) \ge N' - {\rm deg}(m_u) \ge
N'- N = N^2n$. This shows that
\[
M^i \subseteq R_{\ge N^2n}M.
\]
Thus it suffices to show that $R_{\ge N^2n} \subseteq I^n$ in order to
prove ~\ref{eqn:com1}. So let $a \in R_j$ with $j \ge N^2n$ and
write $a = a_1^{t_1} \cdots a_r^{t_r}$ with $a_u \in T$. 
We need to show that $\stackrel{r}{\underset {u=1}{\Sigma}}t_u \ge n$
(which would imply that $a \in I^n$).
However, otherwise we would get
\[
N^2n \le j = \stackrel{r}{\underset {u=1}{\Sigma}} {\rm deg}(a_u) t_u 
<  \stackrel{r}{\underset {u=1}{\Sigma}} Nn \le N^2n,
\]
which is absurd. This proves $(iii)$.  
   
To prove $(iv)$, let $M$ be a graded $R$-module. Then there exists a
direct system $\{M_{\lambda}\}$ of finitely generated graded $R$-submodules 
of $M$ such that 
${\underset {\lambda}{\varinjlim}} M_{\lambda} \xrightarrow{\cong} M$,
which in turn gives ${\underset {\lambda}{\varinjlim}} \left( 
M_{\lambda} {\otimes}_R {\widehat{R}}\right) \xrightarrow{\cong}
M  {\otimes}_R {\widehat{R}}$. This gives us a commutative diagram
\[
\xymatrix{
{\underset {\lambda}{\varinjlim}} \left( 
M_{\lambda} {\otimes}_R {\widehat{R}}\right) \ar[r] \ar[d] &
{\underset {\lambda}{\varinjlim}} \ov{M_{\lambda}} \ar[d] \\
{M {\otimes}_R \widehat{R}} \ar[r]^{{\psi}_M} & {\ov{M}}.}
\]
The top horizontal arrow is an isomorphism by the parts $(ii)$ and
$(iii)$ of the proposition. We have just seen that the left vertical arrow
is an isomorphism. The right vertical arrow is injective by using
the exactness of the graded completion and the direct limit functors,
plus the fact that each $M_{\lambda} \inj M$. Hence the bottom
horizontal arrow must be injective.

To see $(v)$, take $M = R^{\N}$. Then it is easy to check that
$\wt{M} = {\left(\stackrel{\infty}{\underset {j=0}{\prod}} R_j\right)}^{\N}$
whereas $\ov{M} = \stackrel{\infty}{\underset {j=0}{\prod}} {R_j}^{\N}$. Hence
${\psi}_M$ is not surjective.
\end{proof}
We need a few intermediate results before we give our first application of
Proposition~\ref{prop:completion1} to the equivariant higher Chow groups.
\begin{lem}\label{lem:nonreductive}
Let $G$ be a linear algebraic group over $k$ such that the unipotent
radical $R_u G$ is defined over $k$. Let $L = G/{R_uG}$ be
the reductive quotient of $G$. Then there are natural isomorphisms
$R(L) \xrightarrow{\cong} R(G)$ and $S(L) \xrightarrow{\cong} S(G)$.
\end{lem}
\begin{proof} Since $R_u G$ is defined over $k$, it 
follows from \cite[Corollary~12.2.2]{Springer} that the reductive 
quotient $L$ is also defined over $k$.
Let $V$ be an irreducible representation of $G$. Then the
fixed point theorem ({\sl cf.} \cite[Theorem~17.5]{Hump}) implies
that if the unipotent radical $R_uG$ acts non-trivially on $V$, then
the invariance subspace $W \subset V$ is non-zero. Moreover, $W$ is then
$G$-invariant which contradicts the irreducibility of $V$. We conclude that
$R_uG$ acts trivially on $V$. Since $R(G)$ is a free abelian group on the
set of irreducible representations of $G$, we see that the map
$R(L) \to R(G)$ must be an isomorphism.

For proving the corresponding isomorphism of the Chow rings, we fix any
$j \ge 0$ and choose good pairs $(V_1, U_1)$ and $(V_2, U_2)$ for
$G$ and $L$ respectively, and let $G$ act on $V = V_1 \oplus V_2$
diagonally where it acts on $V_2$ via $L$. Then $(V_1 \times V_2,
U_1 \times U_2)$ is a good pair for $G$. Moreover, the $G$-equivariant 
projection map $U_1 \times U_2 \to U_2$ induces the map on quotients
${\left(U_1 \times U_2 \right)}/G \to {U_2}/G = {U_2}/L$. This gives the 
pull-back map on Chow groups $CH^j_L\left(k, i\right) =
CH^j\left({U_2}/L, i \right) \to 
CH^j\left({\left(U_1 \times U_2 \right)}/G, i \right) = 
CH^j_G\left(k, i\right)$.
This gives the natural map $S(L) \to S(G)$. To show that this is an
isomorphism, we use the isomorphism $R(L) \xrightarrow{\cong} R(G)$ shown
above and the following commutative diagram of completions ({\sl cf.} 
\cite[Lemma~3.2]{ED1}).
\[
\xymatrix{
{\wh{R(L)}} \ar[r] \ar[d] & {\underset {j \ge 0}{\prod}} CH^j_L (k)
\ar[d] & S(L) \ar[d] \ar[l] \\
{\wh{R(G)}} \ar[r] & {\underset {j \ge 0}{\prod}} CH^j_G (k) & 
S(L) \ar[l]}
\]
The horizontal arrows in the left square are ring isomorphisms by 
Theorem~\ref{thm:ED} and we have shown above that the left vertical
arrow is an isomorphism of rings. Thus the middle vertical arrow is
also an isomorphism. The horizontal arrows in the right square are
clearly injective. We conclude from this that the right vertical arrow
is injective. On the other hand, the isomorphism of the middle vertical
arrow and the natural surjection ${\underset {j \ge 0}{\prod}} CH^j_G (k)
\surj CH^j_G (k)$ implies that the map $CH^j_L (k) \to CH^j_G (k)$ is
surjective for each $j \ge 0$ and hence, the map $S(L) \to S(G)$ is
also surjective and consequently an isomorphism.
\end{proof}
\begin{remk}\label{remk:nonr*}
We remark here that isomorphism of $S(L) \to S(G)$ in characteristic
zero follows directly from Proposition~\ref{prop:NRL}. The above indirect
argument is given to take care of the positive characteristic case when
$G$ might not have Levi subgroups.
\end{remk}  
\begin{lem}\label{lem:folklore1}
Let $G$ be a linear algebraic group acting on a quasi-projective variety
$X$ over $k$. Let $l/k$ be a finite extension of $k$ and let 
$CH^*_G\left(X_l, i \right)$ denote the equivariant higher Chow groups
for the action of the group $G_l$ on $X_l$. Then there are natural maps
$CH^*_G\left(X, i \right) \to CH^*_G\left(X_l, i \right) \to
CH^*_G\left(X, i \right)$ such that the composite map is multiplication by
the degree $[l:k]$ of the field extension.
\end{lem}   
\begin{proof} Fix $j \ge 0$ and let $(V, U)$ be a good pair for the
$G$-action corresponding to $j$. Since the map $U \to U/G$ is a principal
$G$-bundle of quasi-projective $k$-schemes, we see that 
$U_l \to {(U/G)}_l$ is a principal $G_l$-bundle. In particular,
$(V_l, U_l)$ is a good pair for the $G_l$-action. 
Similarly, the principal $G$-bundle $X \times U \to X_G$ implies that
$X_l {\times}_l U_l \to {(X_G)}_l$ is a principal $G_l$-bundle, which 
shows that the mixed quotient $X_l \stackrel{G_l} {\times} U_l$
is isomorphic to ${(X_G)}_l$. Using this, we get 
\begin{equation}\label{eqn:folk}
CH^j_G\left(X_l, i \right) \cong 
CH^j\left(X_l \stackrel{G_l} {\times} U_l, i \right)
\cong CH^j\left({(X_G)}_l, i \right).
\end{equation}
Let ${\pi} : {(X_G)}_l \to X_G$ denote the natural finite map.
Then the flatness of ${\pi}$ implies that there are pull-back and
push-forward maps $CH^j\left(X_G, i \right) \xrightarrow{{\pi}^*}
CH^j\left({(X_G)}_l, i \right) \xrightarrow{{\pi}_*} 
CH^j\left(X_G, i \right)$ such that ${\pi}_* \circ {\pi}^*$ is
multiplication by $[l:k]$ by \cite[Corollary~1.4]{Bloch}. 
The proof is completed by combining this with the isomorphism in
~\ref{eqn:folk}.
\end{proof}
\begin{remk} It is easy to see from the above proof that if $l/k$ is a
Galois extension in Lemma~\ref{lem:folklore1}, then 
$CH^*_G\left(X, i \right)$ is in fact the Galois invariant of 
$CH^*_G\left(X_l, i \right)$ under ${\pi}^*$ and ${\pi}_*$
is just the trace map. This follows for example, from
\cite[Lemma~8.4]{Gubeladze} and the non-equivariant Riemann-Roch
isomorphism of \cite{Bloch}.
\end{remk}
We next present the following two elementary results in commutative algebra
which are useful in describing when a subring of a Noetherian ring is 
also Noetherian. Recall that if $B$ is a commutative ring containing a
subring $A$, then one says that $A$ is a {\sl pure subring} of $B$, if
for any $A$-module $M$, the natural map $M \to M {\otimes}_A B$ is
injective. One example of pure subrings which often occurs is the case when 
$A$ is a retract of $B$ as an $A$-module. It is known that pure subrings
share many good properties of the ambient ring. We mention here one such
property that will be useful in this paper.
\begin{lem}\label{lem:folklore} 
Let $R$ be a Noetherian ring, $B$ a finitely generated $R$-algebra, and
$A$ a pure $R$-subalgebra of $B$. Then $A$ is finitely generated over
$R$.
\end{lem}
\begin{proof} {\sl Cf.} \cite[Theorem~1]{Hasi}.
\end{proof}
\begin{lem}\label{lem:folklore1*}
Let $R \subset S$ be an inclusion of commutative rings such that
$R$ is Noetherian and $S$ is finitely generated as $R$-algebra.
Let $G$ be a finite group of $R$-algebra automorphisms of $S$. Then
the ring of invariants $S^G$ is also a finitely generated $R$-algebra,
and hence Noetherian.
\end{lem}
\begin{proof} By \cite[Proposition~7.8]{AM}, one only needs to show that
$S$ is integral over $S^G$. So let $s \in S$ and put
\[
f(s) = {\underset {\sigma \in G}{\prod}} \left(s - {\sigma}(s)\right).
\]
Then $f(s)$ is clearly zero (take $\sigma = 1$ on the right).
On the other hand, it is easy to see that $f(s)$ is a monic polynomial
in $s$ of degree equal to the cardinality of $G$ and with coefficients in
$S^G$. 
\end{proof}
Let $G$ be a linear algebraic group over $k$ and let $X \in {\sV}_G$.
In the rest of this paper, we shall follow the notations for the
various completions defined in the beginning of this section for
the ring $S(G)$ and the graded $S(G)$-module $CH^*_G\left(X, \cdot\right)$.
\begin{cor}\label{cor:completion2}
Let $G$ and $X \in {\sV}_G$ be as above. Then the natural maps of 
$\widehat{S(G)}$-modules
\[
\wt{CH^*_G\left(X, \cdot\right)} \to \widehat{CH^*_G\left(X, \cdot\right)}
\ \ {\and}
\]
\[ 
\wt{CH^*_G\left(X, \cdot\right)} \to \ov{CH^*_G\left(X, \cdot\right)}
\]
are injective.
\end{cor}
\begin{proof} By Proposition~\ref{prop:completion1}, we only need to show
that $S(G)$ is a finitely generated $\Q \ (= S(G)_0)$-algebra. 
Now, there is a finite extension $l/k$ such that all the connected components
of algebraic group $G_l$ are defined over $l$, the identity component
$G^0_l$ is split and its unipotent radical $R_u\left(G^0_l\right)$ is
also defined and split over $l$. By applying Lemma~\ref{lem:folklore1} to
$X = {\rm Spec}(k)$, we see that there are pull-back and push-forward
maps $S(G) \to S(G_l) \to S(G)$ such that the composite map is 
multiplication by the degree $[l:k]$. Moreover, the commutativity
of the intersection product with the pull-back, and the projection formula
({\sl cf.} Proposition~\ref{prop:EHCG}) show that these maps are
$S(G)$-linear. We conclude that $S(G)$ is a retract of $S(G_l)$
as an $S(G)$-module, and hence a pure ${\Q}$-subalgebra.
By Lemma~\ref{lem:folklore}, it suffices to show that $S(G_l)$ is a
finitely generated $\Q$-algebra. Hence we can assume that $G$ has all the
properties described in the beginning of the proof.

Using Corollary~\ref{cor:finite2} and Lemma~\ref{lem:folklore1*}, we
can assume that $G$ is connected. By Lemma~\ref{lem:nonreductive}, we
can further assume that $G$ is reductive. Since $G$ is split, we can now
use Corollary~\ref{cor:reduction} and Lemma~\ref{lem:folklore1*} once again
to reduce to the case when $G$ is a split torus. But then $S(G)$ is known to 
be a finitely generated polynomial algebra over $\Q$.
\end{proof}
\section{Cohomological Rigidity And Specializations}
Let $G$ be a split diagonalizable group over $k$ acting on a smooth variety
$X$. Recall ({\sl cf.} \cite[13.2.5]{Springer}) that all the diagonalizable
subgroups of $G$ are defined and split over $k$.
The equivariant $K$-theory of $X$ for the $G$-action was studied
by Vezzosi and Vistoli in \cite{VV}. Their main result (Theorem~1) is  
to show how to reconstruct the $K$-theory ring of $X$ in terms of the
equivariant $K$-theory of the loci where the stabilizers have constant
dimension. In the next two sections, we use the ideas of Vezzosi-Vistoli 
to prove an analogous decomposition theorem for the equivariant higher 
Chow groups of $X$ for the $G$-action. This theorem and its compatibility 
with the corresponding result for $K$-theory will be crucial in the proof
of the main results of this paper. Like in the case of 
$K$-theory ({\sl loc. cit.}), this decomposition can also be used to compute 
the equivariant higher Chow groups of many varieties with an action of the 
group $G$ such as the toric varieties.
This section is concerned with the study of cohomological rigidity and
the construction of the specialization maps in equivariant higher Chow groups.
For the rest of this section and the next, the group $G$ will always denote a 
split diagonalizable group and the varieties will be assumed to be smooth with 
$G$-action.    
We have seen ({\sl cf.} Proposition~\ref{prop:EHCG}) that for such a variety 
$X$, $CH^*_G\left(X, \cdot\right)$ is a bigraged ring which is an
algebra over the ring $CH^*_G\left(k, \cdot\right)$. We denote the full
equivariant Chow ring ${\underset {j, i \ge 0}{\bigoplus}} 
CH^j_G\left(X, i \right)$ of $X$ in short by $CH^*_G\left(X\right)$.
\subsection{Cohomological rigidity}  
\begin{defn}\label{defn:rigidity}
Let $Y \subset X$ be a smooth and $G$-invariant closed subvariety of 
codimension $d \ge 1$ and
let $N_{Y/X}$ denote the normal bundle of $Y$ in $X$. We say that $Y$ is
{\sl cohomologically rigid} inside $X$ if $c^G_d\left(N_{Y/X}\right)$ is
a not a zero-divisor in the ring $CH^*_G\left(Y\right)$.
\end{defn}
As one observes, this definition has any reasonable meaning only in
the equivariant setting, since every element of positive degree in the
non-equivariant  Chow ring is nilpotent. The importance of cohomological
rigidity for the equivariant higher Chow groups comes from the following
analogue of the $K$-theory splitting theorem (Proposition~4.3) of
{\sl loc. cit.}.   
\begin{prop}\label{prop:split}
Let $Y$ be a smooth and $G$-invariant closed subvariety of $X$ of
codimension $d \ge 1$. Assume that $Y$ is cohomologically rigid inside $X$,
and put $U = X - Y$. Let $Y \stackrel{i}{\inj} X$ and
$U \stackrel{j}{\inj} X$ be the inclusion maps. Then \\
$(i)$ The localization sequence 
\[
0 \to CH^*_G(Y) \xrightarrow{i_*} CH^*_G(X) \xrightarrow{j^*}
CH^*_G(U) \to 0
\]
is exact. \\
$(ii)$ The restriction ring homomorphisms
\[
CH^*_G(X) \stackrel{(i^*, j^*)}{\longrightarrow} 
CH^*_G(Y) \times CH^*_G(U)
\]
give an isomorphism of rings
\[
CH^*_G(X) \xrightarrow{\cong} CH^*_G(Y) {\underset {\wt{CH^*_G(Y)}}
{\times}} CH^*_G(U),
\]
where $\wt{CH^*_G(Y)} = {CH^*_G(Y)}/{\left(c^G_d\left(N_{Y/X}\right)\right)}$,
and the maps 
\[
CH^*_G(Y) \to \wt{CH^*_G(Y)}, \ CH^*_G(U) \to \wt{CH^*_G(Y)}
\]
are respectively, the natural surjection and the map 
\[
CH^*_G(U) = {\frac{CH^*_G(X)}{i_*\left(CH^*_G(Y)\right)}}
\xrightarrow{i^*} {\frac{CH^*_G(Y)}{c^G_d\left(N_{Y/X}\right)}}  
= \wt{CH^*_G(Y)},
\]
which is well-defined by Corollary~\ref{cor:ESIF}.
\end{prop}
\begin{proof} The part $(i)$ follows directly from 
Corollary~\ref{cor:ESIF}
and the definition of cohomological rigidity. Since $i^*$ and 
$j^*$ are ring homomorphisms, the proof of the second part follows
directly from the first part and [{\sl loc. cit.}, Lemma~4.4].
\end{proof} 

To apply the above result in our context, we need to have some 
sufficient conditions for checking the cohomological rigidity in 
specific examples. We first have the following elementary result.
\begin{lem}\label{lem:elem}
Let $A$ be a ring which is a $\Q$-algebra. Then an element of the form
$t^d$, where $t = \stackrel{r}{\underset {j=1}{\Sigma}} 
a_j t_{i_j} \in A[t_1, \cdots , t_n]$, is not a zero-divisor for every
$d \ge 1$, whenever $a_j \in {\Q}$ for all $j$ and $a_j \neq 0$
for some $j$.
\end{lem}
\begin{proof} We can assume that $a_j \neq 0 \ \forall j$.
We prove by induction on $n$. For $n = 1$, it is obvious.
So we assume that the lemma holds for all $n' \le n-1$ and $n \ge 2$.

If $r=1$, then also the lemma is again obvious.
So we assume $r \ge 2$.
Let $\underline{t} = \left(t_1, \cdots , t_n\right)$, and let
$f(\underline{t}) = f_0(\underline{t}) + \cdots + f_p(\underline{t})$ be a
polynomial such that each $f_i$ is homogeneous of degree $d_i$ such that
$0 \le d_0 < \cdots < d_p$. If $f(\underline{t}) \neq 0$ and
$t^df(\underline{t}) = 0$, then using the fact that $t^d$ is homogeneous,
it is easy to see that $t^df_q(\underline{t}) = 0$, where $d_q$ is the 
largest integer such that $f_q \neq 0$. Thus we can assume that 
$f(\underline{t}) \neq 0$ is homogeneous of degree, say $p$.

Let $f(\underline{t}) = b_1 f_1 + \cdots + b_mf_m$ be the unique
representation of $f(\underline{t})$ as an $A$-linear combination of
linearly independent monomials $f_i$'s of homogeneous degree $p$. Let
$s \ge 0$ be the largest integer such that $t_1^s$ divides each $f_i$.
Then
$t^d f(\underline{t}) = t^d t_1^s \left(b_1f'_1 + \cdots + b_mf'_m\right)$,
where some $f'_j$, say $f'_1$ is not divisible by $t_1$. 
Now by $r =1$ case, $t^d  f(\underline{t}) = 0$ implies that 
$t^d f'(\underline{t}) = 0$,
where $f'(\underline{t}) = \stackrel{m}{\underset {j=1}{\Sigma}} 
b_j f'_j$. For any element $g \in A[t_1, \cdots , t_n]$, let
$\bar{g}$ denote its image in the quotient $A[t_2, \cdots , t_n]$.
Then we get ${\bar{t}}^d {\bar{f'}} = 0$ in $A[t_2, \cdots , t_n]$.
However, $r \ge 2$ implies that $\bar{t} \neq 0$ and ${\bar{f'}}_1 
\neq 0$ by our choice. By induction on $n$, this leads to a contradiction.
\end{proof}  
\begin{prop}\label{prop:rigiditysuff}
Let $G$ be a split diagonalizable group acting on a smooth variety $X$ and let
$E$ be a $G$-equivariant vector bundle of rank $d$ on $X$. Assume that
there is a subtorus $T \subset G$ of positive rank which acts trivially
on $X$, such that in the eigenspace decomposition of $E$ with respect to 
$T$, the submodule corresponding to the trivial character is zero. Then
$c^G_d(E)$ is not a zero-divisor in $CH^*_G\left(X\right)$.
\end{prop}
\begin{proof} By \cite[Lemma~5.6]{Thomason3}, $E$ has a unique direct
sum decomposition 
\[
E = \stackrel{r}{\underset {i=1}{\bigoplus}} E_{{\chi}_i} {\otimes}
{\chi}_i,
\]
where we choose a splitting $G = D \times T$, and $E_{{\chi}_i}$'s
are $D$-bundles and ${\chi}_i$'s are $1$-dimensional representations of
$T$. This decomposition is via the functor
\[
{\rm Bun}^D_X \times {\rm Rep}(T) \to  {\rm Bun}^G_X
\]
\[
\left(F, \rho \right) \mapsto p_1^*(F) {\otimes} p_2^*(\rho),
\]
where $p_1: D \times T \to D$ and $p_2: D \times T \to T$ are the projections.

Since ${\rm rank}(E) = d$, we have by the Whitney sum formula,    
$c^G_d(E) = \stackrel{r}{\underset {i=1}{\prod}}c^G_{d_i}
\left(E_{{\chi}_i} {\otimes}{\chi}_i\right)$, where $d_i =
{\rm rank}(E_i)$. Thus we can assume that $E = 
E_{\chi} {\otimes}{\chi}$, where $\chi$ is not a trivial character by
our assumption. In particular, if $L_{\chi}$ is the corresponding line 
bundle in ${\rm Pic}_T(k)$, then 
\begin{equation}\label{eqn:NZD}
c^T_1\left(L_{\chi}\right) = t = 
\stackrel{p}{\underset {i=1}{\Sigma}} n_i t_i \in {\Q}[t_1, \cdots , t_n]
\end{equation}
with $n_i \neq 0$ for some $i$. By neglecting those $i$ for which the
coefficients $n_i$'s are zero, we can assume that $n_i \neq 0 \ \forall
i$. Now we have
\[ 
\begin{array}{lll}
c^G_d(E) & = & c^G_d\left(p_1^*\left(E_{\chi}\right) {\otimes}
p_2^*\left(L_{\chi}\right)\right) \\
& = &  \stackrel{d}{\underset {i=0}{\Sigma}}  
c^G_{d-i}\left(p_1^*\left(E_{\chi}\right)\right) \cdot 
{\left(c^G_1\left(p_2^*\left(L_{\chi}\right)\right)\right)}^i \\
& = &  \stackrel{d}{\underset {i=0}{\Sigma}}  
p_1^*\left(c^D_{d-i}\left(E_{\chi}\right)\right) \cdot 
p_2^*\left({\left(c^T_1\left(L_{\chi}\right)\right)}^i\right) \\
& = & \stackrel{d}{\underset {i=0}{\Sigma}} {\alpha}_i t^i,
\end{array}
\]
where ${\alpha}_i \in CH^*_D\left(X\right)$ and $c^G_d(E) \in
CH^*_G\left(X\right) \cong CH^*_D\left(X\right) {\otimes} S(T)$
({\sl cf.} Theorem~\ref{thm:trivial}) 
and the second equality holds by \cite[Remark~3.2.3]{Fulton}. 
Furthermore, one has ${\alpha}_d = 
p_1^*\left(c^D_{0}\left(E_{\chi}\right)\right) = 1$.
Thus we get $c^G_d(E) = t^d + {\alpha}_{d-1} t^{d-1} + \cdots
+ {\alpha}_1 t + {\alpha}_0 = g(t)$.

We need to show that $g(t)$ is not a zero divisor in
$CH^*_D\left(X\right)[t_1, \cdots , t_n]$. So suppose $f(\underline{t})$
is a non-zero   polynomial such that $g(t) f(\underline{t}) = 0$,
and let $f'(\underline{t})$ be the homogeneous part of $f(\underline{t})$  
of largest degree which is not zero. By comparing the
homogeneous parts, it is easy to see that $g(t) f(\underline{t}) = 0$
only if $t^d f'(\underline{t}) = 0$. But this is a contradiction since
$t$ satisfies the condition of Lemma~\ref{lem:elem} by ~\ref{eqn:NZD},
and hence is not a zero-divisor.
\end{proof}  

Let $G$ be a split diagonalizable group as above and let $X \in {\sV}^S_G$.
Following the notations of \cite{VV}, for any $s \ge 0$, we let $X_{\le s}
\subset X$ be the open subset of points whose stabilizers have dimension
at most $s$. We shall often write $X_{\le s-1}$ also as $X_{< s}$.
Let $X_s = X_{\le s} - X_{< s}$ denote the locally closed subset of $X$,
where the stabilizers have dimension exactly $s$. We think of $X_s$
as a subspace of $X$ with the reduced induced structure. It is clear 
that $X_{\le s}$ and $X_s$ are $G$-invariant subspaces of $X$. Let $N_s$ 
denote the normal bundle of $X_s$ in $X_{\le s}$, and let $N_s^0$ denote the 
complement of the 0-section in $N_s$. Then $G$ clearly acts on $N_s$.
The following result of {\sl loc. cit.} describes some
properties of these subspaces which will be useful to us in what follows.
\begin{prop}[\cite{VV}]\label{prop:strata}
Let $s$ be non-zero integer. \\
$(i)$ There exists a finite number of $s$-dimensional subtori 
$T_1, \cdots , T_r$ in $G$ such that $X_s$ is the disjoint union of
the fixed point spaces $X^{T_j}_{\le s}$. \\
$(ii)$ $X_s$ is smooth locally closed subvariety of $X$. \\
$(iii) \ N^0_s = {\left(N_s\right)}_{< s}$. 
\end{prop}
\begin{proof} See [{\sl loc. cit.}, Proposition~2.2].
\end{proof}
\begin{remk}\label{remk:stratasingular}
We mention here that although the above proposition has been stated for
the smooth varieties, the part $(i)$ of the proposition holds
also when $X$ is not necessarily smooth, since the proof only uses
Thomason's generic {\'e}tale slice theorem 
\cite[Proposition~4.10]{Thomason1} which holds very generally.
\end{remk} 
We have the following important application of 
Proposition~\ref{prop:rigiditysuff}.
\begin{cor}\label{cor:rigidstrata}
For $s \ge 1$, $X_s$ is cohomologically rigid inside $X_{\le s}$.
\end{cor}
\begin{proof} Let $d_s$ be the codimension of $X_s$ in $X_{\le s}$.
We need to show that $c^G_{d_s}(N_s)$ is not a zero-divisor in
$CH^*_G(X_s)$. By Proposition~\ref{prop:rigiditysuff}, it suffices to
show that there exists a subtorus $T$ in $G$ of positive rank which
acts trivially on $X_s$, such that in the eigenspace decomposition of
$N_s$ with respect to $T$, the submodule corresponding to the trivial
character is zero. But this follows directly from the parts $(i)$ and 
$(iii)$ of Proposition~\ref{prop:strata} and the fact that $s \ge 1$
(see [{\sl loc. cit.}, Proposition~4.6]).
\end{proof}

\subsection{Specialization maps}
Let $G$ and $X$ be as above, and let $n$ be the dimension of $G$.
As seen above, there is a filtration of $X$ by $G$-invariant open
subsets $\emptyset = X_{\le -1} \subset X_{\le 0} \subset \cdots 
\subset X_{\le n} = X$. In particular, $G$ acts on $X_{\le 0}$ 
with finite stabilizers, and the toral component of $G$ acts trivially
on $X_n$. We fix $1 \le s \le n$ and let $X_s \stackrel{f_s}{\inj}
X_{\le s}$ and $X_{< s} \stackrel{g_s}{\inj} X_{\le s}$ denote the closed 
and the open embeddings respectively. Let ${\pi}: M_s \to {\P}^1$
be the deformation to the normal cone for the embedding $f_s$ as 
in Section~2. We have already observed there that for the trivial action of 
$G$ on ${\P}^1$, $M_s$ has a natural $G$-action. Moreover,
the deformation diagram ~\ref{eqn:DNC} is a diagram of smooth $G$-spaces.
For $0 \le t \le s$, we shall often denote the open subspace
${\left(M_s\right)}_{\le t}$ of $M_s$ by $M_{s, \le t}$. The terms like
$M_{s,t}$ and $M_{s, < t}$ (and also for $N_s$) will have similar
meaning in what follows. Since $G$ acts trivially on ${\P}^1$, it acts
on $M_s$ fiberwise, and one has $N_s = {\pi}^{-1}({\infty})$ and
\begin{eqnarray}\label{eqn:nolabel}
M_{s, \le t} \cap N_s  = N_{s, \le t}; \ \ 
M_{s,t} \cap N_s  = N_{s,t}. \\
M_{s, \le t} \cap {\pi}^{-1}({\A}^1) = X_{\le t} \times {\A}^1 ; \
M_{s,t} \cap {\pi}^{-1}({\A}^1) = X_{t} \times {\A}^1 \\
M_{s, \le t} \cap {\pi}^{-1}(\infty) =  N_{s, \le t} ; \
M_{s,t} \cap {\pi}^{-1}(\infty) =  N_{s,t}. 
\end{eqnarray}
Let $N_{s, \le t} \stackrel{i_{s, \le t}}{\inj} M_{s, \le t}$ and 
$X_{\le t} \times {\A}^1 \stackrel{j_{s,\le t}}{\inj} M_{s, \le t}$
denote the obvious closed and open embeddings. We define $i_{s,t}$ and 
$j_{s,t}$ similarly. Let $N_{s,t} \stackrel{{\eta}_{s,t}}{\inj}
N_{s, \le t}$ and $M_{s,t} \stackrel{{\delta}_{s,t}}{\inj} M_{s,\le t}$
denote the other closed embeddings. 
One has a commutative diagram
\begin{equation}\label{eqn:DNC*}
\xymatrix{
X_{\le t} \ar[r]^{g_{\le t}} \ar[d]_{f_{s, \le t}} & 
{X_{\le t} \times {\A}^1}
\ar[r]^{j_{s, \le t}} \ar[d] & M_{s, \le t} \ar[d] & & \\
X_{\le s} \ar[r]^{g_{0, \le s}} & X_{\le s} \times {\A}^1 \ar[r]^{j_{\le s}}
& M_s \ar[r] & X_{\le s} \times {\P}^1 \ar[r] &  X_{\le s},}
\end{equation}  
where $g_{\le t}$ is the 0-section embedding, and the composite of
all the maps in the bottom row is identity.
This gives us the following  diagram of equivariant higher Chow groups,
where all squares commute ({\sl cf.} Proposition~\ref{prop:EHCG}).
\begin{equation}\label{eqn:DNC*1}
\xymatrix{
CH^*_G\left(N_{s,t}\right) \ar[r]^{{{i_{s,t}}_*}} 
\ar[d]^{{{\eta}_{s,t}}_*} & 
CH^*_G\left(M_{s,t}\right) \ar[r]^{j_{s,t}^*} \ar[d]^{{{\delta}_{s,t}}_*}  & 
CH^*_G\left(X_{t} \times {\A}^1 \right) \ar[r]^{{g_{t}^*}} 
\ar[d]^{{f_t}_*} & CH^*_G\left(X_{t}\right) \ar[d]^{{f_t}_*} \\
CH^*_G\left(N_{s, \le t}\right) \ar[r]^{{{i_{s, \le t}}_*}} &
CH^*_G\left(M_{s, \le t}\right) \ar[r]^{j_{s, \le t}^*} &
CH^*_G\left(X_{s, \le t} \times {\A}^1 \right) 
\ar[r]^{\hspace*{.5cm} {g_{\le t}^*}} & 
CH^*_G\left(X_{\le t}\right)}
\end{equation}
Since the last horizontal maps in both rows are natural isomorphisms by the 
homotopy invariance, we shall often identify the last two terms in both
rows and use ${j_{s, \le t}^*}$ and 
${\left(j_{s, \le t} \circ g_{\le t}\right)}^*$ interchangeably.
\begin{thm}\label{thm:specialization}
The maps ${j_{s, \le t}^*}$ and ${j_{s,t}^*}$ are surjective and there
are ring homomorphisms 
\[
{\ov{Sp}}_{X, s}^{\le t} : CH^*_G\left(X_{\le t}\right)
\to CH^*_G\left(N_{s,\le t}\right);
\]
\[ 
{\ov{Sp}}_{X, s}^{t} : CH^*_G\left(X_{t}\right)
\to CH^*_G\left(N_{s, t}\right)
\]
such that $i_{s,\le t}^* = {\ov{Sp}}_{X, s}^{\le t} \circ
{j_{s, \le t}^*}$ and $i_{s,t}^* = {\ov{Sp}}_{X, s}^{t} \circ
{j_{s,t}^*}$. Moreover, both the squares in the diagram
\begin{equation}\label{eqn:main}
\xymatrix{
CH^*_G\left(X_{\le t}\right) \ar[r]^{f_t^*}   
\ar[d]_{{\ov{Sp}}_{X, s}^{\le t}} & CH^*_G\left(X_{t}\right)
\ar[d]^{{\ov{Sp}}_{X, s}^{t}} \ar[r]^{{f_t}_*} &
CH^*_G\left(X_{\le t}\right) \ar[d]^{{\ov{Sp}}_{X, s}^{\le t}} \\
CH^*_G\left(N_{s,\le t}\right) \ar[r]_{{\eta}_{s,t}^*} &
CH^*_G\left(N_{s,t}\right) \ar[r]_{{{\eta}_{s,t}}_*} &
CH^*_G\left(N_{s,\le t}\right)}
\end{equation} 
commute.
\end{thm}   
\begin{proof} Using the results of this section and the previous ones,
one can prove this theorem along the lines of the proof of the analogous
result [{\sl loc. cit.}, Theorem~3.2] for $K$-theory as given in
\cite{VV1}. However, it is not at all clear from the construction of
the specialization maps in \cite{VV1} that these maps have good functorial 
properties, and if they are ring homomorphisms. In particular, it is not 
clear if these maps will 
have the compatibility properties with the Chern character and Riemann-Roch 
maps from the equivariant $K$-groups to higher Chow groups.
We give here a more direct and functorial construction of 
the specialization maps, which works both for the 
$K$-theory as well as the higher Chow groups, and the proof of various
compatibilities of these maps then becomes essentially obvious. We
give here the construction of these maps for the higher Chow groups.
The same construction works also for the $K$-theory without any change.

First of all, using Corollary~\ref{cor:rigidstrata} and 
Proposition~\ref{prop:split}, we see that for $1 \le s \le n$ and
$0 \le t \le s$, the map $CH^*_G\left(X_{\le s}\right) \to
CH^*_G\left(X_{\le t}\right)$ is surjective. We now consider the
commutative diagram
\[
\xymatrix{
CH^*_G\left(M_{s}\right) \ar[r]^{j_{\le s}^*} \ar[d] &
CH^*_G\left(X_{\le s}\right) \ar[d] \\
CH^*_G\left(M_{s, \le t}\right) \ar[r]^{j_{s, \le t}^*} &
CH^*_G\left(X_{\le t}\right).}
\]
Since the composite map in the bottom row of ~\ref{eqn:DNC*} is identity, 
we see by the homotopy invariance that ${j_{\le s}^*}$ is surjective.
Thus ${j_{s, \le t}^*}$ is also surjective.
Applying this surjectivity for ${j_{s, \le t}^*}$ and ${j_{s, \le t-1}^*}$,
we obtain the following commutative diagram
\begin{equation}\label{eqn:special}
\xymatrix{
& 0 \ar[d] & 0 \ar[d] & 0 \ar[d] & \\
0 \ar[r] & CH^*_G\left(N_{s,t}\right) 
\ar[r]^{{{i_{s,t}}_*}} \ar[d]^{{{\eta}_{s,t}}_*} &
CH^*_G\left(M_{s,t}\right) 
\ar[r]^{j_{s,t}^*} \ar[d]^{{{\delta}_{s,t}}_*} & 
CH^*_G\left(X_{t}\right) \ar[r] \ar[d]^{{f_t}_*} & 0 \\
0 \ar[r] & CH^*_G\left(N_{s,\le t}\right) 
\ar[r]^{{{i_{s, \le t}}_*}} \ar[d] &
CH^*_G\left(M_{s,\le t}\right) 
\ar[r]^{j_{s, \le t}^*} \ar[d] & 
CH^*_G\left(X_{\le t}\right) \ar[r] \ar[d] & 0 \\
0 \ar[r] & CH^*_G\left(N_{s,\le t-1}\right) 
\ar[r]^{{{i_{s, \le t-1}}_*}} \ar[d] &
CH^*_G\left(M_{s,\le t-1}\right) 
\ar[r]^{j_{s, \le t-1}^*} \ar[d] & 
CH^*_G\left(X_{\le t-1}\right) \ar[r] \ar[d] & 0 \\ 
& 0 & 0 & 0 & }
\end{equation}
such that the second and the third rows are exact. All the columns
are exact by Corollary~\ref{cor:rigidstrata} and 
Proposition~\ref{prop:split}. We conclude that the localization sequence
of the top row is also exact. This proves the surjectivity part of the 
theorem.

Next we note from ~\ref{eqn:nolabel} that $N_{s,\le t}$ and 
$N_{s,t}$ are the principal Cartier divisors on $M_{s,\le t}$ and
$M_{s,t}$ respectively. We conclude from Theorem~\ref{thm:SIF} that
the composites $i_{s,\le t}^* \circ {i_{s,\le t}}_*$ and 
$i_{s,t}^* \circ {i_{s,t}}_*$ are zero. The above diagram now
automatically defines the specializations ${\ov{Sp}}_{X, s}^{\le t}$
and ${\ov{Sp}}_{X, s}^{t}$ and gives the desired factorization
of $i_{s,\le t}^*$ and $i_{s,t}^*$. Since $i_{s,t}^*$ and
$j_{s,t}^*$ are ring homomorphisms, and since the latter is
surjective as shown in ~\ref{eqn:special}, we deduce that 
${\ov{Sp}}_{X, s}^{t}$ is also a ring homomorphism. The map 
${\ov{Sp}}_{X, s}^{\le t}$ is a ring homomorphism for the same reason.

We are now left with the proof of the commutativity of ~\ref{eqn:main}.  
To prove that the right square commutes, we consider the following
diagram.
\begin{equation}\label{eqn:nospecial}
\xymatrix@C.5pc{
CH^*_G\left(M_{s,t}\right) \ar@{->>}[drr]^{j_{s,t}^*} 
\ar[rrr]^{{{\delta}_{s,t}}_*} 
\ar[dd]^{i_{s,t}^*} & & & 
CH^*_G\left(M_{s,\le t}\right) \ar@{->>}[drr]^{j_{s,\le t}^*} 
\ar[dd]_{i_{s,\le t}^*} & & \\
& & CH^*_G\left(X_{t}\right) \ar[lld]^{{\ov{Sp}}_{X, s}^{t}} 
\ar[rrr]^{{f_t}_*} & & & 
CH^*_G\left(X_{\le t}\right) \ar[lld]^{{\ov{Sp}}_{X, s}^{\le t}} \\
CH^*_G\left(N_{s, t}\right) \ar[rrr]^{{{\eta}_{s,t}}_*} & & &
CH^*_G\left(N_{s, \le t}\right) & & }
\end{equation}
It is easy to check that $N_{s, \le t}$ and $M_{s,t}$ are 
Tor-independent over $M_{s,\le t}$ and hence the back face of the above
diagram commutes by Lemma~\ref{lem:torind}. The upper face commutes
by diagram ~\ref{eqn:special}. Since $j_{s,t}^*$ is surjective, a
diagram chase shows that the lower face also commutes, which is what
we needed to prove.

Finally, since we have shown that ${{{\eta}_{s,t}}_*}$ is injective, and 
the right square commutes, it now suffices to show that
the composite square in ~\ref{eqn:main} commutes in order to show that
the left square commutes. 

By the projection formula, the composite maps ${f_t}_* \circ f_t^*$
and ${{{\eta}_{s,t}}_*} \circ {\eta}_{s,t}^*$ are multiplication by
${f_t}_* (1)$ and ${{{\eta}_{s,t}}_*}(1)$ respectively. Since 
${\ov{Sp}}_{X, s}^{\le t}$ and ${\ov{Sp}}_{X, s}^{t}$ are ring
homomorphisms, it suffices to show that 
\[
{\ov{Sp}}_{X, s}^{\le t}\left({f_t}_* \circ j_{s,t}^* (1)\right) =
{\ov{Sp}}_{X, s}^{\le t}\left({f_t}_* (1)\right) = {{{\eta}_{s,t}}_*}(1).
\]
But this follows directly from the commutativity of the right square.
\end{proof}  
\begin{lem}\label{lem:torind}
Let $G$ be a linear algebraic group and let 
\[
\xymatrix{
W \ar[r]^{i'} \ar[d]_{j'} & Y \ar[d]^{j} \\
Z \ar[r]_{i} & X}
\]
be a fiber diagram of closed immersions of $G$-varieties such that $X$ and 
$Y$ are smooth.
Then one has $i^*\circ j_* = {j'}_* \circ {i'}^* : 
CH^*_G\left(Y, \cdot \right) \to CH^*_G\left(Z, \cdot \right)$.   
\end{lem}
\begin{proof} By choosing a good pair $(V,U)$ for the $G$-action and then
considering the appropriate mixed quotients, we can reduce to proving
the lemma for the non-equivariant  higher Chow groups. But this is shown
in Lemma~\ref{lem:commute}.
\end{proof} 

\section{Decomposition Theorem For Equivariant Higher Chow Groups}
We use the specialization maps to prove a decomposition theorem
for the equivariant higher Chow groups of $X \in {\sV}^S_G$,
where $G$ is a split diagonalizable group. We continue with the notations
of the previous section. 
\begin{prop}\label{prop:decomposition1}
The restriction maps 
\[
CH^*_G\left(X_{\le s}\right) \stackrel{(f_s^*, g_s^*)}{\longrightarrow}
CH^*_G\left(X_{s}\right) \times CH^*_G\left(X_{< s}\right)
\]
define an isomorphism of rings
\[
CH^*_G\left(X_{\le s}\right) \xrightarrow{\cong}
CH^*_G\left(X_{s}\right) {\underset {CH^*_G\left(N_s^0 \right)}
{\times}} CH^*_G\left(X_{< s}\right),
\]
where $CH^*_G\left(X_{s}\right) \xrightarrow{{\eta}_{s,\le s-1}^*}
CH^*_G\left(N_s^0 \right)$
is the pull-back
\[
CH^*_G\left(X_{s}\right) \xrightarrow{\cong} 
CH^*_G\left(N_{s}\right) \to CH^*_G\left(N_s^0 \right)
\]
and 
\[
CH^*_G\left(X_{< s}\right) \stackrel{{{\ov{Sp}}_{X, s}^{\le s-1}}}
{\longrightarrow} CH^*_G\left(N_{s, \le s-1}\right) =
CH^*_G\left(N_s^0 \right)
\]
is the specialization map of Theorem~\ref{thm:specialization}.
\end{prop} 
\begin{proof} We only need to identify the pull-back and the specialization
maps with the appropriate maps of Proposition~\ref{prop:split}.
In the diagram 
\[
\xymatrix{
0 \ar[r] & CH^*_G\left(X_{s}\right) \ar[r]^{{f_{s, \infty}}_*} 
\ar[dr]_{c^G_{d_s}} & 
CH^*_G\left(N_{s}\right) \ar[r]^{{\eta}_{s, \le s-1}^*} 
\ar[d]^{f_{s, \infty}^*} &
CH^*_G\left(N_s^0 \right) \ar[r] & 0 \\
& & CH^*_G\left(X_{s}\right) & & }
\]
where $f_{s, \infty}: X_s \to N_s$ is the 0-section embedding, the top
sequence is exact, and the lower
triangle commutes by Corollary~\ref{cor:ESIF}. Since ${f_{s, \infty}^*}$
is an isomorphism, this immediately identifies the pull-back map
of the proposition with the quotient map
$CH^*_G\left(X_{s}\right) \to \frac{CH^*_G\left(X_{s}\right)}
{\left({c^G_{d_s}(N_s)}\right)}$.

Next we consider the following diagram.
\[
\xymatrix{
CH^*_G\left(X_{\le s}\right) \ar[d]_{{{{\ov{Sp}}_{X, s}^{\le s}}}}
\ar@{->>}[r]^{f_{s,\le s-1}^*} & CH^*_G\left(X_{< s}\right)
\ar[d]^{{{{\ov{Sp}}_{X, s}^{\le s-1}}}} \\
CH^*_G\left(N_{s}\right) \ar[r] \ar[d]_{f_{s, \infty}^*} &
CH^*_G\left(N_s^0 \right) \\
CH^*_G\left(X_{s}\right) \ar[ur]_{{\eta}_{s, \le s-1}^*} & }
\]
Since the top horizontal arrow in the above diagram is surjective,
we only need to show that
${{{{\ov{Sp}}_{X, s}^{\le s-1}}}} \circ {f_{s,\le s-1}^*}
= {{\eta}_{s, \le s-1}^*} \circ f_s^*$ in order to identify 
${{\ov{Sp}}_{X, s}^{\le s-1}}$ with the map $j^*$ of
Proposition~\ref{prop:split}.
It is clear from the diagram ~\ref{eqn:special} and the definition of the 
specialization maps that the top square above commutes. We have just
shown above that the lower triangle also commutes. This reduces us to
showing that 
\begin{equation}\label{eqn:special2}
{f_{s, \infty}^*} \circ {{{{\ov{Sp}}_{X, s}^{\le s}}}}
= f_s^*.
\end{equation}
If $X_s \times {\P}^1 \xrightarrow{F_s} M_s$ denotes the embedding
({\sl cf.} ~\ref{eqn:DNC}), then for $x \in CH^*_G\left(X_{\le s}\right)$,
we can write $x = j_{\le s}^*(y)$ by Theorem~\ref{thm:specialization}.
Then 
\[
\begin{array}{lllll}
{f_{s, \infty}^*} \circ {{{{\ov{Sp}}_{X, s}^{\le s}}}} \circ
 j_{\le s}^*(y) & = & {f_{s, \infty}^*} \circ i_{s, \le s}^*(y) &
= & g_{\infty, \le s}^* \circ F_s^*(y) = g_{0, \le s}^* \circ F_s^*(y) \\   
& = & f_s^* \circ j_{\le s}^*(y) & = & f_s^*(x),
\end{array}
\]
where the second inequality follows from Lemma~\ref{lem:torind}.
This proves ~\ref{eqn:special2} and the proposition.
\end{proof}
We need the following algebraic result before we prove the main
result of this section. Let $A$ be a $\Q$-algebra 
(not necessarily commutative). For any linear form 
$f(\underline{t}) = \stackrel{n}{\underset{i=1}{\Sigma}} a_i t_i$
in $A[t_1, \cdots , t_n]$ such that $a_i \in {\Q}$ for each $i$, let
$c(f)$ denote the vector $(a_1, \cdots , a_n) \in {\Q}^n$ consisting of
the coefficients of the form $f$. 
\begin{lem}\label{lem:elem1}
Let $A$ be as above and let
$S = \{f_1, \cdots , f_s\}$ be a set of linear forms in
$A[t_1, \cdots , t_n]$ such that the vectors $\{c(f_1), \cdots , c(f_s)\}$
are linearly independent in ${\Q}^n$. Let
\[
{\gamma}_j = \stackrel{d_j}{\underset{i=0}{\Sigma}} m^j_i f^i_j
\] such that $m_{d_j}^j \in {\Q}^*$ for $1 \le j \le s$, and
$m^j_{j'} \in Z(A)$ for all $j, j'$. Then one has
\[
\left({\gamma}_1 \cdots {\gamma}_s\right) = 
\stackrel{s}{\underset{j=1}{\bigcap}} \left({\gamma}_j\right)
\]
as ideals in $A[t_1, \cdots , t_n]$.
\end{lem}  
\begin{proof} Using a simple induction, it suffices to show that for 
$j \neq j'$, the relation ${\gamma}_j | q{\gamma}_{j'}$  
implies that ${\gamma}_j | q$. So we can assume $S = \{f_1, f_2\}$.
Extend $\{c(f_1), c(f_2)\}$ to a basis $B$ of ${\Q}^n$. Applying the
linear automorphism of $A[t_1, \cdots , t_n]$ given by the invertible
matrix $B$, we can assume that $f_j = t_j$ for $j = 1, 2$. 
Now the proof follows along the same lines as the proof of 
Lemma~4.9 of \cite{VV}. We skip the details.
\end{proof}
\begin{thm}\label{thm:decomposition*}
Let $G$ be a split diagonalizable group of dimension $n$ and let 
$X \in {\sV}^S_G$. The ring homomorphism
\[
CH^*_G\left(X \right) \longrightarrow 
\stackrel{n}{\underset{s=0}{\prod}} CH^*_G\left(X_s \right)
\]
is injective. Moreover, its image consists of the $n$-tuples
$\left({\alpha}_s\right)$ in the product with the property that
for each $s = 1, \cdots , n$, the pull-back of ${\alpha}_s \in
CH^*_G\left(X_s \right)$ in $CH^*_G\left(N_{s, s-1} \right)$ is same
as ${{\ov{Sp}}_{X, s}^{s-1}}\left({\alpha}_{s-1}\right) \in
CH^*_G\left(N_{s, s-1} \right)$.
In other words, there is a ring isomorphism
\[
CH^*_G\left(X \right) \stackrel{\cong}{\longrightarrow}
CH^*_G\left(X_n \right) {\underset{CH^*_G\left(N_{n,n-1} \right)}
{\times}} CH^*_G\left(X_{n-1} \right) 
{\underset{CH^*_G\left(N_{n-1,n-2} \right)}
{\times}} \cdots {\underset{CH^*_G\left(N_{1,0} \right)}
{\times}} CH^*_G\left(X_0 \right).
\]
\end{thm}
\begin{proof} We prove by the induction on the largest integer $s$ such that
$X_s \neq \emptyset$. 

If $s = 0$, there is nothing to prove. If $s > 0$, we have by induction
\begin{equation}\label{eqn:decomposition*1}
CH^*_G\left(X_{<s} \right) \stackrel{\cong}{\longrightarrow}
CH^*_G\left(X_{s-1} \right) {\underset{CH^*_G\left(N_{s-1,s-2} \right)}
{\times}} \cdots {\underset{CH^*_G\left(N_{1,0} \right)}
{\times}} CH^*_G\left(X_0 \right).
\end{equation}
Using this and Proposition~\ref{prop:decomposition1}, it suffices to
show that if ${\alpha}_s \in CH^*_G\left(X_s \right)$ and
if ${\alpha}_{<s} \in CH^*_G\left(X_{<s} \right)$ with the restriction
${\alpha}_{s-1} \in CH^*_G\left(X_{s-1} \right)$ are such that
${\alpha}_s \mapsto {\alpha}_s^0 \in CH^*_G\left(N_s^0 \right)$
and ${\alpha}_s \mapsto {\alpha}_{s, s-1} \in     
CH^*_G\left(N_{s-1,s-2} \right)$, then
\[
{{\ov{Sp}}_{X, s}^{\le s-1}}\left({\alpha}_{<s}\right) = 
{\alpha}_s^0 \ \ {\rm iff} \ \
{{\ov{Sp}}_{X, s}^{s-1}}\left({\alpha}_{s-1}\right) = 
{\alpha}_{s, s-1}.
\]
Using the commutativity of the left square in 
Theorem~\ref{thm:specialization},
this is reduced to showing that the restriction map
$CH^*_G\left(N_s^0 \right) \to CH^*_G\left(N_{s-1,s-2} \right)$
is injective. 

To prove this, we first use Proposition~\ref{prop:strata} to
assume that the toral component $T$ of the isotropy groups of the points 
of $X_s$ is fixed, and choose a splitting $G = D \times T$.
 
Now, following the proof of the analogous result for $K$-theory
({\sl cf.} \cite[Theorem~4.5]{VV}), we can write
\[
N_s = E = \stackrel{q}{\underset{i=1}{\bigoplus}} E_i \ \ {\rm and} \ \ 
N_{s,s-1} = {\underset{i}{\coprod}} E_i^0,
\]
where each $E_i$ is of the form ${\bigoplus}{E_{m_j}}_{{\chi}_i}
{\otimes} {m_j {\chi}_i}$, such that for $i \neq j$,
${\chi}_i$ and ${\chi}_j$ are linearly independent characters of $T$, 
and $E_i^0$ is embedded in $E$ by setting
all the other components equal to zero. Let $d_i = {\rm rank}(E_i)$.

Now we see from Proposition~\ref{prop:decomposition1} that
\[
{\rm Ker}\left(CH^*_G\left(X_{s} \right) \to 
CH^*_G\left(N_{s,s-1} \right)\right) =
{\underset{i}{\bigcap}} \left(c^G_{d_i}(E_i)\right) \ \ {\rm and} 
\]   
\[
{\rm Ker}\left(CH^*_G\left(X_{s} \right) \to 
CH^*_G\left(N_{s}^0 \right)\right) = \left(c^G_{d_s}(N_s)\right)
\ \ {\rm with} \ \ d_s = {\Sigma}d_i.
\]
Putting ${\gamma}_i = c^G_{d_i}(E_i)$ and ${\gamma} = 
c^G_{d_s}(N_s)$, we are then reduced to showing that  
\begin{equation}\label{eqn:elem2}
\left({\gamma}\right) = 
\left({\underset{i}{\prod}} {\gamma}_i \right)
= {\underset{i}{\bigcap}} \left({\gamma}_i\right)
\end{equation}
in $CH^*_D\left(X_{s} \right)[t_1, \cdots, t_s]$.

However, we have seen in the proof of Proposition~\ref{prop:rigiditysuff}
that each ${\gamma}_i$ is of the form
\[
{\gamma}_i = u_i^{d_i} + {\alpha}^i_{{d_i}-1} {u_i}^{{d_i}-1} + \cdots
+ {\alpha}^i_1 u_{i} + {\alpha}^i_0,
\]
where ${\alpha}^i_j \in CH^*_D\left(X_s, 0\right) \subset
Z\left(CH^*_D\left(X_s \right)\right)$ and 
$u_i = c^T_1\left(L_{{\chi}_i}\right)
= \stackrel{s}{\underset{j=1}{\Sigma}} b^i_jt_j \neq 0$ in
${\Q}[t_1, \cdots , t_s]$.
Moreover, the linear independence of ${\chi}_i$'s implies that the vectors
$\{c(u_1), \cdots , c(u_q)\}$ are linearly independent. We now apply
Lemma~\ref{lem:elem1} to conclude the proof of~\ref{eqn:elem2}
and hence the theorem.
\end{proof}

\section{Equivariant Chern Character And Riemann-Roch Maps}
The equivariant Riemann-Roch map for the Grothendieck group
of equivariant coherent sheaves on a variety has been constructed by
Edidin and Graham in \cite{ED1}. In this section, we construct these
maps for the higher equivariant $K$-theory and study their properties.
Following the techniques of Gillet ({\sl cf.} \cite{Gillet})
for the construction of the Riemann-Roch maps, 
we first define the equivariant Chern character map for the smooth varieties, 
and then use this Chern character to define the Riemann-Roch for all 
varieties. Recall from Section~1 that for any $G$-variety $X$,
$\widehat{G^G_i(X)}$ denotes the $I_G$-adic completion of the 
$R(G)$-module $G^G_i(X)$, where $I_G$ is the ideal of virtual 
representations of $G$ of rank zero in the representation ring $R(G)$.
We shall follow the notations of Section~4 for the various completions 
of $R(G)$ and $S(G)$-modules in the rest of this paper. In particular,
$\wt{K^G_i(X)}$ will denote the {\sl weak completion}
$K^G_i(X) {\otimes}_{R(G)} \widehat{R(G)}$. For a map $f: Y \to X$
in ${\sV}_G$, we make the convention in this paper that the induced
pull-back (or push-forward) map on $K$-theory will be denoted by
$f^*$ ($f_*$), and the map on the higher Chow groups will be denoted
by ${\bar{f}}^*$ (${\bar{f}}_*$).  
\begin{prop}\label{prop:Chern}
Let $G$ be a linear algebraic group and let $X \in {\sV}^S_G$. Then
for every $i \ge 0$, there is a Chern character map
\[
ch^G_X : K^G_i(X) \longrightarrow {\underset{j}{\prod}} 
CH^j_G\left(X,i\right) = \ov{CH^*_G\left(X,i\right)}
\]
with the following properties. \\
$(i) \ \ ch^G$ is a contravariant functor from ${\sV}^S_G$ to
${\rm Vec}_{\Q}$. \\
$(ii)$ For $\alpha \in K^G_0(X)$ and $x \in K^G_i(X)$, one has
$ch^G_X({\alpha} x) = ch^G_X({\alpha}) \cdot ch^G_X(x)$, where
the $\ov{CH^*_G\left(X,0\right)}$-module structure
on $\ov{CH^*_G\left(X,i\right)}$ is induced by the
intersection product. \\ 
$(iii) \ \ ch^G_X$ factors through the $I_G$-adic completion
\[
\widehat{K^G_i(X)} \stackrel{{\widehat{ch}}^G_X}{\longrightarrow} 
\ov{CH^*_G\left(X,i\right)}.
\]
$(iv)$ If $G$ acts freely on $X$, then $ch^G_X$ coincides with the
non-equivariant  Chern character map $ch_{X/G}$ under the identifications
$K^G_i(X) = K_i(X/G)$ and $CH^*_G\left(X,i\right) = 
CH^*\left(X/G,i\right)$ ({\sl cf.} Proposition~\ref{prop:EHCG}).
\end{prop}
\begin{proof} In order to define $ch^G_X$, it suffices to define the
component ${\left(ch^G_X(x)\right)}_j$ of $ch^G_X(x)$ in 
$CH^j_G\left(X,i\right)$ for every $j \ge 0$,
whenever $x \in K^G_i(X)$. So we fix $j \ge 0$ and choose a good pair
$(V,U)$ for $G$ corresponding to $j$. Let ${\pi}_{UX} : X \times U \to
X$ be the projection map, which is flat. We define  
${\left(ch^G_X(x)\right)}_j$ to be the $j$th component of the composite map
\[
K^G_X(X) \xrightarrow{{\pi}_{UX}^*} 
K^G_{X \times U} \left(X \times U\right) \xrightarrow{\cong} 
K_i(X_G) \xrightarrow{ch_{X_G}} CH^*\left(X_G,i\right),
\]
where $ch_{X_G} : K_i(X_G) \to CH^*\left(X_G,i\right)$ is the 
non-equivariant  Chern character map 
of Bloch and Gillet ({\sl cf.} \cite[Definition~2.34]{Gillet}, 
\cite[Section~7]{Bloch}).
The proof of the independence of the above definition of the choice
of the good pair $(V,U)$ is same as the proof of Proposition~3.1
in \cite{ED1}, using the fact that the non-equivariant  $K$-theory and
higher Chow groups satisfy the homotopy invariance property.

To prove the contravariance property of $ch^G$, we can again reduce to the
non-equivariant  case as above, where this is already known ({\sl cf.}
\cite{Gillet}, see also \cite[Theorem~1.10]{FW}). The same argument also
proves $(ii)$ as the non-equivariant  Chern character is known to be a ring
isomorphism ({\sl loc. cit.}).

To prove $(iii)$, we have $ch^G_k(I_G) \subset
{\underset{j \ge 1}{\prod}} CH^j_G\left(k,0\right)$ by \cite[3.1]{ED1}, 
and hence \\
$ch^G_X\left(I^n_G K^G_i(X)\right) \subset
{\underset{j \ge n}{\prod}} CH^j_G\left(X,i\right)$ by $(ii)$.
This gives a map of inverse systems 
\begin{equation}\label{eqn:pro}
\frac{K^G_i(X)}{I^n_G K^G_i(X)} \xrightarrow{ch^G_X}
\frac{{\underset{j}{\prod}} CH^j_G\left(X,i\right)}
{{\underset{j \ge n}{\prod}} CH^j_G\left(X,i\right)}.
\end{equation}
Taking the inverse limits both sides and using ~\ref{eqn:comp0}, we
get the required map in $(iii)$. The last assertion follows directly from the
construction of $ch^G_X$ above and by applying the contavariance property
of the non-equivariant  Chern character for the natural map $X_G \to X/G$.
\end{proof}
For a $G$-variety $X$ which is not necessarily smooth, we define
({\sl cf.} \cite[2.6]{ED2}) the equivariant operational Chow groups
$A^j_G(X)$ as operations 
$c(Y \to X) : CH^{j'}_G\left(Y, \cdot \right) \to
CH^{j+j'}_G\left(Y, \cdot \right)$ for any map $Y \to X$ in ${\sV}_G$,
satisfying the same properties as in the non-equivariant  case.
If $E$ is an equivariant vector bundle of rank $r$ on $X$, then
for any map $Y \xrightarrow{f} X$, we can apply 
Proposition~\ref{prop:EHCG}(v), to get the Chern class operations
$c^G_j\left(f^*(E)\right): CH^{j'}_G\left(Y, \cdot \right) \to
CH^{j+j'}_G\left(Y, \cdot \right)$ and hence elements in $A^j_G(X)$.
This defines the Chern character
\begin{equation}\label{eqn:ChernS}
{\ov{ch}}^G_X : K^G_0(X) \longrightarrow \stackrel{\infty}
{\underset{j = 0}{\prod}} A^j_G(X)
\end{equation} 
which is a ring homomorphism and the Todd class 
\[
Td^G_X(E) = \stackrel{r}{\underset{j = 1}{\prod}}
\frac{x_j}{1- e^{-x_j}}
\]
where $\{x_1, \cdots , x_r\}$ are the Chern roots of $E$
({\sl cf.} \cite[3.1]{ED1}). The Todd class of $E$ is an invertible 
element of $\stackrel{\infty}{\underset{j = 0}{\prod}} A^j_G(X)$.
Let ${\mathfrak m}_X$ denote the maximal ideal of the ring $K^G_0(X)$ 
consisting of the isomorphism classes of the vector bundles of 
virtual rank zero. The Chern character ${\ov{ch}}^G_X$ has the following 
important property that will be useful to us. 
\begin{prop}\label{prop:unit}
Let $G$ be a linear algebraic group and let $X \in {\sV}_G$. 
Then the map ${\ov{ch}}^G_X$ factors through the localization
\[
{K^G_0(X)}_{{\mathfrak m}_X} 
\xrightarrow{{\ov{ch}}^G_X} \stackrel{\infty}
{\underset{j = 0}{\prod}} A^j_G(X).
\]
\end{prop}
\begin{proof} If $X$ is smooth, this follows from 
\cite[Theorem~4.1, Theorem~6.1]{ED1}. Now suppose that $X$ is not 
necessarily smooth. 
We need to show that for any element $\alpha \in K^G_0(X)$ which
is not in ${\mathfrak m}_X$,
the image of ${\ov{ch}}^G_X(\alpha)$ is an invertible element of
$\stackrel{\infty}{\underset{j = 0}{\prod}} A^j_G(X)$.
Before we show this, we observe that $K^G_0(X)$
has a canonical decomposition 
\begin{equation}\label{eqn:K0}
K^G_0(X) = {\Q} \oplus {\mathfrak m}_X
\end{equation}
as a $\Q$-vector space. Given any $\alpha 
\in K^G_0(X)$, we can use Lemma~\ref{lem:lifting} below to get a
$G$-equivariant embedding $X \stackrel{i}{\inj} M$ such that $M$
is smooth and $\alpha = i^*(\beta)$ for some $\beta \in K^G_0(M)$.
Now the contravariance property of ${\ov{ch}}^G$ gives a commutative
diagram 
\[
\xymatrix{
K^G_0(M) \ar[r]^{{\ov{ch}}^G_M} \ar[d]_{i^*} & 
\stackrel{\infty}{\underset{j = 0}{\prod}} A^j_G(M) \ar[d]^{{\ov{i}}^*} \\
K^G_0(X) \ar[r]_{{\ov{ch}}^G_X} & 
\stackrel{\infty}{\underset{j = 0}{\prod}} A^j_G(X)}
\]
Since $i^*$ is a ring homomorphism which preserves the rank,
we see from ~\ref{eqn:K0} that if $\alpha \notin {\mathfrak m}_X$, then
$\beta$ is also not in ${\mathfrak m}_M$ and hence ${\ov{ch}}^G_M(\beta)$
is a unit in $\stackrel{\infty}{\underset{j = 0}{\prod}} A^j_G(M)$ 
by the smooth case.
Hence ${\ov{ch}}^G(\alpha) = {{\ov{i}}^*} \circ ch^G_M(\beta)$ is 
a unit in $\stackrel{\infty}{\underset{j = 0}{\prod}} A^j_G(X)$.
\end{proof}
\begin{lem}\label{lem:lifting}
For any $X \in {\sV}_G$ and $\alpha \in K^G_0(X)$, there is an equivariant
closed embedding $X \stackrel{i}{\inj} M$ with $M \in {\sV}^S_G$ and
a class $\beta \in K^G_0(M)$ such that $\alpha = i^*(\beta)$.
\end{lem}
\begin{proof} Since $K^G_0(X)$ is the group of isomorphism classes of
equivariant vector bundles, we can reduce the problem to $\alpha$ being
a finite collection of vector bundles. Then by the diagonal embedding of
$X$ into a product of smooth varieties, we reduce the problem to the case
when $\alpha$ is an equivariant vector bundle $E$ of rank $r$.

Since $G$ acts linearly on $X$, there is an equivariant closed
embedding $X \xrightarrow{f} Y$ with $Y$ smooth.
Put $E' = f_*(E)$. Then by \cite{Thomason4}, there is a $G$-equivariant
vector bundle $F$ on $Y$ and an equivariant surjection $F \surj E'$.
Let $M$ be the Grassman bundle of rank $r$ quotient bundles of
$F$ on $Y$ and let $M \xrightarrow{p} Y$ be the projection map.
Let $Q$ denote the universal quotient bundle of rank $r$ on $M$. It is now
easy to see that the action of $G$ on $Y$ and $F$ induces a natural
$G$-action on $M$ such that $p$ is a $G$-equivariant map and $Q$ is
a $G$-equivariant bundle on $M$. The universal property of the Grassman
bundle then implies that the map $f$ factors through a map 
$X \xrightarrow{i} M$ such that $E = {i}^*(Q)$. Clearly, $M$ is smooth
and $i$ is a closed embedding.
\end{proof}
The following property of the non-equivariant  Riemann-Roch map 
${\tau}_X: G_i(X) \to CH^*\left(X, i\right)$ ({\sl cf.} \cite[7.4]{Bloch})
will be used frequently in the construction of these maps in the
equivariant setting. 
\begin{lem}\label{lem:RR1}
Let $f: Y \to X$ be a morphism of varieties such that either \\
$(i) \ \ f$ is a vector bundle morphism $X$, or \\
$(ii) \ \ X$ and $Y$ are smooth and $f$ is an l.c.i. morphism. \\
Then one has ${\tau}_Y \circ f^* = {Td_Y(T_f)} \cdot \left({\bar{f}}^* \circ
{\tau}_X\right)$.
\end{lem} 
\begin{proof} First suppose $X$ is not necessarily smooth and $Y$ is
a vector bundle over $X$. Choose a finite open covering ${\sU}$ of
$X$ such that the restriction of $f$ on every $U \in {\sU}$ is trivial. 
Since the Riemann-Roch map commutes with the restriction to open covers
({\sl cf.} \cite[Theorem~4.1]{Gillet}), we get a commutative diagram
\[
\xymatrix{
G_i(X) \ar[r] \ar[d]_{{\tau}_X} & 
{\check{\H}}\left({\sU}, {\sK}\right) \ar[d]^{{\tau}_{\sU}} \\
CH^*\left(X,i\right) \ar[r] & 
{\check{\H}}\left({\sU}, {\sC}{\sH}(-, i)\right),}
\]
where the terms on the right are the $\check{\rm C}$ech cohomology of
the sheaves of $K$-groups and Chow groups for the cover $\sU$. The 
horizontal maps are isomorphisms by the Mayer-Vietoris property of 
$K$-theory and higher Chow groups. Thus it suffices to
prove the desired result for the open subsets in $\sU$. Hence we can
assume that $f$ is a trivial bundle, which can be further assumed to
be of rank one by induction. Thus we have $Y = X \times {\A}^1$ and
$f$ is the projection map. Put ${\wt{Y}} = X \times {\P}^1$ and
let $\wt{f} : \wt{Y} \to X$ be the projection and $j : Y \to \wt{Y}$
the inclusion. Now we apply the isomorphisms $G_i(Y) = G_i(X) {\otimes}
G_0({\P}^1)$, $CH^*\left(Y, i\right) = CH^*\left(X, i\right) {\otimes}
CH^*\left({\P}^1, 0\right)$ and the proof of the Riemann-Roch theorem
for the map ${\P}^1_X \to X$ ({\sl cf.} \cite[Lemma~4.4]{Gillet}) to get 
\[
\begin{array}{lllll}
{\tau}_{\wt{Y}} \circ {\wt{f}}^*(a) & = & {\tau}_{\wt{Y}}
\left(a \otimes 1\right) & =
& {\tau}_X(a) \otimes {\tau}_{{\P}^1}(1) \\
& = & {\tau}_X(a) \otimes Td\left({\P}^1\right) & = &
\left(1 \otimes Td\left({\P}^1\right)\right) \cdot 
\left({\tau}_X(a) \otimes 1\right) \\
& = & Td\left(T_{\wt{f}}\right) \cdot \left({\tau}_X(a) \otimes 1\right) &
= &  Td\left(T_{\wt{f}}\right) \cdot 
\left({\ov{{\wt{f}}}}^* \circ {\tau}_X (a)\right).
\end{array}
\]
Finally, we have
\[
{\tau}_Y \circ f^* = {\tau}_Y \circ j^* \circ {\wt{f}}^* = 
{\bar {j}}^* \circ {\tau}_{\wt{Y}} \circ {\wt{f}}^* 
= {\bar {j}}^* \circ {\ov{{\wt{f}}}}^* \circ {\tau}_X = 
{\bar {f}}^* \circ {\tau}_X.
\]
This proves the part $(i)$.

Now suppose $X$ and $Y$ are smooth and $f$ is an l.c.i. morphism.
Then we have
\[
\begin{array}{lllll}
{\tau}_Y \circ f^* & = & \left(ch_Y \circ f^*\right) \cdot Td(Y) 
& = &  \left({\bar{f}}^* \circ ch_X\right) \cdot Td(Y) \\
& = & \left({\bar{f}}^* \circ ch_X\right) \cdot 
\left[Td\left(f^*(X)\right) \cdot Td\left(T_f\right)\right] & = &   
{\bar{f}}^*\left(ch_X \cdot Td(X)\right) \cdot Td\left(T_f\right) \\
& = & \left({\bar{f}}^* \circ {\tau}_X\right) \cdot Td\left(T_f\right)
& & 
\end{array}
\]
\end{proof}
\begin{thm}\label{thm:RRoch}
Let $G$ be a linear algebraic group and let $X \in {\sV}_G$. Then
there is a Riemann-Roch map
\begin{equation}\label{eqn:RRoch0}
{\tau}^G_X : G^G_i(X) \to \ov{CH^*_G\left(X,i\right)} 
\end{equation}
which satisfies the following properties. \\
$(i) \ \ {\tau}^G$ is covariant for proper maps in ${\sV}_G$. \\
$(ii)$ For $\alpha \in K^G_0(X)$ and $x \in G^G_i(X)$, one has
${\tau}^G_X({\alpha} x) = {\ov{ch}}^G_X({\alpha}) \cdot {\tau}^G_X(x)$, with
respect to the $K^G_0(X)$-module structure on $G^G_i(X)$. \\
$(iii) \ \ {\tau}^G_X$ factors through the $I_G$-adic completion
\begin{equation}\label{eqn:RRoch1}
\widehat{G^G_i(X)} \stackrel{{\widehat{\tau}}^G_X}{\longrightarrow} 
\ov{CH^*_G\left(X,i\right)}.
\end{equation}
$(iv)$ For a morphism $f: Y \to X$ in ${\sV}_G$ such that 
$f$ is either an equivariant vector bundle morphism, or
it is an equivariant l.c.i. morphism in ${\sV}^S_G$, one has
${\tau}^G_Y \circ f^* = {Td^G_Y(T_f)} \cdot \left({\bar{f}}^* \circ
{\tau}^G_X\right)$. If $j: U \to X$ is a $G$-equivariant open immersion,
then ${\tau}^G_U \circ j^* = {\bar{j}}^* \circ {\tau}^G_X.$ \\
$(v)$ If $G$ acts freely on $X$, then ${\tau}^G_X$ coincides with the
non-equivariant  Riemann-Roch map ${\tau}_{X/G}$ under the identifications
$G^G_i(X) = G_i(X/G)$ and $CH^*_G\left(X,i\right) = 
CH^*\left(X/G,i\right)$.
\end{thm} 
\begin{proof} As in the case of Chern character, we need to define each
component of ${\tau}^G_X$ in the product ${\underset{j \ge 0}{\prod}} 
CH^j_G\left(X, i\right)$. So we fix $x \in G^G_i(X)$ and $j \ge 0$ and
choose a good pair $(V,U)$ for $G$ corresponding to $j$. Let
${\pi}_{UX} : X \times U \to X$ be the projection map, and consider
the diagram
\[
\xymatrix{
G^G_X(X) \ar[r]^{{{\pi}_{UX}^*}} \ar[d] &  
G^G_{X \times U} \left(X \times U\right) \ar[r]^{\cong} &
G_i(X_G) \ar[d]^{{\tau}_{X_G}} \\ 
{\ov{CH^*_G\left(X,i\right)}} & &  
CH^*\left(X_G,i\right) \ar@{_{(}->}[ll],}
\]
where ${\tau}_{X_G} : G_i(X/G) \to CH^*\left(X_G,i\right)$ is the
non-equivariant  Riemann-Roch map ({\sl cf.} \cite[7.4]{Bloch}) of
Bloch and Gillet. We define 
\begin{equation}\label{eqn:RR0}
{\tau}^G_X(x) = \frac{{\tau}_{X_G} \circ {{{\pi}_{UX}^*}}(x)}
{Td_{X_G}\left(E_V\right)},
\end{equation}
where $E_V$ is the vector bundle $X \stackrel{G}{\times}
\left(U \times V \right) \to X \stackrel{G}{\times} U = X_G$.

The proof that the above definition is independent of the choice of
the good pair $(V,U)$ follows exactly along the lines of the 
proof of \cite[Proposition~31.1]{ED1}. The only extra ingredient we need
in our case is Lemma~\ref{lem:RR1}.

The properties $(ii)$ and $(iii)$ follow directly from the definition
of ${\tau}^G$ and the analogous properties of the non-equivariant  
Riemann-Roch map ({\sl cf.} \cite[Theorem~4.1]{Gillet}). The proof
of property $(iii)$ is same as the proof of the corresponding 
property of the equivariant Chern character map in 
Proposition~\ref{prop:Chern}. The property $(iv)$ follows from the 
corresponding non-equivariant  result in Lemma~\ref{lem:RR1}.
The statement about open immersion follows from the analogous property
of the non-equivariant  Riemann-Roch ({\sl cf.} \cite[Theorem~4.1]{Gillet}).
The proof of property $(v)$ is the same as the proof of 
Proposition~\ref{prop:Chern} $(iv)$.
\end{proof} 
\begin{cor}\label{cor:local}
The Riemann-Roch map factors through the localization
\[
{\tau}^G_X : {G^G_i(X)}_{{\mathfrak m}_X} \to \ov{CH^*_G\left(X,i\right)}.
\]
\end{cor}
\begin{proof} We have seen that ${{\ov{ch}}}^G_X$ is a ring homomorphism.
Moreover, it follows from Theorem~\ref{thm:RRoch} $(ii)$ that 
${\tau}^G_X$ is $K^G_0(X)$-linear. Thus it suffices to show that any
element $\alpha \in K^G_0(X)$ which is not in ${{\mathfrak m}_X}$, acts
as a unit in $\ov{CH^*_G\left(X,i\right)}$. But this follows immediately
from Proposition~\ref{prop:unit}.
\end{proof}

\section{Riemann-Roch For Finite Group Actions}
In this section, we review the equivariant $K$-theory for finite group
actions and prove our Riemann-Roch isomorphism in this case. 
As remarked before, such an isomorphism can also be obtained using the
study of equivariant $K$-theory for finite group actions in \cite{Vistoli1}.
Let $G$ be a finite group and let
$X$ be an equidimensional $G$-variety. Let $p:X \to Y = X/G$ be the 
quotient for the $G$-action. Note that this quotient always exists and 
is quasi-projective because we are dealing with the linear actions on 
quasi-projetcive varieties. Let ${\sZ}^j\left(X, \cdot\right)$ 
denote the cycle complex of codimension $j$ cycles on $X$. Then $G$ acts
on ${\sZ}^j\left(X, \cdot\right)$ by letting it act on an irreducible 
cycle $\sigma \in {\sZ}^j\left(X, n\right)$ by
\[
\left(g, \sigma\right) \mapsto {\mu}_g^*(\sigma)
\]
and extending this linearly to all of $\sigma \in {\sZ}^j\left(X, n\right)$.
Note that $G$ here acts on $X \times {\Delta}^n$ via the diagonal action
for the trivial action of ${\Delta}^n$. In particular, the action of
$G$ on ${\sZ}^j\left(X, \cdot\right)$ commutes with the boundary maps,
and we get an action of $G$ on the complex ${\sZ}^j\left(X, \cdot\right)$.
Let ${\sZ}^j_G\left(X, \cdot\right)$ denote the subcomplex of invariant
cycles, and let $d^G$ denote the restriction of the differential $d$ of the
complex ${\sZ}^j\left(X, \cdot\right)$ to 
${\sZ}^j_G\left(X, \cdot\right)$. Put
\[
\wt{CH}^j_G\left(X, i\right) = 
H_i\left({\sZ}^j_G\left(X, \cdot\right)\right) \ \ {\rm for} \ \ i \ge 0.
\]
This gives a natural map
\begin{equation}\label{eqn:finite0}
\wt{CH}^j_G\left(X, \cdot\right) \stackrel{{\delta}^G_X}{\longrightarrow}
{\left(CH^j\left(X, \cdot\right)\right)}^G.
\end{equation}
\begin{lem}\label{lem:finite1}
There are canonical isomorphisms
\[
CH^j\left(Y, \cdot\right) \xrightarrow{p^*} 
\wt{CH}^j_G\left(X, \cdot\right) \xrightarrow{{\delta}^G_X}
{\left(CH^j\left(X, \cdot\right)\right)}^G.
\]
\end{lem}
\begin{proof} Let $V \subset Y \times {\Delta}^n$ be an irreducible cycle
intersecting all faces of $ Y\times {\Delta}^n$ properly and let
$\ov{V} = p^{-1}(V)$. Then $\ov{V} \to V$ is a finite morphism.
Moreover, the trivial action of $G$ on ${\Delta}^n$ implies that
$p^{-1}\left(V \cap Y \times {\Delta}^m\right) =
\ov{V} \cap X \times {\Delta}^m$ for all $m \le n$. Hence $H$ defines
a cycle on $ X \times {\Delta}^n$ which is clearly $G$-invariant.
This gives a natural map ${\sZ}^j\left(Y, \cdot\right) 
\xrightarrow{p^*} {\sZ}^j_G\left(X, \cdot\right)$.
To show that $p^*$ is an isomorphism of complexes, we need to give a unique
representation for every cycle $\sigma \in {\sZ}^j_G\left(X, \cdot\right)$
in terms of image of $p^*$. Now we can write $\sigma$ as
\[
\sigma = {\Sigma} a_j V_j - {\Sigma} a'_j V'_j, \ \ {\rm with} \ \
a_j, a'_j > 0.
\]
Then it is easy to see that both the sums on the right are $G$-invariant.
Hence can assume that $\sigma = {\Sigma} a_j V_j$ with $a_j > 0$.
Since $V_j$' are irreducible, $\sigma$ is $G$-invariant if and only if
it is of the form ${\Sigma} b_j H_j$, where each $H_j$ is of the form
${\underset{g \in G}{\Sigma}} g{\wt{H_j}}$ where ${\wt{H_j}}$ is
an irreducible cycle. Taking $W_j = p\left({\wt{H_j}}\right)$, it is then
easy to see that $\sigma = {\Sigma} a_j p^{-1}(W_j) = 
p^{-1}\left(\delta\right)$, where $W_j$'s (and hence $\delta$) are uniquely
determined by $\sigma$.     

We now show that ${\delta}^G_X$ is an isomorphism. 
The injectivity is proved by defining the trace map
${\sZ}^j\left(X, \cdot\right) \xrightarrow{tr} 
{\sZ}^j_G\left(X, \cdot\right)$
\begin{equation}\label{eqn:trace}
tr(x) = \frac{1}{|G|} {\underset{g \in G}{\Sigma}} {\mu}^*_g (x),
\end{equation}
and checking that $tr \circ {\delta}^G_X$ is the identity map on the
homology.
Thus we need to prove the surjectivity of ${\delta}^G_X$. So let
$x \in {\sZ}^j\left(X, n \right)$ be such that $d_n(x) = 0$ and
$x - gx \in {\rm Image}(d_{n+1}) \ \forall \ g \in G$ (since
$x \in {\left(CH^j\left(X, n\right)\right)}^G$). Putting $x_g =
x - gx$, we get 
\[
x =  \frac{1}{|G|} {\underset{g \in G}{\Sigma}} x = 
\frac{1}{|G|} {\underset{g \in G}{\Sigma}} gx +
\frac{1}{|G|} {\underset{g \in G}{\Sigma}} x_g =
tr(x) + y,\]
where $y = \frac{1}{|G|} {\underset{g \in G}{\Sigma}} x_g 
\in {\rm Image}(d_{n+1})$. Since $d_n(x) = 0$, it is easy to check that
$d^G_n\left(tr(x)\right) = 0$ and $x = {\delta}^G_X \left(tr(x)\right)$.
\end{proof}  
\begin{cor}\label{cor:finite2*}
There are canonical isomorphisms
\begin{equation}\label{eqn:finite3}
CH^*_G\left(X, \cdot\right) \to {\left(CH^j\left(X, \cdot\right)\right)}^G
\xleftarrow{p^*} CH^*\left(Y, \cdot\right).
\end{equation}
In particular, the natural maps 
\begin{equation}\label{eqn:finite4}
CH^*_G\left(X, \cdot\right) \to 
\wt{CH^*_G\left(X, \cdot\right)} 
\to \widehat{CH^*_G\left(X, \cdot\right)}
\to \ov{CH^*_G\left(X, \cdot\right)}
\end{equation}
are all isomorphisms.
\end{cor}
\begin{proof} To prove ~\ref{eqn:finite3}, we only need to prove the 
first isomorphism. 
So fix $j \ge 0$ and choose a good pair $(V,U)$ for the $G$-action
corresponding to $j$. Then we get the canonical isomorphisms
\[
CH^j_G\left(X, \cdot\right) \xrightarrow{\cong} 
CH^j\left(X_G, \cdot\right) \xrightarrow{\cong} 
{\left(CH^j\left(X \times U, \cdot\right)\right)}^G,
\]
where the second map is isomorphism by Lemma~\ref{lem:finite1}.
On the other hand we have natural $G$-equivariant maps
\[
CH^j\left(X, \cdot\right) \to 
CH^j\left(X \times V, \cdot\right) \to 
CH^j\left(X \times U, \cdot\right). 
\]
The first map is an isomorphism by the homotopy invariance and the 
second map is isomorphism by the localization sequence since $(V,U)$
is a good pair. Taking the $G$-invariants both sides, we conclude the
proof.

The isomorphism of the last and the composite of all maps in ~\ref{eqn:finite4}
follows from the isomorphisms in ~\ref{eqn:finite3}, which implies
that the $J_G$-adic filtration and the filtration by grading on 
$CH^j_G\left(X, \cdot\right)$ are nilpotent. The isomorphism of the
first map (hence the second map) now follows from 
Corollary~\ref{cor:completion2}.
\end{proof}
\begin{cor}\label{cor:finite2}
Let $G$ be a linear algebraic group over $k$ such that all its irreducible 
components are also defined over $k$. Let $G^0$ denote the identity
component of $G$ and let $H = G/{G^0}$. Let $X$ be a quasi-projective
$k$-variety with a $G$-action. Then the finite group $H$ naturally acts
on $CH^*_{G^0}\left(X, i \right)$ such that 
${\left(CH^*_{G^0}\left(X, i \right)\right)}^H \cong 
CH^*_{G}\left(X, i \right)$.
\end{cor}  
\begin{proof} We first observe that since all the components of $G$ are 
defined over $k$, $H$ is a finite constant group.
Fix $j \ge 0$ and choose a good pair $(V, U)$ for the
$G$-action on $X$. Then this is also a good pair for the $G^0$-action on
$X$. Moreover, $H$ naturally acts on the mixed quotient
$X \stackrel{G^0} {\times} U$ such that the corresponding quotient is
$X \stackrel{G} {\times} U$. Hence by Lemma~\ref{lem:finite1},
$H$ acts on $CH^j\left(X \stackrel{G^0} {\times} U, i \right)$
such that 
${\left(CH^j\left(X \stackrel{G^0} {\times} U, i \right)\right)}^H
\cong CH^j\left(X \stackrel{G} {\times} U, i \right)$. The corollary now
follows from the isomorphisms $CH^j_{G^0}\left(X, i \right) \cong
CH^j\left(X \stackrel{G^0} {\times} U, i \right)$ and
$CH^j_{G}\left(X, i \right) \cong
CH^j\left(X \stackrel{G} {\times} U, i \right)$.
\end{proof}
For a finite group $G$ acting on $X$, let $\sigma : G \times X \to X$
denote the action map. It is then clear that $\sigma$ is finite and
{\'e}tale. Put ${\alpha}_X = {\sigma}_*\left([{\sO}_{G \times X}]\right)
\in K^G_0(X)$. 
\begin{lem}\label{lem:etale}
For $G$ and $X$ as above, one has 
${\tau}^G_{G \times X} \circ  {\sigma}^* = 
{\bar{\sigma}}^* \circ {\tau}^G_X$.
\end{lem}
\begin{proof} We embed $X$ equivariantly into a smooth $G$-variety $M$.
By Corollary~\ref{cor:finite2}, we can choose a good pair $(V,U)$ for
the $G$-action such that $CH^*_G\left(X, \cdot\right) = 
CH^*\left(X_G, \cdot\right)$ and same for $M$.
Then as in the proof of Theorem~\ref{thm:RRoch} $(iv)$, we are reduced
to proving the non-equivariant  version of the lemma for the induced map
of quotients $G \stackrel{G}{\times} \left(X \times U\right) 
\xrightarrow{f} X  \stackrel{G}{\times} U$ which is also 
finite and {\'e}tale. Note that 
$G \stackrel{G}{\times} \left(X \times U\right) \cong X \times U$
and the map $f$ is the quotient map for the free action of $G$ on
$X \times U$. Letting $M \times U \xrightarrow{q} M \stackrel{G}{\times}
U$ denote the corresponding map for $M$, our problem is reduced to proving
that for a fiber diagram 
\begin{equation}\label{eqn:etale0}
\xymatrix{
Y \ar[r]^{i} \ar[d]_{p} & Z \ar[d]^{q} \\
S \ar[r]_{j} & W}
\end{equation}
of varieties such that the bottom arrow is the closed embedding
of quotients for a free action of a finite group $G$ on the top arrow, which
is an equivariant closed embedding, one has
${\tau}_Y \circ p^* = {\bar{p}}^* \circ {\tau}_S$.

Before we prove this, we recall the construction
of the Riemann-Roch map from \cite{Bloch}. Given any variety $X$, we
first embed $X \stackrel{i}{\inj} M$ with $M$ smooth. Let $U$ denote
the complement of $X$ in $M$. Then there is a diagram of homotopy
fibration of spectra
\begin{equation}\label{eqn:RR21}
\xymatrix{
G(X) \ar[r]^{i_*} & G(M) \ar[r]^{j^*} \ar[d]_{ch_M} & 
G(U) \ar[d]^{ch_U} \\
{\sH}_X  \ar[r]_{i_*} & {\sH}_M  \ar[r]_{j^*} &
{\sH}_U,}
\end{equation}    
where the terms on the bottom arrow are spectra whose homotopy groups 
give the higher Chow groups ({\sl cf.} \cite[Theorem~1.10]{FW}).
Since this is a diagram of homotopy fibrations and the right square
commutes, we get an induced Chern character with support
$ch^M_X : G(X) \to {\sH}_X$. Then one defines the Riemann-Roch map
for $X$ on the homotopy groups by ${\tau}_X (x) = 
Td(i^*T_M) \cdot ch^M_X (x)$. This is independent of $M$.

Coming back to the proof of ~\ref{eqn:etale0}, we first claim
that $p_* \circ p^* = |G| id$. But the projection formula
reduces this to showing that $p_*(1) = |G|$.
By the Riemann-Roch theorem ({\sl cf.} \cite{Bloch}), it suffices
to show that ${\bar{p}}_* \circ {\tau}_Y (1) = |G| {\tau}_S (1)$.
Using the above construction of the Riemann-Roch map, we have 
\[
\begin{array}{lll}
{\bar{p}}_* \circ {\tau}_Y (1) & =  & 
{\bar{p}}_* \left[ch_Y(1) {\bar{i}}^*\left(Td(Z)\right)\right] \\
& = & {\bar{p}}_* \circ {\bar{i}}^*\left(Td(Z)\right) \\
& = & {\bar{p}}_* \circ {\bar{i}}^* \circ  {\bar{q}}^*\left(Td(W)\right) \\ 
& = & {\bar{p}}_* \circ {\bar{p}}^* \circ  {\bar{j}}^*\left(Td(W)\right) \\ 
& = & |G| {\tau}_S (1),
\end{array}
\]
where the third equality holds because $q$ is a finite and {\'e}tale
map of smooth varieties. This proves the claim.

Now for any $a \in G_i(S)$, we have
\[
\begin{array}{lll}
{\bar{p}}^* \circ {\tau}_S (a) & = & 
{\bar{p}}^* \circ {\tau}_S \left({\frac{1}{|G|}} p_* \circ p^*(a)\right) \\
& = & {\frac{1}{|G|}} \left({\bar{p}}^* \circ {\tau}_S \circ
p_* \circ p^*(a)\right) \\ 
& = & {\frac{1}{|G|}} \left({\bar{p}}^* \circ {\bar{p}}_* \circ
{\tau}_Y \circ  p^*(a)\right) \\ 
& = & tr \left({\tau}_Y \circ  p^*(a)\right) \\ 
& = & {\tau}_Y \circ  p^*\left(tr(a)\right) \\
& = & {\tau}_Y \circ  p^*(a).
\end{array}
\]
Here, the first equality follows from the above claim. The fourth
equality occurs because of the fact that for any irreducible cycle
$V \in {\sZ}^* \left(Y, n \right)$, one has ${\bar{p}}^* \circ {\bar{p}}_*
(V) = {\underset{g \in G}{\Sigma}} {\mu}^*(V) = |G| tr(V)$. Finally,
the last equality holds because $G$ acts trivially on $S$.
This completes the proof of the lemma.
\end{proof}
\begin{thm}\label{thm:finiteI}
For $X \in {\sV}_G$ as above, the Riemann-Roch map
$G^G_i(X)  \xrightarrow{{\tau}^G_X} CH^*_G\left(X,i\right)$
is surjective. Moreover, for any $x \in G^G_i(X)$, one has
${\tau}^G_X (x) = 0$ if and only if ${\alpha}_X \cdot x = 0$.
\end{thm} 
\begin{proof} Before we begin the proof, we note that ${\tau}^G_X$
maps $G^G_i(X)$ into $\ov{CH^*_G\left(X,i\right)}$, but we can now
replace the latter by $CH^*_G\left(X,i\right)$ using  
Corollary~\ref{cor:finite2}.

To show the surjectivity of ${\tau}^G_X$, we consider the following
diagram
\[
\xymatrix{
G^G_i(G \times X) \ar[d]_{{\sigma}_*} \ar[r]^{{\tau}_{G \times X}} &
CH^*_G\left(G \times X, i\right) \ar[d]^{{\ov{\sigma}}_*} \\
G^G_i(X) \ar[r]_{{\tau}^G_X} & CH^*_G\left(X,i\right)}
\]
which commutes by Theorem~\ref{thm:RRoch}. 
Since $G$ acts freely on $G \times X$, we can apply 
Theorem~\ref{thm:RRoch} and the non-equivariant   Riemann-Roch ({\sl cf.}
\cite[Theorem~9.1]{Bloch} to conclude that the top horizontal arrow
is an isomorphism. Moreover, as $\sigma$ is finite and {\'e}tale,
we have ${\ov{\sigma}}_* \circ {\ov{\sigma}}^* = |G| id$. In particular,
${\ov{\sigma}}_*$ is surjective and hence so is ${\tau}^G_X$.

To prove the remaining part of the theorem,
we first show that if $G$ acts freely on $X$, then ${\alpha}_X$ is an
invertible element of $K^G_0(X)$. Let $X \xrightarrow{p} Y = X/G$ be the
quotient map. Then $K^G_0(X) \cong K_0(Y)$ and ${\alpha}_X =
p_*\left([{\sO}_X]\right)$. Now the assertion follows from the fact that
$K_0(Y)$ has a canonical decomposition $K_0(Y) = \wt{K_0(Y)} \oplus {\Q}$,
where $\wt{K_0(Y)}$ is nilpotent and ${\alpha}_X$ has positive rank.

Now suppose $x \in G^G_i(X)$ is such that ${\tau}^G_X (x) = 0$.
Then by Lemma~\ref{lem:etale}, we get ${\tau}^G_{G \times X} \circ
{\sigma}^*(x) = 0$, which in turn implies that ${\sigma}^*(x) = 0$
as ${\tau}^G_{G \times X}$ is an isomorphism. Composing this with
${\sigma}_*$ and applying the projection formula, we get 
${\alpha}_X \cdot x = 0$. 

Conversely, ${\alpha}_X \cdot x = 0$ implies that 
${\sigma}^* \left({\alpha}_X\right) \cdot {\sigma}^* 
(x) = {\sigma}^* \left({\alpha}_X \cdot x\right) = 0$.
But we have just shown that ${\sigma}^* \left({\alpha}_X\right) = 
{\alpha}_{G \times X}$  is a unit in $K^G_0(G \times X)$.
So we must have ${\sigma}^*(x) = 0$. Composing this with 
${\tau}^G_{G \times X}$ and applying Lemma~\ref{lem:etale}, we 
conclude that ${\ov{\sigma}}^* \circ {\tau}^G_X (x) = 0$. This in turn
gives ${\ov{\sigma}}_* \circ {\ov{\sigma}}^* \circ {\tau}^G_X (x) = 0$
and hence $|G|{\tau}^G_X (x) = 0$.
\end{proof}
\begin{cor}\label{cor:finiteII}
Let $G$ be a finite group and let $X \in {\sV}_G$. Then \\
$(i)$ The maps
\[
\wt{G^G_i(X)} \to \widehat{G^G_i(X)} \to  
{\widehat{G^G_i(X)}}_{{\mathfrak m}_X} \leftarrow 
{G^G_i(X)}_{{\mathfrak m}_X}
\]
are all isomorphisms. \\
$(ii)$ The Riemann-Roch map ${\tau}^G_X$ induces isomorphism
\[
\wt{G^G_i(X)} \stackrel{{\wt{\tau}}^G_X}{\underset{\cong}{\longrightarrow}}
CH^*_G\left(X,i\right).
\]
$(iii)$ If $X$ is smooth, then the Chern character map in
Proposition~\ref{prop:Chern} induces an isomorphism
\[
\wt{K^G_i(X)} \stackrel{{\wt{ch}}^G_X}{\underset{\cong}{\longrightarrow}}
CH^*_G\left(X,i\right).
\]
\end{cor}
\begin{proof} First we note that ${\alpha}_X$ is a unit in 
${K^G_0(X)}_{{\mathfrak m}_X}$ and so it
acts as a unit on ${G^G_i(X)}_{{\mathfrak m}_X}$. 
It then follows from Corollary~\ref{cor:local} and Theorem~\ref{thm:finiteI}
that ${\tau}^G_X$ induces an isomorphism
\begin{equation}\label{eqn:finitelocal}
{G^G_i(X)}_{{\mathfrak m}_X} \xrightarrow{\cong}
CH^*_G\left(X,i\right).
\end{equation}
To prove the isomorphism of the second and the third
arrows in $(i)$, it suffices to show that the ${{\mathfrak m}_X}$-adic
filtration on $G^G_i(X)$ is nilpotent.
But we have seen in the proof of Proposition~\ref{prop:unit} that
${\ov{ch}}^G_X\left({{\mathfrak m}^n_X}\right) \subset
{\underset{j \ge n}{\prod}}A^j_G(X)$ for every $n \ge 1$.
We conclude from Corollary~\ref{cor:finite2} that 
${{\mathfrak m}^n_X} CH^*_G\left(X, i \right) = 0$ for $n \gg 0$ and 
hence from ~\ref{eqn:finitelocal}, we get
${{\mathfrak m}^n_X}{G^G_i(X)}_{{\mathfrak m}_X} = 0$ for
$n \gg 0$.  To prove the isomorphism for the first arrow in $(i)$,
we note from \cite[Theorem~1]{Vistoli1} that 
$R(G) = I_G \times \left({R(G)}/{I_G}\right)$, where $I_G$ is a finite 
product of fields. This implies that $I^n_G = I_G$ for all $n$ and hence
$\wt{G^G_i(X)} \cong {{G^G_i(X)}}/{I_G} \cong
\widehat{G^G_i(X)}$. The part $(ii)$ now follows directly from
the part $(i)$ and ~\ref{eqn:finitelocal}.

If $X$ is smooth, then Corollary~\ref{cor:finite2} implies that
there is a good pair $(V,U)$ for the $G$-action such that
$CH^*_G\left(X, \cdot\right) = CH^*\left(X_G, \cdot\right)$. Now we see
from the the construction of the non-equivariant  Riemann-Roch
map in the proof of Lemma~\ref{lem:etale}, and its equivariant
construction in ~\ref{eqn:RR0} that
\begin{equation}\label{eqn:RRC}
{\tau}^G_X = \frac{Td(X_G)}{Td(E_V)}  ch^G_X.
\end{equation}
Now the isomorphism of ${\wt{ch}}^G_X$ follows from $(ii)$ and the fact that
the Todd classes are invertible elements in $CH^*\left(X_G, 0\right)$.
\end{proof}

\section{Riemann-Roch Isomorphism For Action With Finite Stabilizers}
As was mentioned in the beginning, our approach to the proof of 
Theorem~\ref{thm:main*} is to eventually reduce the problem to the
case when the underlying group acts either with finite stabilizers or
with a constant dimension of stabilizers. In the case of action with finite 
stabilizers, we shall in fact prove the following stronger version of 
Theorem~\ref{thm:main*}, where we do not assume that the given variety is 
smooth. The equivariant $K$-theory of smooth varieties for actions with 
finite stabilizers was studied before in \cite{Toen} and \cite{VV0}. The main 
result of \cite{VV0} is a decomposition of the equivariant $K$-theory from 
which one can deduce the Riemann-Roch isomorphism. Our approach in this case
is focussed more on directly constructing a Riemann-Roch isomorphism
between the equivariant $K$-theory and higher Chow groups, without relying
on the results of To{\"e}n and Vezzosi-Vistoli. Moreover, we prove this in a 
form which will be most suitable for us in the proof of our main result of this
paper.
\begin{thm}\label{thm:singular}
Let $G$ be a linear algebraic group acting on a (possibly singular) variety 
$X$. If $G$ acts with finite stabilizers, then the Riemann-Roch map of 
~\ref{eqn:RRoch1} gives rise to a commutative diagram
\begin{equation}\label{eqn:singular*}
\xymatrix{
 {G^G_i(X){\otimes}_{R(G)} {\widehat{R(G)}}}
\ar[r]^{{\wt{\tau}}^G_X} \ar[d]_{u^G_X} &
{CH^*_G(X,i){\otimes}_{S(G)} {\widehat{S(G)}}}
\ar[d]^{{\ov{u}}^G_X} \\
{\widehat{G^G_i(X)}} \ar[r]_{{\widehat{\tau}}^G_X} &
{\widehat{CH^*_G(X,i)}},}
\end{equation}
where all the maps are isomorphisms.
If $G$ is a diagonalizable group, then the horizontal maps
are isomorphisms and the vertical maps injective,
even when it acts on $X$ with a fixed dimension of stabilizers. 
\end{thm}
In this section, we study the Chern character and the
Riemann-Roch maps for the torus action with finite stabilizers, and  
prove Theorem~\ref{thm:singular} in this special case. 
We begin with the following general result. 
\begin{prop}\label{prop:FST}
Let $G$ be a linear algebraic group acting on a variety $X$ with 
finite stabilizers. Then for any $i \ge 0$, one has
$CH^j_G\left(X,i\right) = 0$ for $j \gg 0$. In particular,
the natural maps
\[ 
CH^*_G\left(X, i\right) \to 
\wt{CH^*_G\left(X, i\right)} 
\to \widehat{CH^*_G\left(X, i\right)}
\to \ov{CH^*_G\left(X, i\right)}
\]
are all isomorphisms.
\end{prop}
\begin{proof} We only need to prove the first assertion. The remaining
part then follows exactly in the same way as the proof of 
Corollary~\ref{cor:finite2*}. We embed $G$ as a closed subgroup of $GL_n$
for some $n$. Then $GL_n$ acts naturally on the quotient 
$GL_n \stackrel{G} {\times} X$, and this action is with finite stabilizers
if and only if $G$-action on $X$ is with finite stabilizers
({\sl cf.} \cite[Section~5]{ED1}).
By Corollary~\ref{cor:Morita}, we can assume that $G$ is the general
linear group and hence a connected and split reductive group.
We can now use Corollary~\ref{cor:reduction} to reduce to the case when
$G$ is a split torus $T$. 

Now, by Thomason's generic slice Theorem
({\sl cf.} \cite[Proposition~4.10]{Thomason5}), there exists a non-empty
$T$-invariant open subset $U \subset X$ and a diagonalizable subgroup
$T_1 \subset T$ with quotient $T/{T_1} = T_2$ such that $T$ acts on
$U$ via $T_2$, which in turn acts freely on $U$ such that there is a
$T$-equivariant isomorphism 
\[
U \xrightarrow{\cong} U/T \times T_2 \cong U/T \stackrel{T_1}{\times} T.
\]
Since $T$ acts on $U$ with finite stabilizers, $T_1$ must be a finite
diagonalizable group. Now we apply the Morita isomorphism of 
Corollary~\ref{cor:Morita} to get
\[
CH^*_T\left(U, i \right) \cong 
CH^*_T\left(U/T \stackrel{T_1}{\times} T, i \right) \cong
CH^*_{T_1}\left(U/T, i \right).
\]
Since $T_1$ acts trivially on $U/T$, we now apply Theorem~\ref{thm:trivial}
to get
\begin{equation}\label{eqn:MoritaC}
CH^*_T\left(U, i \right) \cong CH^*\left(U/T, i\right) {\otimes}_{\Q}
S(T_1).
\end{equation}
Since $CH^j\left(U/T, i\right) = 0$ for $j \gg 0$ and since $T_1$ is
a finite group, we must have $CH^j_T\left(U, i \right) = 0$
for $j \gg 0$. Now we apply the Noetherian induction and the
localization sequence ({\sl cf.} Proposition~\ref{prop:EHCG})
to conclude that $CH^j_T\left(X, i \right) = 0$ for $j \gg 0$.
\end{proof}  
\begin{lem}\label{lem:FST1} 
Let $G$ and $X$ be as in Proposition~\ref{prop:FST}. Then the natural map
\[
{G^G_i(X)}_{I_G} \to \widehat{G^G_i(X)}
\]
is an isomorphism.
\end{lem}
\begin{proof} This is a straightforward generalization of 
\cite[Proposition~5.1]{ED1} to higher $K$-theory and we only 
give the main steps.  We only need to show that
\[
I^n_G {G^G_i(X)}_{I_G} = 0 \ \ {\rm for} \ \ n \gg 0.
\]
By embedding $G$ into $GL_n$ and using the Morita equivalence
({\sl cf.} \cite[Theorem~1.10]{Thomason1})
\begin{equation}\label{eqn:MoritaK}
G^{GL_n}_i\left(GL_n \stackrel{G}{\times} X \right) \xrightarrow{\cong}
G^G_i(X),
\end{equation}
we can reduce to the $G = GL_n$ case. In this case, we use the 
Merkurjev's isomorphism ({\sl cf.} \cite[Proposition~8.1]{Merkurjev})
\[
G^G_i(X) {\otimes}_{R(G)} R(T) \xrightarrow{\cong} G^T_i(X)
\]
to reduce to the case of $G = T$ a torus.

Now, we can use Thomason's generic slice Theorem
as above  and then \cite[Lemma~5.6]{Thomason3} to get a non-empty
$T$-invariant open subset $U \subset X$ and a finite subgroup
$T_1 \subset T$ such that there is a $T$-equivariant isomorphism 
$U \xrightarrow{\cong} U/T \stackrel{T_1}{\times} T$ and
\[
G^T_i(U) \xrightarrow{\cong} G_i(U/T) {\otimes}_{\Q} R(T_1).
\]
Since $T_1$ is finite, we have $I_{T_1} R(T_1) = 0$ by
\cite[Theorem~1]{Vistoli1}, and hence
$I_T G^T_i(U) = 0$. Now the corresponding result for $G^T_i(X)$ follows
from the Noetherian induction and the localization
sequence in the equivariant $K$-theory 
({\sl cf.} \cite[Theorem~1.5]{Thomason1}).
\end{proof}
\begin{thm}\label{thm:main-torus}
Let $G$ be a diagonalizable group acting on a variety $X$ with finite
stabilizers. Then 
${\tau}^G_X$ induces the following commutative diagram where
all maps are isomorphisms.
\begin{equation}\label{eqn:FST*}
\xymatrix{
{\wt{G^G_i(X)}} \ar[d]_{u^G_X} \ar[r]^{{\wt{\tau}}^G_X} &
{\wt{CH^*_G\left(X, i \right)}} \ar[d]^{{\ov{u}}^G_X} \\
{\widehat{G^G_i(X)}} \ar[r]_{{\widehat{\tau}}^G_X} &
{\widehat{CH^*_G\left(X, i \right)}}}
\end{equation}
If $X$ is smooth, then the Chern character map of 
Proposition~\ref{prop:Chern} induces an isomorphism
\[
{\wt{K^G_i(X)}} \xrightarrow{{\wt{ch}}^G_X} {\wt{CH^*_G\left(X, i \right)}}.
\]
\end{thm} 
\begin{proof} We see from Proposition~\ref{prop:FST} that the map
${\widehat{G^G_i(X)}} \xrightarrow{{{\widehat{\tau}}^G_X}} 
\ov{CH^*_G\left(X, i \right)}$ actually lifts to a map
${\widehat{G^G_i(X)}} \xrightarrow{{{\widehat{\tau}}^G_X}} 
{\widehat{CH^*_G\left(X, i \right)}}$.
The proposition also shows that the map ${{\ov{u}}^G_X}$ is an isomorphism.
This automatically gives a natural
map ${{\wt{\tau}}^G_X}$ as in the above diagram, which is functorial
in $X$ and which makes the diagram commute.

As in the proof of Lemma~\ref{lem:FST1}, we can choose a non-empty
$G$-invariant open subset $U \subset X$ and a finite subgroup
$T_1 \subset G$ such that
$G^G_i(U) \xrightarrow{\cong} G_i(U/G) {\otimes}_{\Q} R(T_1)$,
which in particular gives
\[
G^G_i(U) {\otimes}_{R(G)} {\widehat{R(G)}} \xrightarrow{\cong}
G_i(U/G) {\otimes}_{\Q} {\left(R(T_1) {\otimes}_{R(G)} 
{\widehat{R(G)}}\right)}.
\]
Since the natural map $R(G) \to R(T_1)$ is finite, we have
\[
R(T_1) {\otimes}_{R(G)} {\widehat{R(G)}} \xrightarrow{\cong}
{\widehat{R(T_1)}}_G \xrightarrow{\cong} {\widehat{R(T_1)}},
\]
where the last equality follows from \cite[Corollary~6.1]{ED1}.
Thus we get an isomorphism 
\begin{equation}\label{eqn:FST-torus1}
G^G_i(U) {\otimes}_{R(G)} {\widehat{R(G)}} \xrightarrow{\cong}
G_i(U/G) {\otimes}_{\Q} {\widehat{R(T_1)}}.
\end{equation}
Similarly, we use ~\ref{eqn:MoritaC} and \cite[Theorem~6.1 (b)]{ED1}
to get an isomorphism
\begin{equation}\label{eqn:FST-torus2}
CH^*_G\left(U,i\right) {\otimes}_{S(G)} {\widehat{S(G)}} \xrightarrow{\cong}
CH^*\left(U/G,i\right) {\otimes}_{\Q} {\widehat{S(T_1)}}
\end{equation}
and a commutative diagram
\begin{equation}\label{eqn:FST-torus3}
\xymatrix{
{G^G_i(U) {\otimes}_{R(G)} {\widehat{R(G)}}} \ar[d]_{{\wt{\tau}}^G_U}
\ar[r]^{\cong} & {G_i(U/G) {\otimes}_{\Q} {\widehat{R(T_1)}}}
\ar[d]^{{\tau}_{U/G} \otimes {\tau}^G_k} \\
{CH^*_G\left(U,i\right) {\otimes}_{S(G)} {\widehat{S(G)}}} \ar[r]_{\cong}
& {CH^*\left(U/G,i\right) {\otimes}_{\Q} {\widehat{S(T_1)}}}.}
\end{equation}
The map ${\tau}_{U/G}$ is an isomorphism by the non-equivariant   Riemann-Roch
({\sl cf.} \cite[Theorem~9.1]{Bloch}) and ${\tau}^G_k$ is an isomorphism
by Theorem~\ref{thm:ED}. In particular, the right vertical map is an
isomorphism and hence so is the left vertical map.

To prove the isomorphism of ${{\wt{\tau}}^G_X}$, we let $Y = X - U$ and
consider the diagram of localization sequences 
\[
\xymatrix@C.7pc{
{\wt{G^G_{i+1}(U)}} \ar[r] \ar[d]^{{\wt{\tau}}^G_U} &
{\wt{G^G_i(Y)}} \ar[r] \ar[d]^{{\wt{\tau}}^G_Y} &
{\wt{G^G_i(X)}} \ar[r] \ar[d]^{{\wt{\tau}}^G_X} &
{\wt{G^G_i(U)}} \ar[r] \ar[d]^{{\wt{\tau}}^G_U} &
{\wt{G^G_{i-1}(Y)}} \ar[d]^{{\wt{\tau}}^G_Y} \\
{\wt{CH^*_G\left(U,i+1 \right)}} \ar[r] &
{\wt{CH^*_G\left(Y,i \right)}} \ar[r] & 
{\wt{CH^*_G\left(X,i \right)}} \ar[r] &
{\wt{CH^*_G\left(U,i \right)}} \ar[r] &
{\wt{CH^*_G\left(Y,i-1 \right)}}} 
\]
which commutes by Theorem~\ref{thm:RRoch}.
The rows are exact since ${\widehat{R(G)}}$ and ${\widehat{S(G)}}$ are
flat over $R(G)$ and $S(G)$ respectively.
The map ${{\wt{\tau}}^G_Y}$ is an isomorphism by the Noetherian induction,
and we have shown above that ${{\wt{\tau}}^G_U}$ is an isomorphism.
We conclude from 5-lemma that ${{\wt{\tau}}^G_X}$ is also an isomorphism.

To show that $u^G_X$ is an isomorphism, the natural
maps ${G^G_i(X)}_{I_G} \to \wt{G^G_i(X)} \to \widehat{G^G_i(X)}$
and Lemma~\ref{lem:FST1} imply that $u^G_X$ is surjective. It is injective
because we have just shown that ${\widehat{\tau}}^G_X \circ u^G_X =
{\ov{u}}^G_X \circ {\wt{\tau}}^G_X$ is an isomorphism. 
Finally, the isomorphism of ${\widehat{\tau}}^G_X$ now follows since all 
other maps in ~\ref{eqn:FST*} are isomorphisms.

If $X$ is smooth, we can use Proposition~\ref{prop:FST} to choose a good
pair $(V,U)$ for the $G$-action such that $CH^*_G\left(X, i \right)
\cong CH^*\left(X_G, i \right)$. The rest of the argument is exactly the same
as the proof of the corresponding result for finite group actions in
Corollary~\ref{cor:finiteII}.
\end{proof}

\section{Proofs Of Theorem~\ref{thm:main*} For Diagonalizable 
Groups} 
In this section, we prove Theorem~\ref{thm:main*} for the action of
diagonalizable groups on smooth varieties. 
We first deal with the case when a diagonalizable group $G$ acts on a 
variety (possibly singular) $X$ such that the stabilizers have a constant 
dimension.
\begin{prop}\label{prop:fixedS}
Let $G$ be a split diagonalizable group and let $X \in {\sV}_G$ be such that
the stabilizers of all points of $X$ have a constant dimension. 
Then Theorem~\ref{thm:main*} and Theorem~\ref{thm:singular} hold.
\end{prop}
\begin{proof} Using Proposition~\ref{prop:strata} and
Remark~\ref{remk:stratasingular}, we can assume that
there exists a subtorus $T \subset G$ such that $T$ acts trivially
on $X$ and $H = G/T$ acts on $X$ with finite stabilizers. 
We have then
\begin{equation}\label{eqn:fixedS2}
\begin{array}{lll}
{\wt{G^G_i(X)}} & \cong & G^G_i(X) {\otimes}_{R(G)} \widehat{R(G)} \\
& \cong & \left(G^H_i(X) {\otimes}_{\Q} R(T)\right) {\otimes}_{R(G)}
\widehat{R(G)} \\  
& \cong & G^H_i(X) {\otimes}_{R(H)} 
\left(\left(R(H) {\otimes}_{\Q} R(T)\right) {\otimes}_{R(G)}
\left(\widehat{R(H)} {\otimes}_{\Q} \widehat{R(T)}\right)\right) \\
& \cong & G^H_i(X) {\otimes}_{R(H)} \left(R(G) {\otimes}_{R(G)}
\left(\widehat{R(H)} {\otimes}_{\Q} \widehat{R(T)}\right)\right) \\
& \cong & \left(G^H_i(X) {\otimes}_{R(H)} \widehat{R(H)}\right)
{\otimes}_{\Q} \widehat{R(T)} \\
& \cong & {\wt{G^H_i(X)}} {\otimes}_{\Q} \widehat{R(T)},
\end{array}
\end{equation}
where the second isomorphism and the isomorphism $R(G) \cong
R(T) {\otimes}_{\Q} R(H)$ follow from \cite[Lemma~5.6]{Thomason3}.
We can similarly use Theorem~\ref{thm:trivial} to
get
\begin{equation}\label{eqn:fixedS02}
\wt{CH^*_G\left(X, i \right)} \xrightarrow{\cong} 
\wt{CH^*_H\left(X, i \right)} {\otimes}_{\Q} \widehat{S(T)}.
\end{equation} 

First we assume $X \in {\sV}^S_G$ and prove Theorem~\ref{thm:main*}.
Since the construction of the equivariant Chern character 
({\sl cf.} Proposition~\ref{prop:Chern}) at the individual factors
of $\ov{CH^*_G\left(X,i\right)}$ is given by the non-equivariant  
Chern character on a mixed quotient, we have a commutative diagram
\begin{equation}\label{eqn:fixedS1}
\xymatrix{
{{\wt{K^H_i(X)}} {\otimes}_{\Q} \widehat{R(T)}} \ar[r]^{\hspace*{.5cm}\cong} 
\ar[d]_{{\wt{ch}}^H_X \otimes {\wt{ch}}^T_k} &
{{\wt{K^G_i(X)}}} \ar[r]^{u^G_X} \ar@{.>}[d] & 
{{\widehat{K^G_i(X)}}} \ar[d]^{{\widehat{ch}}^G_X} \\
{\wt{CH^*_H\left(X, i \right)} {\otimes}_{\Q} \widehat{S(T)}} 
\ar[r]_{\hspace*{1cm} \cong}
& {\wt{CH^*_H\left(X, i \right)}} \ar@{^{(}->}[r]_{{\ov{u}}^G_X} &
{\ov{CH^*_H\left(X, i \right)}},} 
\end{equation}  
where the map ${{\ov{u}}^G_X}$ is injective by 
Corollary~\ref{cor:completion2}. The left vertical map is
an isomorphism by Theorem~\ref{thm:main-torus} and Theorem~\ref{thm:ED}.
The first arrow in both the rows are isomorphisms by ~\ref{eqn:fixedS2}
and ~\ref{eqn:fixedS02}.
This automatically induces the middle vertical arrow ${\wt{ch}}^G_X$,
which is an isomorphism and makes both the squares commute.
This also implies that $u^G_X$ is injective. 
This proves Theorem~\ref{thm:main*}. The proof of Theorem~\ref{thm:singular}
follows exactly the same way by replacing the Chern character by the
Riemann-Roch map, since Theorem~\ref{thm:main-torus} holds even when 
$X$ is singular.
\end{proof}
\begin{thm}\label{thm:main*-torus}
Let $G$ be diagonalizable group and let $X \in {\sV}^S_G$. Then the 
Chern character map $ch^G_X$ induces an isomorphism
\[
\wt{K^G_i(X)} \xrightarrow{{\wt{ch}}^G_X} \wt{CH^*_G\left(X, i \right)}
\]
such that the diagram
\begin{equation}\label{eqn:singular**}
\xymatrix{
 {K^G_i(X){\otimes}_{R(G)} {\widehat{R(G)}}}
\ar[r]^{{\wt{ch}}^G_X} \ar[d]_{u^G_X} &
{CH^*_G(X,i){\otimes}_{S(G)} {\widehat{S(G)}}}
\ar[d]^{{\ov{u}}^G_X} \\
{\widehat{K^G_i(X)}} \ar[r]_{{\widehat{ch}}^G_X} &
{\ov{CH^*_H\left(X, i \right)}}}
\end{equation}
commutes.
\end{thm}
\begin{proof} In the rest of this section, we shall use the terms and the
notations of Sections 5 and 6 freely without explaining them again.
Let $s \ge 0$ be the largest integer such that
$X_s \neq \emptyset$. We prove the theorem by induction on $s$.
If $s = 0$, then $G$ acts on $X$ with finite stabilizers in which case
the result is already proved in Theorem~\ref{thm:main-torus}. So we assume 
that $s \ge 1$. Let $X_s \stackrel{f_s}{\inj} X$ and $X_{< s} 
\stackrel{g_s}{\inj} X$ be the inclusions of closed and open sets as
before. Since $G$ acts on $X_s$ with constant dimension of stabilizers,
the result holds for $X_s$ by Proposition~\ref{prop:fixedS}.
This also holds for $X_{< s}$ and $N^0_s$ by induction on $s$
(since $N^0_s = {(N_s)}_{< s}$).
Thus we have the maps ${{\wt{ch}}^G_{X_{< s}}}$ such that
${\widehat{ch}}^G_{X_{< s}} \circ u^G_{X_{< s}} = {\ov{u}}^G_{X_{< s}} \circ
{{\wt{ch}}^G_{X_{< s}}}$ and similarly for $X_s$ and $N^0_s$. 
We now consider the following diagram.
\begin{equation}\label{eqn:torus**}
\xymatrix@C6pc{
{\wt{K^G_i(X)}} \ar[r]^{(g^*_s, f^*_s)} \ar@/_3pc/[dd] \ar@{.>}[d] & 
{{\wt{K^G_i(X_{< s})}} \oplus {\wt{K^G_i(X_s)}}} 
\ar[r]^{\hspace*{.5cm}{Sp^{\le s-1}_{X,s}} - {{\eta}^*_{s, \le s-1}}}
\ar[d]^{{\wt{ch}}^G_{X_{< s}} \oplus {{\wt{ch}}^G_{X_s}}} &
{\wt{K^G_i(N^0_s)}} \ar[d]^{{\wt{ch}}^G_{N^0_s}} \\
{\wt{CH^*_G\left(X, i \right)}} \ar[r]^{({\ov{g}}^*_s, {\ov{f}}^*_s)}
\ar@{^{(}->}[d]^{{\ov{u}}^G_X} & 
{{\wt{CH^*_G\left(X_{< s}, i \right)}} \oplus
{\wt{CH^*_G\left(X_s, i \right)}}} 
\ar[r]^{\hspace*{1.2cm}
{{\ov{Sp}}^{\le s-1}_{X,s}} - {{\ov{\eta}}^*_{s, \le s-1}}}
\ar@{^{(}->}[d]^{{\ov{u}}^G_{X_{< s}} \oplus {\ov{u}}^G_{X_s}} &
{\wt{CH^*_G\left(N^0_s, i \right)}} \ar@{^{(}->}[d]^{{\ov{u}}^G_{N^0_s}} \\
{\ov{CH^*_G\left(X, i \right)}} \ar[r]^{({\ov{g}}^*_s, {\ov{f}}^*_s)} &
{{\ov{CH^*_G\left(X_{< s}, i \right)}} \oplus
{\ov{CH^*_G\left(X_s, i \right)}}} 
\ar[r]^{\hspace*{1.2cm}
{{\ov{Sp}}^{\le s-1}_{X,s}} - {{\ov{\eta}}^*_{s, \le s-1}}} &
{\ov{CH^*_G\left(N^0_s, i \right)}},} 
\end{equation}
where the curved downward arrow on the left is ${\widehat{ch}}^G_X
\circ u^G_X$. 
The composite square on the left commutes by the contravariance of
the Chern character. The lower right square clearly commutes. We
show that the upper right square commutes. For this, it suffices to show 
separately that 
\begin{equation}\label{eqn:torus**0}
(i) \ {\wt{ch}}^G_{N^0_s} \circ {\eta}^*_{s, \le s-1} 
= {\ov{\eta}}^*_{s, \le s-1} \circ {\wt{ch}}^G_{X_s} \ \  {\rm and}
\ \ (ii) \ {\wt{ch}}^G_{N^0_s} \circ {Sp^{\le s-1}_{X,s}} =
{\ov{Sp}}^{\le s-1}_{X,s} \circ {\wt{ch}}^G_{X_{< s}}.
\end{equation}
But ${\eta}^*_{s, \le s-1}$ is the composite of pull-backs 
\[
\wt{K^G_i(X_s)} \xrightarrow{{\ov{f}}_{s, \infty}^*} 
\wt{K^G_i(N_s)} \to \wt{K^G_i(N^0_s)},
\]
where the last map is the pull-back to an open subset. The
same factorization holds for the higher Chow groups as well. Hence 
$(i)$ follows directly from the contravariance of the Chern character.
To prove $(ii)$, we consider the diagram
\[
\xymatrix@C2pc{
{\wt{K^G_i(M_{s, \le s-1})}} \ar@{->>}[r]^{j^*_{s, \le s-1}} 
\ar[d]_{{\wt{ch}}^G_{M_{s, \le s-1}}} &  
{{\wt{K^G_i(X_{< s})}}} \ar[r]^{Sp^{\le s-1}_{X,s}} 
\ar[d]^{{\wt{ch}}^G_{X_{< s}}} &
{\wt{K^G_i(N^0_s)}} \ar[d]^{{\wt{ch}}^G_{N^0_s}} \\
{{\wt{CH^*_G\left(M_{s, \le s-1}, i \right)}}} 
\ar@{->>}[r]^{{\ov{j}}^*_{s, \le s-1}} &
{\ov{CH^*_G\left(X_{< s}, i \right)}} \ar[r]^{{\ov{Sp}}^{\le s-1}_{X,s}}
& {\ov{CH^*_G\left(N^0_s, i \right)}}.}
\]
The left square is a pull-back diagram and hence commutes.
The composite top horizontal arrow  is just the pull-back $i^*_{s, \le s-1}$ 
and the composite bottom horizontal arrow is ${\ov{i}}^*_{s, \le s-1}$
({\sl cf.} Diagram~\ref{eqn:nospecial}). Hence the 
composite square commutes. The map ${j^*_{s, \le s-1}}$ is surjective
from the $K$-rigidity ({\sl cf.} 
\cite[Proposition~4.3, Proposition~4.6]{VV}). It follows that the
right square also commutes, which proves ~\ref{eqn:torus**0}.
Hence we have shown that all the completed squares in the diagram
~\ref{eqn:torus**} commute.

Now, the top row in the diagram ~\ref{eqn:torus**} is exact with the
first map injective by \cite[Proposition~4.7]{VV} and the fact that
$\widehat{R(T)}$ is flat over $R(T)$. The same conclusion holds for the
middle row by Proposition~\ref{prop:decomposition1}. The bottom row is exact
with the first map injective by Proposition~\ref{prop:decomposition1}
and Proposition~\ref{prop:completion1}. It follows from a diagram chase
that there exists a unique map ${\wt{ch}}^G_X : \wt{K^G_i(X)}
\to \wt{CH^*_G\left(X, i \right)}$ such that the upper left square
commutes. The injectivity of ${({\ov{g}}^*_s, {\ov{f}}^*_s)}$ now
shows that ${\widehat{ch}}^G_X \circ u^G_X = {\ov{u}}^G_X \circ
{{\wt{ch}}^G_X}$. This proves the theorem.
\end{proof}

\section{Reduction Steps For Arbitrary Groups}
In this section, we establish some reduction steps for deducing the
proofs of our main results for an arbitrary group from the case of
diagonalizable groups. We first do this for $G = GL_n$.
\begin{prop}\label{prop:main-GLn}
Let $G$ be a linear algebraic group and let $X \in {\sV}_G$. Suppose that
Theorem~\ref{thm:main*} (resp. Theorem~\ref{thm:singular}) holds when
$G$ is a diagonalizable group. Then Theorem~\ref{thm:main*} 
(resp. Theorem~\ref{thm:singular}) also holds when $G = GL_n$.
\end{prop}
\begin{proof} We first prove the assertion for Theorem~\ref{thm:main*}.
The other case is similar and we shall indicate the specific changes.
So let $T$ be a split maximal torus of $G = GL_n$ and $W$ the Weyl group.
Then by Corollary~\ref{cor:reduction}, $W$ acts on $CH^*_T\left(X , \cdot
\right)$ as graded automorphisms and $CH^j_G\left(X, i \right)
= {\left(CH^j_T\left(X, i \right)\right)}^W$ for every $i, j \ge 0$.
In particular, $W$ acts on $\ov{CH^*_T\left(X, i \right)}$
and $\ov{CH^*_G\left(X, i \right)} = {\left(\ov{CH^*_T\left(X, i \right)}
\right)}^W$. Since $W$ acts through the pull-backs, we see that it acts on 
$\widehat{S(T)}$ as $\widehat{S(G)}$-algebra automorphisms and on 
$\ov{CH^*_T\left(X, i \right)}$ as $\widehat{S(G)}$-linear maps.
We also have the trace maps ({\sl cf.} ~\ref{eqn:trace})
\[
\ov{CH^*_G\left(X, i \right)} \inj  \ov{CH^*_T\left(X, i \right)}
\xrightarrow{tr} \ov{CH^*_G\left(X, i \right)}
\]
such that the composite is identity.
Letting $W$ act on $CH^*_T\left(X, i \right) {\otimes}_{S(G)} 
{\widehat{S(T)}}$ diagonally, we get natural maps
\begin{equation}\label{eqn:trace1}
CH^*_G\left(X, i \right) {\otimes}_{S(G)} {\widehat{S(G)}} \to
CH^*_T\left(X, i \right) {\otimes}_{S(G)} 
{\widehat{S(T)}} \xrightarrow{tr} CH^*_G\left(X, i \right) {\otimes}_{S(G)} 
{\widehat{S(G)}}
\end{equation}
such that the composite is identity. We note here that since $S(T)$ and 
$S(G)$ are polynomial algebras, the $J_T$-adic and the graded filtrations
on $S(T)$ coincide, and similarly for $S(G)$. In particular, the
maps $\widehat{S(T)} \to \ov{CH^*_T\left(k, 0 \right)}$ and
$\widehat{S(G)} \to \ov{CH^*_G\left(k, 0 \right)}$ are isomorphisms. 

Since $W$ preserves the graded filtration of $CH^*_T\left(X, i \right)$,
it preserves the augmentation ideal $J_T$ of $S(T)$. Moreover, as $W$
acts as an $S(G)$-algebra on $S(T)$, we see that $W$ preserves all
powers of $J_G$ and hence the $J_G$-adic filtration of
$CH^*_T\left(X, i \right)$. Since $W$ clearly preserves the $J_G$-adic
filtration on $S(T)$ and since this filtration induces the same topology
on $S(T)$ as the $J_T$-adic filtration ({\sl cf.} \cite[Theorem~6.1]{ED1}), 
we see that the maps
\[
CH^*_T\left(X, i \right) 
\stackrel{}{\underset {S(G)}{\otimes}}{S(T)} \to    
CH^*_T\left(X, i \right) \stackrel{}{\underset {S(G)}{\otimes}}
\frac{S(T)}{J^j_G S(T)} \to
\frac{CH^*_T\left(X, i \right)}{J^j_T CH^*_T\left(X, i \right)} \to
\frac{CH^*_T\left(X, i \right)}{CH^{\ge j}_T\left(X, i \right)}
\]
are all $W$-equivariant. We conclude that the composite map
\begin{equation}\label{eqn:GTT}
\xymatrix{
CH^*_T\left(X, i \right) {\otimes}_{S(G)} 
{\widehat{S(T)}} \ar[rr]^{w^G_X} \ar@{->>}[dr] & & 
{\ov{CH^*_T\left(X, i \right)}} \\
& CH^*_T\left(X, i \right) {\otimes}_{S(T)} {\widehat{S(T)}} 
\ar@{^{(}->}[ur] & }
\end{equation}
is $W$-equivariant, where the slanted downward arrow is the natural 
surjection and the upward arrow is injective by 
Corollary~\ref{cor:completion2}. We now have the following diagram. 
\[
\xymatrix{
{CH^*_G\left(X, i \right) {\otimes}_{S(G)} 
{\widehat{S(G)}}} \ar@{^{(}->}[r] \ar@{^{(}->}[d]_{{\ov{u}}^G_X} &
{CH^*_T\left(X, i \right) {\otimes}_{S(G)} 
{\widehat{S(T)}}} \ar@{->>}[r]^{tr} \ar[d]_{w^G_X} &
{CH^*_G\left(X, i \right) {\otimes}_{S(G)} {\widehat{S(G)}}} 
\ar@{^{(}->}[d]^{{\ov{u}}^G_X} \\
{\ov{CH^*_G\left(X, i \right)}} \ar@{^{(}->}[r] & 
{\ov{CH^*_T\left(X, i \right)}}
\ar@{->>}[r]_{tr} & {\ov{CH^*_G\left(X, i \right)}}}
\]
The left square clearly commutes since $W$ acts trivially on the terms 
in the left column. Since $w^G_X$ is $W$-equivariant and ${{\ov{u}}^G_X}$
is its restriction on ${CH^*_G\left(X, i \right) {\otimes}_{S(G)} 
{\widehat{S(G)}}}$, the right square also commutes. 
Since both the composite horizontal arrows are identity, we get the
natural maps
\begin{equation}\label{eqn:trace2}
\frac{{\ov{CH^*_G\left(X, i \right)}}}
{CH^*_G\left(X, i \right) {\otimes}_{S(G)} {\widehat{S(G)}}}
\inj \frac{{\ov{CH^*_T\left(X, i \right)}}}
{CH^*_T\left(X, i \right) {\otimes}_{S(G)} {\widehat{S(T)}}}
\xrightarrow{tr} 
\frac{{\ov{CH^*_G\left(X, i \right)}}}
{CH^*_G\left(X, i \right) {\otimes}_{S(G)} {\widehat{S(G)}}}
\end{equation}
such that the composite is identity. Moreover, the middle term is
same as $\frac{{\ov{CH^*_T\left(X, i \right)}}}
{{\wt{CH^*_T\left(X, i \right)}}}$ by ~\ref{eqn:GTT}.
We now consider the following diagram.
\begin{equation}\label{eqn:main*-GLn}
\xymatrix{
{\wt{K^G_i(X)}} \ar[d]_{u^G_X} \ar@{.>}[rr] \ar[ddr] & &
{\wt{CH^*_G\left(X, i \right)}} \ar@{^{(}->}[d]^{{\ov{u}}^G_X}
\ar[ddr] & \\
{\widehat{K^G_i(X)}} \ar[rr]^{{\widehat{ch}}^G_X} \ar[ddr] & &
{\ov{CH^*_G\left(X, i \right)}} \ar[ddr] & \\
& {\wt{K^T_i(X)}} \ar@{^{(}->}[d]_{u^T_X} \ar[rr]^{{\wt{ch}}^T_X} & &
{\wt{CH^*_T\left(X, i \right)}} \ar@{^{(}->}[d]^{{\ov{u}}^T_X} \\ 
& {\widehat{K^T_i(X)}} \ar[rr]_{{\widehat{ch}}^T_X} & & 
{\ov{CH^*_T\left(X, i \right)}}}
\end{equation}
All the completed faces of the above diagram commute and the map
${\wt{ch}}^T_X$ is an isomorphism by our assumption about the torus
action. We want to show that the dotted arrow can be completed by
the map ${\wt{ch}}^G_X$ such that the back face (and hence all faces)
of the cube commutes.
Using ~\ref{eqn:trace2} and a diagram chase, we see that it suffices to 
show that the image of the composite map
\[
{\wt{K^G_i(X)}} \xrightarrow{{\widehat{ch}}^G_X \circ u^G_X}
{\ov{CH^*_G\left(X, i \right)}} \to {\ov{CH^*_T\left(X, i \right)}}
\]
lies in ${\wt{CH^*_T\left(X, i \right)}}$. But this follows directly from the
commutativity of all completed squares in the above diagram.

It remains now to show that ${\wt{ch}}^G_X$ is an isomorphism.
We first note that the completed arrows in the above diagram are
$\widehat{R(G)}$-linear via the isomorphism of the rings
$\widehat{R(G)} \xrightarrow{\cong} \widehat{S(G)}$. In particular,
${\wt{ch}}^G_X$ is $\widehat{R(G)}$-linear. Moreover, 
$\widehat{R(T)} \cong \widehat{S(T)}$ is faithfully flat over
$\widehat{R(G)} \cong \widehat{S(G)}$.
Thus it suffices to show that ${\wt{ch}}^G_X {\otimes}_{\widehat{R(G)}}
\widehat{R(T)}$ is an isomorphism. However, we have the following
commutative diagram.
\begin{equation}\label{eqn:main*-GLn1}
\xymatrix{
{{\wt{K^G_i(X)}} {\otimes}_{\widehat{R(G)}} \widehat{R(T)}} 
\ar[d]_{\cong} \ar[r]^{{\wt{ch}}^G_X \otimes {\wt{ch}}^G_k} &
{{\wt{CH^*_G\left(X, i \right)}} {\otimes}_{\widehat{S(G)}}
{\widehat{S(T)}}} \ar[d]^{\cong} \\
{K^G_i(X) {\otimes}_{R(G)} \widehat{R(T)}} \ar[d]_{\cong} &
{CH^*_G\left(X, i \right) {\otimes}_{S(G)} \widehat{S(T)}} 
\ar[d]^{\cong} \\
{\left(K^G_i(X) {\otimes}_{R(G)} R(T)\right) 
{\otimes}_{R(T)} \widehat{R(T)}} \ar[d]_{{\psi}^G_T} &
{\left(CH^*_G\left(X, i \right) {\otimes}_{S(G)} S(T)\right) 
{\otimes}_{S(T)} \widehat{S(T)}} \ar[d]^{{\phi}^G_T} \\
{K^T_i(X) {\otimes}_{R(T)} \widehat{R(T)}} 
\ar[r]_{{\wt{ch}}^T_X \otimes {\wt{ch}}^T_k} &
{CH^*_T\left(X, i \right) {\otimes}_{S(T)} \widehat{S(T)}}}
\end{equation} 
The map ${\psi}^G_T$ is an isomorphism by 
\cite[Proposition~8]{Merkurjev} and ${\phi}^G_T$ is an isomorphism by
Theorem~\ref{thm:reductiveI}. The bottom horizontal arrow is an
isomorphism by our assumption. We conclude that the top horizontal
arrow is also an isomorphism.

For the proof of Theorem~\ref{thm:singular} for $G = GL_n$, we first
observe that ~\ref{eqn:trace2} is general and it holds for any
reductive group $G$ with a split maximal torus $T$ and any $X \in {\sV}_G$ by 
Corollary~\ref{cor:reduction}. 
Hence we can consider the same diagram as ~\ref{eqn:main*-GLn} with
$K$-groups replaced by $G$-groups and the Chern character replaced
by the equivariant Riemann-Roch of Theorem~\ref{thm:RRoch}
to get the desired map ${\wt{\tau}}^G_X$. Then all the maps in the
diagram are $\widehat{R(G)}$-linear by Theorem~\ref{thm:RRoch} $(ii)$.
Since $G$ acts with finite stabilizers, so does $T$ and hence
${{\wt{\tau}}^T_X}$ is an isomorphism by Proposition~\ref{prop:fixedS}.
Moreover, as \cite[Proposition~8]{Merkurjev} also holds for the $G$-theory, 
the same argument as above shows that ${\wt{\tau}}^G_X$ is an
isomorphism.
\end{proof}
We need the following finiteness results for the maps of the equivariant Chow
rings and the completions of the representation rings
for generalizing Proposition~\ref{prop:main-GLn} to any linear 
algebraic group. For compact Lie groups and complex
algebraic groups, Lemma~\ref{lem:finitering} follows from the stronger 
results of Segal ({\sl cf.} \cite{Segal}) and Edidin-Graham
({\sl cf.} \cite[Proposition~2.3]{ED3}).
\begin{lem}\label{lem:finiteringChow} 
Let $G$ be a linear algebraic group and let $H \subset G$ be a closed
subgroup. Then the restriction map $S(G) \to S(H)$ is finite. In particular,
$S(G)$ is Noetherian.
\end{lem}
\begin{proof} We first assume that $G$ is a diagonalizable
group and $H \subset G$ is a subtorus. Then we have an
isomorphism $R(G) \cong R(H) {\otimes}_{\Q} R(G/H)$ by 
\cite[Lemma~5.6]{Thomason3} and the 
natural map $R(G) \to R(H)$ is surjective and hence finite.
In general, we embed $G$ inside some $GL_n$ to get the maps
$R(GL_n) \to R(G) \to R(H)$. Hence we can assume that $G = GL_n$.

There is a finite extension $l/k$ such that all the connected components
of the algebraic group $H_l$ are defined over $l$, the identity component
$H^0_l$ is split and its unipotent radical $R_u\left(H^0_l\right)$ is
also defined and split over $l$. Moreover, there is maximal torus $T$ of 
$G_l$ such that $T \cap H_l = T'$ is a split maximal torus of $H_l$. 
We now consider the following commutative diagram.
\begin{equation}\label{eqn:finiteS*}
\xymatrix{
S(G) \ar[r]^{\cong} \ar[d] & S(G_l) \ar[r] \ar[d] & S(T) \ar[d] \\
S(H) \ar[r] & S(H_l) \ar[r] & S(T')}
\end{equation}
We have shown above that the right vertical arrow is finite. The first
horizontal arrow on the top is an isomorphism since $G = GL_n$
({\sl cf} \cite[Section~3.2]{ED1}), and the second horizontal arrow
on the top is finite by Corollary~\ref{cor:reduction}.
In particular, $S(T')$ is finite over $S(G)$.
To see that the middle vertical arrow is finite, we can use 
Lemma~\ref{lem:nonreductive} and Corollary~\ref{cor:finite2}
to assume that $H_l$ is connected and split reductive with split
maximal torus $T'$. In that case, it follows from 
Corollary~\ref{cor:reduction} and the fact that $S(G)$ is Noetherian.
Finally, we conclude from Lemma~\ref{lem:folklore1} that
$S(H)$ is an $S(G)$-submodule of $S(H_l)$ and hence finite over
$S(G)$. The claim about the Noetherian property follows from this finiteness
result and the fact that $S(GL_n)$ is Noetherian.
\end{proof}

\begin{lem}\label{lem:finitering}
Let $G$ be a linear algebraic group and let $H \subset G$ be a closed
subgroup. Then the restriction maps $R(H) {\otimes}_{R(G)}
\widehat{R(G)} \xrightarrow{r^G_H} \widehat{R(H)}$ and
$S(H) {\otimes}_{S(G)} \widehat{S(G)}$\\
$\xrightarrow{r^G_H} \widehat{S(H)}$
are isomorphisms of $\Q$-algebras.
\end{lem}
\begin{proof} Let ${\widehat{R(H)}}_{G}$ and ${\widehat{S(H)}}_{G}$ denote 
the $I_G$-adic completion of $R(H)$ and $J_G$-adic completion
of $S(H)$ respectively. Then we have ${\widehat{R(H)}}_{G} 
\xrightarrow{\cong} \widehat{R(H)}$ and ${\widehat{S(H)}}_{G}
\xrightarrow{\cong} \widehat{S(H)}$ by \cite[Theorem~6.1]{ED1}.
The isomorphism of the map
$S(H) {\otimes}_{S(G)} \widehat{S(G)} \xrightarrow{r^G_H} \widehat{S(H)}$
now follows directly from Lemma~\ref{lem:finiteringChow}.
We consider the following commutative diagram.
\begin{equation}\label{eqn:finitering0}
\xymatrix{
R(G) \ar[r] \ar[d] & {\widehat{R(G)}} \ar@{=}[r] \ar[d] &
{\widehat{R(G)}} \ar[r]^{{\widehat{ch}}^G_k} \ar[d] & 
{\widehat{S(G)}} \ar[d] \\
R(H) \ar[r] & {R(H) {{\otimes}_{R(G)}} {\widehat{R(G)}}} \ar[r] &
{{\widehat{R(H)}}_{G}} \ar@{.>}[r] \ar[d]^{\cong} &   
{{\widehat{S(H)}}_{G}} \ar[d]^{\cong} \\
& & {\widehat{R(H)}} \ar[r]_{{\widehat{ch}}^H_k} & {\widehat{S(H)}}}
\end{equation}
Since ${\widehat{ch}}^G_k$ and ${\widehat{ch}}^H_k$ are ring isomorphisms by 
Theorem~\ref{thm:ED}, we conclude that there exists a unique
isomorphism ${{\widehat{R(H)}}_{G}} \xrightarrow{{\widehat{ch}}^H_k}
{{\widehat{S(H)}}_{G}}$ such that the above diagram commutes.
We have shown that ${{\widehat{S(H)}}_{G}} \cong
S(H) {\otimes}_{S(G)} {\widehat{S(G)}}$, and the latter is then
finite over ${\widehat{S(G)}}$. Now we conclude from the above 
diagram that ${{\widehat{R(H)}}_{G}}$ is finite over ${\widehat{R(G)}}$.

Using the indicated isomorphisms of rings in the above diagram, we need to 
show that the composite map ${R(H) {{\otimes}_{R(G)}}} 
{\widehat{R(G)}} \xrightarrow{{\phi}^G_H} {{\widehat{S(H)}}_{G}}$ 
in the middle row is an isomorphism, in order to prove the
proposition. We have now the natural maps
\[
{R(H) {{\otimes}_{R(G)}} {\widehat{R(G)}}} \to 
{K^G_0(G/H)  {{\otimes}_{R(G)}} {\widehat{R(G)}}}
\xrightarrow{{\wt{ch}^G_{G/H}}}
{CH^*_G\left(G/H, 0 \right) {{\otimes}_{S(G)}} {\widehat{S(G)}}}.
\]
The first map is an isomorphism by the Morita equivalence 
\cite[Proposition~3.2]{ED1} and the second map is an isomorphism by
Theorem~\ref{thm:main*-torus} and Proposition~\ref{prop:main-GLn}.
However, the last term is same as $CH^*_H\left(k, 0\right)
{{\otimes}_{S(G)}} {\widehat{S(G)}} \cong
S(H) {\otimes}_{S(G)} {\widehat{S(G)}}$
by \cite[Proposition~3.2]{ED1}, and
we have seen above that the latter term is same as
${{\widehat{S(H)}}}_G$. This proves the lemma.
\end{proof}
\begin{prop}\label{prop:main-general}
Let $G$ be a linear algebraic group and let $X \in {\sV}_G$. Suppose that
Theorem~\ref{thm:main*} (resp. Theorem~\ref{thm:singular}) holds when
$G$ is a diagonalizable group. Then Theorem~\ref{thm:main*} 
(resp. Theorem~\ref{thm:singular}) holds for any $G$.
\end{prop}
\begin{proof} As in the proof of Proposition~\ref{prop:main-GLn},
we first prove the assertion for Theorem~\ref{thm:main*}.
The other case is similar and we shall indicate the specific changes.
Choosing an embedding $G \inj GL_n$ as a closed subgroup and using
\cite[Theorem~1.10]{Thomason1} and Corollary~\ref{cor:Morita} above,
we have natural isomorphisms
\[
G^{GL_n}_i\left(GL_n \stackrel{G}{\times} X\right) \xrightarrow{\cong}
G^G_i(X) \ \ {\rm and}
\]
\[
CH^*_{GL_n}\left(GL_n \stackrel{G}{\times} X, i \right) \xrightarrow{\cong}
CH^*_G\left(X, i \right).
\]
Putting $\wt{X} = GL_n \stackrel{G}{\times} X$ and replacing the $G$-groups
by $K$-groups (as $X$ is smooth), we get the following commutative
diagram.
\[
\xymatrix{
{K^G_i(X) {\otimes}_{R(G)} \left(R(G) {\otimes}_{R(GL_n)} 
\widehat{R(GL_n)}\right)} \ar[r]^{\cong} \ar[dddd]_{r^{GL_n}_G} &
{K^{GL_n}_i(\wt{X}) {\otimes}_{R(GL_n)} \widehat{R(GL_n)}} 
\ar[d]^{{\wt{ch}}^{GL_n}_{\wt{X}}} \\
& {CH^*_{GL_n}\left(\wt{X}, i \right) {\otimes}_{S(GL_n)} 
\widehat{S(GL_n)}} \ar[d] \\
& {CH^*_{G}\left(X, i \right) {\otimes}_{S(GL_n)} 
\widehat{S(GL_n)}} \ar[d]^{\cong} \\
& {CH^*_{G}\left(X, i \right) {\otimes}_{S(G)} \left(
S(G) {\otimes}_{S(GL_n)} \widehat{S(GL_n)}\right)} \ar[d]^{r^{GL_n}_G}  
\\
{K^G_i(X) {\otimes}_{R(G)} \widehat{R(G)}} \ar[d]_{u^G_X} 
\ar@{.>}[r]^{{\wt{ch}}^G_X}
& {CH^*_{G}\left(X, i \right) {\otimes}_{S(G)} \widehat{S(G)}} 
\ar@{^{(}->}[d]^{{\ov{u}}^G_X} \\
{\widehat{K^G_i(X)}} \ar[r]_{{\widehat{ch}^G_X}} & 
{\ov{CH^*_{G}\left(X, i \right)}}} 
\]
The top horizontal map is the composite
\[
{K^G_i(X) {\otimes}_{R(G)} \left(R(G) {\otimes}_{R(GL_n)} 
\widehat{R(GL_n)}\right)} \xrightarrow{\cong} 
{K^G_i(X) {\otimes}_{R(GL_n)} \widehat{R(GL_n)}}
\]
\[
\hspace*{9cm} \xrightarrow{\cong}
{K^{GL_n}_i(\wt{X}) {\otimes}_{R(GL_n)} \widehat{R(GL_n)}}.
\]  
The second vertical arrow on the right is an isomorphism as 
mentioned above. The maps $r^{GL_n}_G$ on both sides are isomorphisms
by Lemma~\ref{lem:finitering}, ${{\wt{ch}}^{GL_n}_{\wt{X}}}$ is
an isomorphism by Proposition~\ref{prop:main-GLn}
and ${{\ov{u}}^G_X}$ is injective by
Corollary~\ref{cor:completion2}. We conclude that the map 
${{\widehat{ch}^G_X}} \circ u^G_X$ factors through a map
${{\wt{ch}}^G_X}$ which is an isomorphism.

To deduce Theorem~\ref{thm:singular} from the case of diagonalizable 
groups, we use the same diagram as above for $G$-groups and the same
changes as in the proof of the theorem for $GL_n$ in 
Proposition~\ref{prop:main-GLn}, and observe that if $G$ acts on $X$ with
finite stabilizers, then so does $GL_n$ on $\wt{X}$.
\end{proof} 

\section{Proof Of The main result and consequences}
Our strategy to prove Theorem~\ref{thm:main*} and Theorem~\ref{thm:singular} 
is to use the results of the previous section to reduce the proofs to the 
case of diagonalizable groups.
\\
{\bf{Proof of Theorem~\ref{thm:main*}:}} 
The existence of
${\wt{ch}}^G_X$ and its isomorphism follows from Theorem~\ref{thm:main*-torus}
and Proposition~\ref{prop:main-general}. 
We have moreover seen from the construction of the equivariant Chern
character in Proposition~\ref{prop:Chern} that 
$\widehat{R(G)} \xrightarrow{{\widehat{ch}}^G_k} \widehat{S(G)}$ is an
isomorphism of rings and ${\wt{ch}}^G_X$ is $\widehat{R(G)}$-linear
under this isomorphism. Taking the $\widehat{I_G}$-completion of the top row 
in the diagram ~\ref{eqn:main**} under the isomorphism 
$\widehat{I_G} \cong \widehat{J_G}$, we see that the isomorphism
${\wt{ch}}^G_k$ induces the isomorphism of completions
$\widehat{K^G_i(X)} \xrightarrow{{\widehat{ch}}^G_k}
\widehat{CH^*_G\left(X, i \right)}$.
  
Finally, the map ${\ov{u}}^G_X$ is injective by 
Corollary~\ref{cor:completion2}. Hence $u^G_X$ is also injective.
$\hfill \square$
\\
\\ 
{\bf{Proof of Corollary~\ref{cor:mainRR}:}}
The ring homomorphism $\widehat{K^G_0(X)} \xrightarrow{{\widehat{ch}}^G_X}
\widehat{CH^*_G\left(X, 0\right)}$ is an isomorphism by 
Theorem~\ref{thm:ED}. Moreover, the map ${\widehat{\tau}}^G_X$ is
a $\widehat{K^G_0(X)}$-linear under the above isomorphism by
Theorem~\ref{thm:RRoch}.
Since the Todd classes of the vector bundles are invertible
elements of $\widehat{CH^*_G\left(X, 0\right)}$, we see from the
construction of the Riemann-Roch map in Theorem~\ref{thm:RRoch}
that there is an invertible element ${\beta}_X$ in
$\widehat{K^G_0(X)}$ such that ${\widehat{\tau}}^G_X = {\beta} \cdot
{\widehat{ch}}^G_X$. The corollary now follows from Theorem~\ref{thm:main*}.   
$\hfill \square$
\\

{\bf{Proof of Theorem~\ref{thm:singular}:}} If a diagonalizable group acts
with a fixed dimension of stabilizers, then this is proved in
Proposition~\ref{prop:fixedS}. If any linear algebraic group $G$
acts on $X$ with finite stabilizers, then ${\wt{\tau}}^G_X$ is an
isomorphism by the case of diagonalizable groups and 
Proposition~\ref{prop:main-general}. Since ${\wt{\tau}}^G_X$ is an
$\widehat{R(G)}$-linear ({\sl cf.} Theorem~\ref{thm:RRoch}), the 
isomorphism of $\widehat{G^G_i(X)} \xrightarrow{{\widehat{\tau}}^G_k}
\widehat{CH^*_G\left(X, i \right)}$ follows exactly in the same way
as the isomorphism of ${\widehat{ch}}^G_X$ in the proof of
Theorem~\ref{thm:main*}.

To prove the isomorphism of vertical arrows in the diagram
~\ref{eqn:singular*}, we first see that the map ${\ov{u}}^G_X$ is an
isomorphism by Proposition~\ref{prop:FST}. This implies then that
$u^G_X$ is injective. 
On the other hand, the natural maps
${G^G_i(X)}_{I_G} \to \wt{G^G_i(X)} \xrightarrow{u^G_X} 
{\widehat{G^G_i(X)}}$ 
and Lemma~\ref{lem:FST1} imply that $u^G_X$ is surjective.
$\hfill \square$
\\ 
\\ 
{\bf{Proof of Theorem~\ref{thm:finite}:}}
For any $X \in {\sV}_G$, we have the natural maps
\[
G^G_i(X) \to {G^G_i(X)}_{I_G} \xrightarrow{{\wt{\tau}}^G_X}
\wt{CH^*_G\left(X, i \right)} \leftarrow CH^*_G\left(X, i \right), 
\]
where the last map is an isomorphism by Proposition~\ref{prop:FST} and
${{\wt{\tau}}^G_X}$ is an isomorphism by Theorem~\ref{thm:singular}.
This gives the desired lifting of the Riemann-Roch map
$G^G_i(X) \xrightarrow{{\tau}^G_X} CH^*_G\left(X, i \right)$ such that
$\rm{Ker}\left({\tau}^G_X\right)$ has the desired property.

For the surjectivity of ${\tau}^G_X$, we first prove the case when $X$ is 
smooth. By \cite[Theorem~6.1]{Se}, there is a finite $G$-equivariant
cover $f: X' \to X$ such that $X'$ is normal and $G$ acts freely on $X'$. 
In particular, $X'$ is quasi-projective and the $G$-action on $X'$ is linear
({\sl cf.} \cite[Proposition~7.1]{GIT}). We claim that for any $i,j \ge 0$,
the map ${CH_G^j(X',i)} \xrightarrow{f_*} {CH_G^j(X,i)}$ is 
surjective. 
To see this, choose a representation $V$ of $G$ and a $G$-invariant open
subset $U$ of $V$ such that $G$ acts freely on $U$ and the codimension of 
$V-U$ is larger than $j$. Then
the induced map ${\tilde f} : X'_G \to X_G$ is a finite map
of quasi-projective schemes ({\sl loc. cit.})
such that $X_G$ is smooth. Hence by \cite[Theorem~4.1 and 5.8]{Bloch},
there is a pull-back map ${\tilde f}^* : {CH^j(X_G,i)} \to 
{CH^j(X'_G,i)}$ such that ${\tilde f}_* \circ {\tilde f}_*$ is the
multiplication by the degree of ${\tilde f}$, which proves the claim.

Now let $Y' = {X'}/G$ and consider the following Riemann-Roch diagram.
\[
\xymatrix{
{G^G_i(X')} \ar[r]^{{\tau}^G_{X'}} \ar[d]_{f_*} &
{CH^*_G(X',i)} \ar[d]^{{\bar f}_*} \\
{G^G_i(X)} \ar[r]_{{\tau}^G_X} & {CH^*_G(X,i)}}
\]
Note that like $X'_G$ and $X_G$, $Y'$ is also quasi-projective.
Moreover, one has $G^G_i(X') \cong G^i(Y')$ and $CH^*_G(X',i) \cong
CH^*(Y',i)$. In particular, we conclude from the Bloch's non-equivariant  
Riemann-Roch Theorem that the top horizontal arrow is an isomorphism.
We have just shown above that the right vertical map is surjective.
We conclude that the bottom horizontal map must be surjective.

If $X$ is singular and the stack $[X/G]$ has a coarse moduli scheme, then
by \cite[Theorem~1]{KV}, there a $X' \in {\sV}_G$ and finite flat
$G$-equivariant map $f : X' \to X$ such that $G$ acts freely on $X'$.
As $f$ is finite and flat, we have maps ${\tilde f}^* : {CH^j(X_G,i)} \to 
{CH^j(X'_G,i)}$ and ${\tilde f}_* : {CH^j(X'_G,i)} \to 
{CH^j(X_G,i)}$ such that ${\tilde f}_* \circ {\tilde f}_*$ is the
multiplication by the degree of ${\tilde f}$ by the same reason.
The same argument as in the smooth case above now proves that ${\tau}^G_X$
is surjective.
$\hfill \square$
\\  

As another application of our Riemann-Roch isomorphisms to the equivariant
$K$-theory, we can prove the following partial generalization of 
\cite[Proposition~8]{Merkurjev} from simply connected to arbitrary reductive 
groups. For simply connected reductive groups, Merkurjev's result is stronger 
in that his result holds with the integral coefficients. 
But the main advantage of the following result is that it does not assume
anything about $G$. It will turn out in the forthcoming sequel that it
may not be necessary to complete the ring $R(T)$ in the corollary below. 
\begin{cor}\label{cor:Merkurjev1}
Let $G$ be a connected and split reductive group and let $T$ be a split 
maximal torus of $G$. Then for any smooth variety $X$ with $G$-action and for 
any $i \ge 0$, the natural map
\[
K^G_i(X) {\otimes}_{R(G)} {\widehat{R(T)}} \xrightarrow{{\eta}^G_T}
K^T_i(X) {\otimes}_{R(T)} {\widehat{R(T)}}
\]
is an isomorphism. For actions with finite stabilizers, 
${\eta}^G_T$ is an isomorphism even if $X$ is not smooth.
\end{cor}
\begin{proof}
We first note that there is an $\widehat{R(G)}$-linear isomorphism
\[
K^G_i(X) {\otimes}_{R(G)} \widehat{R(T)} 
\xrightarrow{\cong} {\wt{K^G_i(X)}} {\otimes}_{\widehat{R(G)}} 
\widehat{R(T)}.
\]
Thus we need to show that the natural map
${\wt{K^G_i(X)}} {\otimes}_{\widehat{R(G)}} \widehat{R(T)}
\to {\wt{K^T_i(X)}}$ is an isomorphism.
For this, we consider the following commutative diagram.
\[
\xymatrix{
{{\wt{K^G_i(X)}} {\otimes}_{\widehat{R(G)}} \widehat{R(T)}}
\ar[r] \ar[d] & {\wt{K^T_i(X)}} \ar[d] \\
{{\wt{CH^*_G\left(X, i \right)}} 
{\otimes}_{\widehat{S(G)}} \widehat{S(T)}} \ar[r] &
{{\wt{CH^*_T\left(X, i \right)}}}}
\]
The vertical maps are isomorphisms by Theorem~\ref{thm:main*}.
So we need to show that the bottom horizontal map is an isomorphism.
However, we have isomorphism
\[
{{\wt{CH^*_G\left(X, i \right)}} 
{\otimes}_{\widehat{S(G)}} \widehat{S(T)}}
\cong CH^*_G\left(X, i \right) {\otimes}_{S(G)} \left({\widehat{S(G)}}
{\otimes}_{\widehat{S(G)}} {\widehat{S(T)}}\right)\]
\[
\hspace*{5cm} \cong \left(CH^*_G\left(X, i \right) 
{\otimes}_{S(G)} S(T) \right) {\otimes}_{S(T)} {\widehat{S(T)}}.
\]
But the last term is naturally isomorphic to
$CH^*_T\left(X, i \right)  {\otimes}_{S(T)} {\widehat{S(T)}}
= \wt{CH^*_T\left(X, i \right)}$ by Theorem~\ref{thm:reductiveI}.
This proves the corollary when $X$ is smooth. The same proof also
applies when $X$ is not smooth and $G$ acts with finite stabilizers
by Theorem~\ref{thm:singular} since our extra ingredient
Theorem~\ref{thm:reductiveI} holds for all $X \in {\sV}_G$.
\end{proof}
\begin{ack}
The author is deeply indebted to Angelo Vistoli for many
fruitful discussions which were very valuable in this work.
The suggestion of using the results of \cite{VV} in this work
was essentially by him. The author also thanks him for suggesting
several improvements in the earlier version of this paper. 
Part of this work was done when the author was visiting
Scuola Normale Superiore during the summer of 2008. He
would like to thank the institute for its invitation and 
financial support during his stay.
\end{ack} 

School of Mathematics,
Tata Institute Of Fundamental Research, \\
Homi Bhabha Road,
Mumbai, 400005, India. \\ 
{\sl E-mail address :} amal@math.tifr.res.in  
\end{document}